\RequirePackage{etex}
\documentclass[final,3p]{elsarticle}
\usepackage{amssymb}
\usepackage{amsmath}
\usepackage{amsthm}
\usepackage{mathtools}
\usepackage{bm}
\usepackage[shortlabels]{enumitem}
\usepackage{hyperref}
\usepackage[capitalise]{cleveref}
\usepackage{xcolor}
\usepackage{algorithmic}
\usepackage{algorithm}
\newtheorem{theorem}{Theorem}
\newtheorem{corollary}[theorem]{Corollary}
\newtheorem{proposition}[theorem]{Proposition}
\newtheorem{lemma}[theorem]{Lemma}

\theoremstyle{definition}
\newtheorem{definition}[theorem]{Definition}

\theoremstyle{remark}
\newtheorem{remark}[theorem]{Remark}

\crefname{prty}{property}{properties}
\Crefname{prty}{Property}{Properties}
\crefname{assump}{assumption}{assumptions}

\newcommand{\bbR}{\mathbb{R}}
\newcommand{\bbC}{\mathbb{C}}
\newcommand{\bbN}{\mathbb{N}}

\DeclarePairedDelimiter\abs{\lvert}{\rvert}
\DeclarePairedDelimiter\norm{\lVert}{\rVert}
\def\sto{\xrightarrow{\text{s}}} % strong "\to"
\def\srto{\xrightarrow{\text{sr}}} % strong resolvent "\to"
\def\sdto{\xrightarrow{\text{sd}}} % strong dynamical "\to"
\DeclareMathOperator{\ran}{ran}
\DeclareMathOperator{\Id}{Id}
\DeclareMathOperator{\Real}{Re}
\DeclareMathOperator{\Imag}{Im}
\DeclareMathOperator{\spn}{span}
\DeclareMathOperator{\diag}{diag}
\DeclareMathOperator{\divr}{div}

\journal{}

\begin{document}

\begin{frontmatter}



\title{Physics-informed spectral approximation of Koopman operators}


\author[a]{Dimitrios Giannakis} %% Author name
\author[b]{Claire Valva} %% Author name

\affiliation[a]{organization={Department of Mathematics, Dartmouth College},
                addressline={29 N.\ Main St.},
                city={Hanover},
                postcode={03755},
                state={New Hampshire},
                country={USA}}
\affiliation[b]{organization={Courant Institute of Mathematical Sciences, New York University},
                addressline={251 Mercer St.},
                city={New York},
                state={New York},
                postcode={10012},
                country={USA}}

%% Abstract
\begin{abstract}
    Koopman operators and transfer operators represent nonlinear dynamics in state space through its induced action on linear spaces of observables and measures, respectively. This framework enables the use of linear operator theory for supervised and unsupervised learning of nonlinear dynamical systems, and has received considerable interest in recent years. Here, we propose a data-driven technique for spectral approximation of Koopman operators of continuous-time, measure-preserving ergodic systems that is asymptotically consistent and makes direct use of known equations of motion (physics). Our approach is based on a bounded transformation of the Koopman generator (an operator implementing directional derivatives of observables along the dynamical flow), followed by smoothing by a Markov semigroup of kernel integral operators. This results in a skew-adjoint, compact operator whose eigendecomposition is expressible as a variational generalized eigenvalue problem. We develop Galerkin methods to solve this eigenvalue problem and study their asymptotic consistency in the large-data limit. A key aspect of these methods is that they are physics-informed, in the sense of making direct use of dynamical vector field information through automatic differentiation of kernel functions. Solutions of the eigenvalue problem reconstruct evolution operators that preserve unitarity of the underlying Koopman group while spectrally converging to it in a suitable limit. In addition, the computed eigenfunctions have representatives in a reproducing kernel Hilbert space, enabling out-of-sample evaluation of learned dynamical features. Numerical experiments performed with this method on integrable and chaotic low-dimensional systems demonstrate its efficacy in extracting dynamically coherent observables under complex dynamics.
\end{abstract}


\begin{highlights}
\item New, spectrally accurate method to approximate Koopman operators from data.
\item Makes direct use of equations of motion (physics).
\item Extracts dynamically coherent observables under complex dynamics.
\item Approximation can be computed from samples drawn from invariant measure, without need of time-ordered snapshots.
\item Allows out-of-sample evaluation of eigenfunctions.
\end{highlights}

%% Keywords

\end{frontmatter}



\section{Introduction}

Dynamical systems can be characterized via an operator-theoretic formulation of ergodic theory through their induced action on linear spaces of observables, realized through Koopman (composition) operators and their duals, the Ruelle-Perron-Frobenius (transfer) operator \cite{Baladi00,EisnerEtAl15}. This framework allows one to translate problems of nonlinear dynamics, to equivalent linear, if infinite dimensional, problems about linear evolution operators which act on observables by composition with the flow map. As such, many relevant problems in the sciences, such as coherent feature extraction or statistical prediction can be formulated as linear problems. Starting from work in the late 1990s and early 2000s \cite{Froyland99,DellnitzJunge99,Mezic05}, the operator-theoretic approach has seen a surge of interest as a framework for data-driven techniques for analysis and modeling of dynamical systems \cite{KlusEtAl18,BruntonEtAl22,OttoRowley21,Colbrook24}. These methods have been successfully applied in diverse scientific domains, including fluid dynamics, climate dynamics, molecular dynamics, energy systems science, and control; see, for example, \cite{Mezic13,FroylandEtAl21,WuEtAl17,SusukiEtAl16,MauroyEtAl20}.

Despite the enticing theoretical properties of Koopman/transfer operators and recent progress in computational methodologies, the design of associated data-driven spectral approximation techniques is challenging. For example, a measure-preserving flow is weak-mixing if and only if the Koopman operator on $L^2$ has a simple eigenvalue at 1, with a constant corresponding eigenfunction, and no other eigenvalues \cite{Walters81}. Many commonly studied systems are rigorously known to be mixing, such as the Lorenz 63 (L63) system \cite{Tucker99, LuzzattoEtAl05}. Further, even if the point spectrum of a dynamical system is nontrivial, it is not guaranteed to be discrete (e.g., the spectrum could possibly be dense in the unit circle). As such, to study systems with high dynamical complexity, numerical methods must be able to stably and consistently approximate the often non-discrete Koopman spectrum.

In scientific applications, it is additionally desirable to be able to use data-driven algorithms in combination with known physics to study a given system. Some popular and successful methods of this type include physics-informed neural networks \cite{RaissiEtAl19} or physics-informed dynamic mode decomposition (piDMD) \cite{BaddooEtAl23} that make use of structural information of a given dynamical system, e.g., enforcing that a fluid be incompressible or that the DMD modes are measure-preserving. Here, we will develop a method in this vein, where we make use of the equations of motion (i.e., known \emph{physics}) from a given dynamical system time to approximate the spectrum of the Koopman generator.

Specifically, we propose a data-driven methodology for the approximation of Koopman and transfer operators in continuous-time measure-preserving ergodic flows that directly uses vector field information of the dynamical system, via the use of automatic differentiation of kernel functions. Our primary focus is the generator $V$ of the unitary Koopman group $U^t = e^{tV}$ of such systems. The generator acts as a generalized derivative (detailed more in \cref{sec:dynamical_system}), i.e., $V f = \lim_{t\to 0} \frac{U^t f - f}{t}$, and as such, if we have vector field data, we can directly compute the action of $V$ upon a given function of sufficient regularity. We combine this observation with a bounded transformation of the Koopman generator and smoothing by kernel integral operators to build a family of skew-adjoint, unbounded operators $V_{z,\tau}$ with compact resolvents, parameterized by $z,\tau>0$. The operators $V_{z,\tau}$ have purely discrete spectra by compactness of their resolvent and will be shown to converge spectrally to $V$ in the iterated limit of $z\to 0^+$ after $\tau \to 0^+$ (see \cref{thm:vztau_conv}). Here, $\tau$ is a smoothing parameter and $z$ will correspond to the smallest frequency scale the algorithm can faithfully resolve. We formulate the computation of eigenvalues and eigenvectors of $V_{z,\tau}$ as a variational generalized eigenvalue problem that admits well-posed Galerkin approximations in a learned basis of kernel eigenfunctions.

The eigendecomposition of $V_{z,\tau}$ then allows for the identification coherent features in a dynamical system through the eigenpairs ($\omega_j, \xi_j$), $V_{z,\tau} \xi_j = e^{i \omega_j t} \xi_j$, where $\xi_j \in L^2$ is the eigenfunction representing a coherent feature that evolves with characteristic frequency $\omega_j \in \mathbb R$. Additionally, each eigenfunction obtained via our scheme has a representative in a reproducing kernel Hilbert space (RKHS) of continuously differentiable functions. This enables out-of-sample evaluation to reconstruct eigenfunction time series at high temporal resolution from coarsely time-sampled (even independent) training data, among other supervised and unsupervised learning tasks.

In previous work, the authors have developed related techniques for spectral approximation of Koopman operators of measure-preserving ergodic flows which are based on compact approximations of the generator \cite{DasEtAl21} or its resolvent \cite{GiannakisValva24}. Hereafter, we collectively reference these papers as DGV. The new contributions of this work, realized through explicit use of of equations of motion, include:
\begin{enumerate}
    \item Approximations of the Koopman generator and operator can be obtained from independent samples drawn from the invariant measure, rather than the time-ordered snapshots commonly used in DMD-type techniques.
    \item Various discretization errors such as finite-difference approximation of the generator \cite{DasEtAl21} or numerical quadrature for approximation of the resolvent \cite{GiannakisValva24} are avoided.
    \item The use of automatic differentiation makes the scheme amenable to implementation using basis functions derived from eigendecomposition of general classes of kernel integral operators, including Markov smoothing operators adapted to the invariant measure.
\end{enumerate}
We demonstrate our approach with numerical applications to low-dimensional dynamical systems exhibiting different types of spectral characteristics of the Koopman operator: an ergodic torus rotation, a Stepanoff flow on the 2-torus, and the L63 system on $\mathbb R^3$. The examples illustrate the ability of our approach to stably compute eigenvalue and eigenfunctions of Koopman operators for systems that have them, and to identify slowly decorrelating observables that behave as approximate Koopman eigenfunctions under mixing dynamics.

The plan of the paper is as follows. In \cref{sec:prelims}, we describe the dynamical system under study, along with our assumptions and function spaces employed in this work. In \cref{sec:scheme,sec:numimpl}, we describe our spectral approximation scheme and its numerical implementation, respectively. \Cref{sec:examples} presents our numerical experiments. \Cref{sec:conclusions} contains our primary conclusions and outlook on future work. Auxiliary technical results are collected in \ref{app:aux_results}. \ref{app:vbkernel} summarizes the construction of variable-bandwidth kernels used in our numerical experiments.

\section{Preliminaries and assumptions}
\label{sec:prelims}

\subsection{Dynamical system}
\label{sec:dynamical_system}

Our setup and notation follows closely DGV. In particular, we consider a continuous-time, continuous flow $\Phi^t \colon \mathcal M \to \mathcal M$, $t \in \mathbb R$, on a metric space $\mathcal M$, possessing an ergodic invariant Borel probability measure $\mu$ with compact support $X\subseteq \mathcal M$. We assume that there is a compact $C^1$ manifold $M \subseteq \mathcal M$ which is forward-invariant (i.e., $\Phi^t(M) \subseteq M$ for any $t \geq 0$) and contains $X$. These assumptions are met by many dynamical systems encountered in applications, including classes of measure-preserving flows on compact manifolds ($X = M = \mathcal M$), dissipative flows on manifolds with attractors ($ X \subseteq M \subseteq \mathcal M$, with $X$ being an attractor and $M$ an absorbing set), and infinite-dimensional systems governed by dissipative partial differential equations with inertial manifolds.

The flow $\Phi^t$ induces a strongly-continuous group of unitary Koopman operators $U^t \colon H \to H$ that act on observables in the Hilbert space $H=L^2(\mu)$ by composition, $U^t f = f \circ \Phi^t$ \cite{Koopman31,KoopmanVonNeumann32}. Moreover, the adjoint of the Koopman operator $U^{t*} \equiv U^{-t}$ is the transfer operator that governs the evolution of densities of measures in $H$ under the dynamics. By Stone's theorem on strongly continuous, one-parameter unitary groups \cite{Stone32}, the Koopman group $\{ U^t \}_{t \in \bbR}$ has a skew-adjoint generator $V\colon D(V) \to H$, defined on a dense subspace $D(V) \subseteq H$ by means of the norm limit
\begin{equation*}
V f = \lim_{t\to 0} \frac{U^t f - f}{t}.
\end{equation*}
The generator $V$ completely characterizes the Koopman group, in the sense that if two strongly continuous, one-parameter unitary groups have the same generator $V$ they are identical \cite{Schmudgen12}. Moreover, $V$ reconstructs the Koopman operator at any time $t \in \mathbb R$ by exponentiation, $U^t = e^{tV}$, defined via the Borel functional calculus.

When acting on elements with continuously differentiable representatives in $C^1(M)$, $V$ reduces to a directional derivative along the vector field $\vec V \colon M \to TM$ that generates $\Phi^t$. Specifically, we have
\begin{equation}
    \label{eq:v_grad}
    V \iota f = \iota \vec V \cdot \nabla f, \quad f \in C^1(M),
\end{equation}
where $\iota \colon C(M) \to H$ is the map that sends continuous functions on $M$ to their corresponding equivalence classes in $H$, and the vector field $\vec V$ satisfies
\begin{equation}
    \label{eq:v_phi}
    \frac{d\ }{dt} \Phi^t(x) = \vec V(\Phi^t(x)), \quad x \in M.
\end{equation}
 In other words, $\vec V$ represents the equations of motion (``physics'') governing the flow.
 In this work, we assume that the vector field $\vec V$ is known to us in the form of governing equations for the initial-value problem~\eqref{eq:v_phi}.

In the following, we write $\langle f, g\rangle = \int_X f^* g \, d\mu$ and $\norm{f}_H = \langle f, f\rangle$ for the inner product and norm of $H$, respectively, where the inner product is taken conjugate-linear in the first argument. Moreover, $\bm 1\colon \mathcal M \to \mathbb R$ is the function on $\mathcal M$ equal to 1 everywhere. We will denote the space of bounded linear maps between Banach spaces $\mathbb E$ and $\mathbb F$ as $B(\mathbb E, \mathbb F)$, and $\lVert A\rVert$ will be the operator norm of $A \in B(\mathbb E, \mathbb F)$. We will use the abbreviation $B(\mathbb E) \equiv B(\mathbb E, \mathbb E)$. Given an operator $A\colon D(A) \to \mathbb E$ defined on a subspace $D(A) \subseteq \mathbb E$, $\rho(A)$ and $\sigma(A)$ will denote the resolvent set and spectrum of $A$, respectively. For a complex number $z \in \rho(A)$, $R(z, A) \in B(\mathbb E)$ will be the corresponding resolvent operator, $R(z, A) = (z - V)^{-1}$. Note that for $z \in \rho(A)$, $z - V$ is a surjective operator.

\subsection{Spectral decomposition}
\label{sec:spec_decomp}

Our objective is to approximate the generator $V$ from \cref{sec:dynamical_system} by a family of operators whose spectral properties are amenable to numerical approximation, while also providing dynamically relevant feature extraction. Importantly, we seek that our approximations are physics-informed, in the sense of making use of the known dynamical vector field $\vec V$ via~\eqref{eq:v_grad}. In this subsection, we give a brief outline of results from spectral theory that underpin our constructions. For detailed expositions of these topics see, e.g., \cite{Baladi00,EisnerEtAl15,Oliveira09,Schmudgen12}.

By the spectral theorem for skew-adjoint operators, the spectrum $\sigma(V)$ of the generator $V$ is a closed subset of the imaginary line, and there exists a projection-valued measure (PVM) $E\colon \mathcal{B}(i\bbR) \to B(H)$ on the Borel $\sigma$-algebra $\mathcal B(i \mathbb R)$ of the imaginary line that decomposes the generator and Koopman operators via the spectral integrals
\begin{equation}
    \label{eq:spec_decomp}
    V = \int_{i\bbR} i\omega \, dE(i\omega), \quad U^t = \int_{i\bbR} e^{i \omega t} \, dE(i\omega).
\end{equation}
Moreover, there is a $U^t$-invariant splitting $H = H_p \oplus H_c$ into orthogonal subspaces $H_p$ and $H_c$ such that the spectral measure $E_p \colon \mathcal B(i \mathbb R) \to B(H_p)$ of $V \rvert_{H_p} $ is discrete and the spectral measure  $E_c \colon \mathcal B(i \mathbb R) \to B(H_c)$ of $V \rvert_{H_c}$ is purely continuous (i.e., has no atoms), yielding the decomposition $E = E_p \oplus E_c$.

The atoms of $E_p$ are singleton sets $ \{ i\omega_j \}$ containing the eigenvalues of $V$,
\begin{displaymath}
    V \xi_j = i \omega_j \xi_j, \quad \omega_j \in \mathbb R, \quad \xi_j \in H_p \setminus \{ 0 \},
\end{displaymath}
and the corresponding eigenfunctions $\xi_j$ form an orthonormal basis of $H_p$. By~\eqref{eq:spec_decomp}, we have $U^t \xi_j = e^{i\omega_j t} \xi_j$, so $\xi_j$ has periodic evolution with period $2 \pi / \omega_j$. We therefore interpret the real numbers $\omega_j$ as eigenfrequencies of the dynamical system. It can be shown that, by ergodicity, the eigenvalues of $V$ are all simple, and, by definition of $V$, come in complex-conjugate pairs. We can therefore choose a $\mathbb Z$-based indexing such that $\omega_0 = 1$, $\xi_0 = \bm 1$, and $\omega_{-j} = - \omega_j$, $\xi_{-j} = - \xi_j^* $ for $j \neq 0$.  It can further be shown that the eigenfunctions $\xi_j$ lie in $L^\infty(\mu)$ with $\lvert \xi_j\rvert = 1 $ $\mu$-a.e.\ when $\xi_j$ is normalized to unit $L^2(\mu)$ norm, and they form a multiplicative group, $\xi_j \xi_k = \xi_l$, with the eigenfrequencies forming an additive group, $\omega_j + \omega_k = \omega_l$.

Based on the above, a collection $ \xi_{j_1}, \ldots, \xi_{j_d} $ of eigenfunctions induces a feature map $F \colon \mathcal M \to \mathbb C^d$, where $F(x) = (\xi_{j_1}(x), \ldots, \xi_{j_d}(x))$ $\mu$-a.e.\ and the range of $F$ is a subset of the $d$-dimensional torus $\mathbb T^d \subset \mathbb C^d$. A key property of this feature map is that it intertwines the dynamics on $\mathcal M$ with a rotation system $R^t \colon \mathbb T^d \to \mathbb T^d$ parameterized by the eigenfrequencies corresponding to $\xi_{j_1}, \ldots, \xi_{j_d}$; that is, we have $ R^t \circ F(x) = F \circ \Phi^t(x) $, $\mu$-a.e., where $R^t(\theta) = \theta + \alpha t \mod 2 \pi $ and $\alpha = (\omega_{j_1}, \ldots, \omega_{j_d})$. Such a map $F$ is called a measure-theoretic semiconjugacy, and it is called a topological semiconjugacy if the functions $\xi_{j_1}, \ldots, \xi_{j_d}$ are continuous. In addition, the existence of the Koopman eigenfunction basis of $H_p$ leads to a closed-form expression for the dynamical evolution of observables in this space,
\begin{equation}
    \label{eq:koopman_evo_hp}
    U^t f = \sum_j \langle \xi_j, f\rangle e^{i\omega_j t} \xi_j, \quad \forall f \in H_p,
\end{equation}
which can be used to build predictive models when coupled with methods for estimating the expansion coefficients $\langle \xi_j, f\rangle$. Defining the cross-correlation function of two observables $f,g \in H$ as $C_{fg}(t) = \langle f, U^t g\rangle $, it follows from~\eqref{eq:koopman_evo_hp} that $C_{fg}(t) = \sum_j \langle g, \xi_j\rangle e^{i\omega_j t} \langle \xi_j, f\rangle$ has a quasiperiodic evolution that does not decay to 0 as $t \to \infty$ whenever $f, g \in H_p$. This suggests the possibility of long-term predictability of observables in this space.

In contrast to the structured nature of the Koopman spectra and evolution of observables in $H_p$, observables in the continuous spectrum subspace $H_c$ are characterized by a decay of correlations which is characteristic of weak-mixing dynamics, $ \lim_{T\to\infty} \int_0^T \lvert C_{fg}(t)\rvert \, dt/ T =0 $ for all $ f \in H$ and $g \in H_c$. In particular, there is no basis of $H_c$ compatible with the evolution of observables analogously to the Koopman eigenfunction basis of $H_p$, and this poses a significant methodological obstacle when dealing with observables in $H$ that are not well-approximated by elements of $H_p$. Indeed, by fundamental results from ergodic theory, the flow $\Phi^t$ is weak-mixing if and only if $H_p$ is the one-dimensional subspace of $H$ spanned by constant functions. If we consider mixing a hallmark of high dynamical complexity in a measure-theoretic sense, this suggests that in many real-world problems feature extraction and prediction based purely on Koopman eigenfunctions in $H$ will perform poorly.

To overcome these obstacles, considerable efforts have been made in recent years to develop techniques capable of consistently approximating the continuous spectra of unitary Koopman operators induced from measure-preserving dynamics; see, e.g., \cite{OttoRowley21,Colbrook24} for surveys. Building on DGV, the approach taken in this paper will be to approximate $V$ by a family of skew-adjoint operators that are diagonalizable on the entire Hilbert space $H$, and converge to $V$ in a spectral sense. In the next subsection, we lay out the notions of spectral convergence employed for that purpose.

\subsection{Spectral convergence}

We will use the following notions of the convergence of skew-adjoint operators \cite{Chatelin11,Oliveira09}.
\begin{definition}[Convergence of skew-adjoint operators] Let $A\colon D(A) \to \mathbb H$ be a skew-adjoint operator on a Hilbert space $\mathbb H$ and $A_{\tau} \colon D(A_{\tau}) \to \mathbb H$ a family of skew-adjoint operators indexed by $\tau>0$.
  \begin{enumerate}[(i)]
      \item The family $A_\tau$ is said to converge in strong resolvent sense to $A$ as $\tau\to 0^+$, denoted $A_\tau \srto A$, if for some (and thus, every) $z \in \mathbb C \setminus i \mathbb R$ the resolvents $R(z, A_\tau)$ converge strongly to $R(z, A)$; that is, $\lim_{\tau\to 0^+}R(z, A_\tau) f = R(z, A) f$ for every $f\in \mathbb H$.
      \item The family $A_\tau$ is said to converge in strong dynamical sense to $A$ as $\tau\to 0^+$, denoted $A_\tau \sdto  A$ if, for every $t \in \mathbb R$, the unitary operators $e^{tA_\tau}$ converge strongly to $e^{tA}$; that is, $\lim_{\tau\to0^+} e^{tA_\tau} f = e^{tA} f$ for every $f\in \mathbb H$.
  \end{enumerate}
  \label{def:src}
\end{definition}

It can be shown, \cite[e.g.,][Proposition~10.1.8]{Oliveira09}, that strong resolvent convergence and strong dynamical convergence are equivalent notions. For our purposes, this implies that if a family of skew-adjoint operators converges to the Koopman generator $V$ in strong resolvent sense, the unitary evolution groups generated by these operators consistently approximate the Koopman group generated by $V$.

Strong resolvent convergence and strong dynamical convergence imply the following form of spectral convergence \cite[e.g.,][Proposition~13]{DasEtAl21}.

\begin{theorem}
    With the notation of \cref{def:src}, let $\tilde E\colon \mathcal B(i \mathbb R) \to B(\mathbb H)$ and $\tilde E_\tau \colon \mathcal B(i \mathbb R) \to B(\mathbb H)$ be the spectral measures of $A$ and $A_\tau$, respectively, i.e., $A = \int_{i \mathbb R} \lambda\, d\tilde E(\lambda)$ and $A_\tau = \int_{i \mathbb R} \lambda\, d\tilde E_\tau(\lambda)$. Then, the following hold under strong resolvent convergence of $A_\tau$ to $A$.
\begin{enumerate}[(i)]
        \item For every element $\lambda \in \sigma(A)$ of the spectrum of $A$, there exists a sequence $\tau_1, \tau_2, \ldots \searrow 0$ and elements $\lambda_n \in \sigma(A_{\tau_n})$ of the spectra of $A_{\tau_n}$  such that $\lim_{n\to \infty} \lambda_n = \lambda$.
        \item For every bounded continuous function $h\colon i \mathbb R \to \mathbb C$, as $\tau\to 0^+$ the operators $h(A_\tau) = \int_{i \mathbb R} h(\lambda)\,d\tilde E_\tau(\lambda)$ converge strongly to $h(A) = \int_{i \mathbb R} h(\lambda) \,d\tilde E(\lambda)$.
        \item For every bounded Borel-measurable set $\Theta \in \mathcal B(i \mathbb R)$ such that $\tilde E(\partial \Theta) = 0$ (i.e., the boundary of $\Theta$ does not contain eigenvalues of $A_\tau$), as $\tau\to 0^+$ the projections $\tilde E_\tau(\Theta)$ converge strongly to $\tilde E(\Theta)$.
    \end{enumerate}
    \label{thm:spec-conv}
\end{theorem}

\subsection{Markov smoothing operators}
\label{sec:markov}

We will regularize the generator $V$ by pre- and post-composing it with kernel integral operators $G_\tau$ that converge strongly to the identity as $\tau \to 0^+$. These operators are induced from a family of kernel functions $p_\tau\colon \mathcal M \times \mathcal M \to \mathbb R$, $\tau > 0$, with the following properties.

% We will regularize the generator $V$ by pre- and post-composing it with kernel integral operators. These operators are induced from a family of kernel functions $p_\tau\colon \mathcal M \times \mathcal M \to \mathbb R$, $\tau > 0$, with the following properties.
\begin{enumerate}[label=(K\arabic*)]
    \item \label[prty]{prty:K1} $p_\tau$ is measurable, and it is continuously differentiable on $M\times M$.
    \item \label[prty]{prty:K2} $p_\tau$ is Markovian with respect to $\mu$; i.e., $p_\tau \geq 0$ and for every $x\in \mathcal M $ the normalization condition $\int_X p_\tau(x,\cdot)\,d\mu = 1$ holds.
\end{enumerate}
For every such kernel function $p_\tau$, we let $G_\tau \colon H \to H$ be the corresponding integral operator on $H$ with
\begin{equation}
    \label{eq:g_tau}
    G_\tau f = \int_X p_\tau(\cdot, x)f(x) \, d\mu(x).
\end{equation}
We then require that the following hold.
\begin{enumerate}[label=(K\arabic*),resume]
    \item \label[prty]{prty:K3} The family $\{G_\tau \colon H \to H\}_{\tau > 0} \cup \{ G_0 \colon= \Id \}$ forms a strongly continuous semigroup; i.e., $G_\tau G_{\tau'} = G_{\tau+\tau'}$ for every $\tau,\tau' \geq 0$ and $\lim_{\tau\to 0^+} G_\tau f = f$ for every $f \in H$.
    \item \label[prty]{prty:K4} $G_\tau$ is strictly positive; i.e., $\langle f, G_\tau f \rangle > 0$ for every nonzero $f \in H$.
    \item \label[prty]{prty:K5} $G_\tau$ is ergodic; i.e., $G_\tau f = f$ for all $\tau > 0$ iff $f$ is constant $\mu$-a.e.
\end{enumerate}
Note that every such kernel $p_\tau$ is necessarily symmetric, $p_\tau(x, y) = p_\tau(y, x)$, and the corresponding integral operator $G_\tau$ is self-adjoint and of trace class.

Let $k\colon \mathcal M \times \mathcal M \to \mathbb R_+$ be a positive measurable kernel function that is continuously differentiable and integrally strictly positive definite on $M \times M$. The latter requirement means \cite{SriperumbudurEtAl11} that for every finite Borel measure $\nu$ on $\mathcal M$ with support contained in $M$ we have
\begin{equation}
    \label{eq:intregrally_pos_def}
    \int_M \int_M k(x, y) \, d\nu(x)\, d\nu(y) > 0.
\end{equation}

Following DGV, we use such a kernel $k$ to build the family $p_\tau$ by first constructing a symmetric Markov kernel $p\colon \mathcal M \times \mathcal M \to \mathbb R$ by means of a bistochastic kernel normalization procedure proposed by \cite{CoifmanHirn13}, and adapted for the purposes of this work. The normalization procedure involves computing the continuous functions $d, q \in C^1(M)$ and the asymmetric kernel function $\hat k\colon \mathcal M \times \mathcal M \to \mathbb R$ defined as
\begin{equation}
    \label{eq:k_asym}
    d = \int_M k(\cdot, x) \, d\mu(x), \quad q = \int_M \frac{k(\cdot, x)}{d(x)} \, d\mu(x), \quad \hat k(x, y) = \frac{k(x,y)}{d(x) q^{1/2}(y)}.
\end{equation}
We then set
\begin{equation}
    \label{eq:p_kernel}
    p(x, y) = \int_X \hat k(x, z ) \hat k(z, y) \, d\mu(z).
\end{equation}
One can readily verify that $p(x, y) > 0$, $p(x, y) = p(y, x)$, and $\int_X p(x, \cdot) \, d\mu = 1$ for every $x, y \in M$; that is, $p$ is a strictly positive Markovian kernel. Moreover, by compactness of $X$, the restriction of $p$ on $M \times M$ is $C^1$.

Consider now the integral operator $\hat K \colon H \to C^1(M)$, where
\begin{displaymath}
    \hat K f = \int_M \hat k(\cdot, x) f(x) \, d\mu(x),
\end{displaymath}
and define $\tilde K \colon H \to H$ and $G \colon H \to H$ as $\tilde K = \iota \hat K$, $G = \tilde K \tilde K^*$. Observe that $G$ is a kernel integral operator with kernel $p$,
\begin{displaymath}
    G f = \int_X p(\cdot, x) f(x) \, d\mu(x),
\end{displaymath}
and by Markovianity and continuity of $p$, $G$ is a self-adjoint Markov operator of trace class. Moreover, by~\eqref{eq:intregrally_pos_def}, $G$ is a strictly positive operator, and thus admits an eigendecomposition
\begin{displaymath}
    G \phi_j = \lambda_j \phi_j,
\end{displaymath}
where $j \in \mathbb N_0$, the $\phi_j$ form an orthonormal basis of $H$, and the eigenvalues $\lambda_j$ are strictly positive and have finite multiplicities. By convention, we order the eigenvalues $\lambda_j$ in decreasing order, so that $\lambda_0 = 1$ by Markovianity. It also follows by strict positivity of $k$ and compactness of $X$ that the Markov process generated by $G$ is ergodic. This means that $\lambda_0$ is a simple eigenvalue, and the corresponding eigenfunction $\phi_0$ can be chosen as being $\mu$-a.e.\ equal to 1.

For the purposes of the physics-informed approximation schemes studied in this paper, it is important that (i) the eigenfunctions $\phi_j$ have representatives in $C^1(M)$; and (ii) directional derivatives of these representatives can be computed by taking derivatives of the kernel functions $\hat k$. First, observe that an eigendecomposition of $G$ can be obtained from a singular value decomposition (SVD) of the compact operator $\tilde K$,
\begin{displaymath}
    \tilde K = \sum_{j=0}^\infty \phi_j \sigma_j \langle \gamma_j, \cdot \rangle,
\end{displaymath}
where $\phi_j$ (resp.\ $\gamma_j$) are left (resp.\ right) singular vectors forming an orthonormal basis of $H$, and $\sigma_j = \lambda_j^{1/2}$ are strictly positive singular values. One finds that the functions
\begin{equation}
    \label{eq:varphi}
    \varphi_j = \frac{1}{\sigma_j} \hat K \gamma_j
\end{equation}
are  representatives of $\phi_j$ in $C^1(M)$; that is, $\phi_j = \iota \varphi_j$ which implies that $\phi_j$ lies in $D(V)$. As a result, using~\eqref{eq:v_grad}, we can compute the action of the generator on the eigenbasis as
\begin{equation}
    \label{eq:vgrad_phi}
    V \phi_j = \iota \vec V \cdot \nabla \varphi_j \equiv \tilde K' \gamma_j,
\end{equation}
where $\tilde K' = \iota \hat K'$ and $\hat K' \colon H \to C(M)$ is the integral operator
\begin{displaymath}
    \hat K' f = \int_X \hat k'(\cdot, y) f(y) \, d\mu(y), \quad \hat k'(\cdot, y) = \vec V \cdot \nabla \hat k(\cdot, y).
\end{displaymath}
We therefore see that $V \phi_j$ can be computed from the right singular vectors $\gamma_j$ so long as the directional derivatives $\hat k'(\cdot, y)$ of the kernel sections $\hat k(\cdot, y)$ along $\vec V$ are known. For later convenience, we set $\phi_j' = V \phi_j$ and $\varphi_j' = \vec V \cdot \nabla \varphi_j$.

Finally, to build the Markov semigroup $G_\tau$, we define the self-adjoint operator $\Delta\colon D(\Delta) \to H$ on the dense domain $D(\Delta) = \{ f \in H: \sum_{j=1}^\infty \eta_j^2 \lvert \langle \phi_j, f\rangle\rvert^2 < \infty \} $ via the eigendecomposition
\begin{displaymath}
    \Delta \phi_j = \eta_j \phi_j, \quad \eta_j = \frac{\lambda_j^{-1} -1}{\lambda_1^{-1} -1}.
\end{displaymath}
Intuitively, we think of $\Delta$ as a Laplace-type operator, and we will use it as the generator of our Markov semigroup,
\begin{displaymath}
    G_\tau \colon= e^{-\tau \Delta}, \quad \tau \geq 0.
\end{displaymath}
In \cite[Theorem~1]{DasEtAl21}, it is shown that for each $\tau>0$, $G_\tau$ is a Markovian kernel integral operator of the form~\eqref{eq:g_tau}, with kernel $p_\tau \in C^1(M\times M)$ given by a Mercer sum,
\begin{displaymath}
    p_\tau(x, y) = \sum_{j=0}^\infty \lambda_{j,\tau} \varphi_j(x) \varphi_j(y), \quad \lambda_{j,\tau} = e^{-\tau \eta_{j,\tau}},
\end{displaymath}
that converges in $C^1(M\times M)$ norm. In particular, the kernels $p_\tau$ satisfy all the requisite \crefrange{prty:K1}{prty:K5} listed above. Moreover, building the semigroup $G_\tau $ from the generator $\Delta$ means that the eigenvalues/eigenvectors of $G_\tau$ can be directly obtained from eigenvalues/eigenvectors of $G$. This will be useful in \cref{sec:numimpl} for developing (data-driven) Galerkin methods.

\subsection{Reproducing kernel Hilbert spaces}
\label{sec:rkhs}

We outline a few results from RKHS theory used in our spectral approximation scheme. Further details on these constructions and terminology can be found in \cite{DasEtAl21}. Comprehensive expositions on RKHSs are available, e.g., in \cite{PaulsenRaghupathi16,SteinwartChristmann08}.

\begin{lemma}
    \label{lem:pos_def}
    Let $\kappa\colon M \times M \to \mathbb R$ be a positive-definite, Borel-measurable kernel function. Then, for every finite Borel measure $\nu$ on $M$ with (compact) support $X_\nu$, the kernel $\kappa_2\colon M \times M \to \mathbb R$ defined as $\kappa_2(x,y) = \int_M \kappa(x, z) \kappa(z, y) \, d\nu(z)$ is positive definite. Moreover, if $\kappa$ is strictly positive-definite on $X_\nu$ then so is $\kappa_2$.
\end{lemma}

\begin{proof}
    See \cite[Lemma~12]{DasEtAl21}.
\end{proof}

Next, observe that the kernel~$p$ from~\eqref{eq:p_kernel} can be expressed as
\begin{displaymath}
    p(x, y) = \int_M \frac{\tilde k(x,z) \tilde k(z, y)}{\tilde d(x) \tilde d(z)} \, d\mu(z), \quad \tilde k(x,y) = \frac{k(x,y)}{q^{1/2}(x) q^{1/2}(y)}, \quad \tilde d(x) = \frac{d(x)}{q^{1/2}(x)}.
\end{displaymath}
By strict positivity of the kernel $k$, $\tilde d$ and $q$ are strictly positive, continuous functions, and thus bounded away from zero on the compact set $M$. As a result, since $k$ is integrally strictly positive so is $\tilde k$. Moreover, by \cref{lem:pos_def}, $(x,y) \mapsto \int_M \tilde k(x, z) \tilde k(z, y) \, d\mu(z)$ is positive-definite on $M$ and strictly positive-definite on $X_\nu$, and $p$ inherits these properties by strict positivity of $\tilde d$. Similarly to $p_\tau$, $p$ has a Mercer expansion,
\begin{displaymath}
    p(x,y) = \sum_{j=0}^\infty \lambda_j \varphi_j(x) \varphi_j(y),
\end{displaymath}
which converges in $C^1(M \times M)$ norm.

Let $\mathcal H \subset C^1(M)$ be the RKHS on $M$ with reproducing kernel $p$. We denote the inner product of $\mathcal H$ by $\langle \cdot, \cdot \rangle_{\mathcal H}$. For a subset $S \subseteq M$ we define the Hilbert subspace $\mathcal H(S) \subseteq \mathcal H$ as $\mathcal H(S) = \overline{\spn\{ p(x, \cdot) : x \in S\}}^{\lVert \cdot\rVert_{\mathcal H}}$. This space can be identified with the restriction $\mathcal H \rvert_S \subseteq C(S)$ of $\mathcal H$ onto $S$. By strict positivity of the $\lambda_j$, $\mathcal H(X)$ can thus be identified with a dense subspace of $C(X)$. Moreover, $H^1 := \iota \mathcal H$ is a dense subspace of $H$ and the embedding $\mathcal H(X) \hookrightarrow H$ is compact. The subspace $H^1$ can be equivalently characterized by a decay condition of expansion coefficients of elements of $H$ with respect to the $\phi_j$ basis,
\begin{displaymath}
    H^1 = \left\{ f \in H: \sum_{j=0}^\infty \frac{\lvert \langle \phi_j, f\rangle\rvert^2}{\lambda_j} < \infty \right\},
\end{displaymath}
and becomes a Hilbert space for the inner product $\langle f, g\rangle_{H^1} = \sum_{j=0}^\infty \lambda_j^{-1} \langle f, \phi_j\rangle \langle \phi_j, g\rangle$. Note that, as a vector space, $H^1$ is equal to the domain of $G^{-1/2}$, and
\begin{displaymath}
    \lVert f \rVert_{H^1} = \lVert G^{-1/2} f\rVert_H \geq \lVert f\rVert_H
\end{displaymath}
holds since $G$ is a contraction. Moreover, $H^1$ is a subspace of $D(V)$ and, by exponential decay of $\lambda_{j,\tau}$ with increasing $\eta_j$, contains $\ran G_\tau$ as a subspace for all $\tau>0$. In particular, $G^{-1/2} G_\tau$ is a bounded operator on $H$ for every $\tau>0$.

Using standard RKHS constructions (see, e.g., \cite[Section~4.1]{DasEtAl21}), it can be shown that $P\colon H \to \mathcal H$ with
\begin{displaymath}
    P f = \int_X p(\cdot, x)f(x) \, d\mu(x)
\end{displaymath}
is a well-defined, Hilbert--Schmidt operator, whose range is a dense subspace of $\mathcal H(X)$. Moreover, the adjoint $P^*\colon \mathcal H \to H$ maps functions in $\mathcal H$ to their corresponding equivalence classes in $H$ (i.e., $P^* f = \iota f)$), leads to the factorization $G = P^*P$ of the Markov integral operator $G$ from \cref{sec:markov}. Using these facts, we can deduce the following polar decomposition of $P$,
\begin{equation}
    \label{eq:polar_decomp_p}
    P = T G^{1/2},
\end{equation}
where $T\colon H \to \mathcal H$ is an isometry with $\ran T = \mathcal H(X)$. Moreover, the functions $\psi_j := \lambda_j^{1/2} \varphi_j $, $j \in \mathbb N_0$, form an orthonormal basis of $\mathcal H(X)$, while $\{\phi^{(1)}_j := \iota \psi_j\}_{j \in \mathbb N_0}$ is an orthonormal basis of $H^1$. By \eqref{eq:polar_decomp_p}, we have $P^* T = G^{1/2}$. This implies that
\begin{equation}
    \label{eq:rkhs_unitary}
    T f = \sum_{j=0}^\infty \langle \phi_j, f\rangle \psi_j,
\end{equation}
and thus that the $L^2$ element $ \iota Tf = P^* T f$ has basis expansion $\iota T f  = \sum_{j=0}^\infty \lambda_j^{1/2} \langle \phi_j, f \rangle \phi_j$. Since $\lambda_j < 1$ for all $j \in \mathbb N$, we can interpret $T$ as a smoothing operator even though it is unitary as a map from $H$ to $\mathcal H(X)$. Later on, we will make use of a truncated version $T_{L'} \colon H \to \mathcal H$ of $T$, defined for $L' \in \mathbb N_0$ as

\begin{equation}
    \label{eq:rkhs_partial_iso}
    T_{L'} f = \sum_{j=0}^{L'} \langle \phi_j, f\rangle \psi_j.
\end{equation}
These operators are partial isometries that converge strongly to $T$ as $L' \to \infty$.

Next, the Nystr\"om operator $\mathcal N\colon D(\mathcal N) \to \mathcal H$ is an unbounded operator with dense domain $D(\mathcal N) = H^1 \subseteq H$ and range $\mathcal H(X)$ that maps elements in its domain to their representatives in $\mathcal H(X)$ according to the formula \begin{equation}
    \label{eq:nystrom}
    \mathcal N f = \sum_{j=0}^\infty \frac{\langle \phi_j, f \rangle}{\lambda_j^{1/2}} \psi_j.
\end{equation}
This induces in turn a Dirichlet energy functional $\mathcal E \colon H^1 \to \mathbb R$ that provides a measure of average spatial variability of elements in its domain with respect to the RKHS $\mathcal H$,
\begin{displaymath}
    \label{eq:dirichlete}
    \mathcal E(f) = \frac{\lVert \mathcal N f\rVert_{\mathcal H}^2}{\lVert f\rVert_H^2} - 1.
\end{displaymath}
Note that $\mathcal E(f) \geq 0$ with equality iff $f$ is constant $\mu$-a.e.

\begin{remark}
    When implemented using the pullback of a kernel $k^{(Y)}$ on data space $Y$ under a map $F \colon \mathcal M \to Y$, the isometric embedding~\eqref{eq:rkhs_unitary} and Nystr\"om extension~\eqref{eq:nystrom} are pullbacks of everywhere-defined functions on $Y$, i.e., $T f = g \circ F$ and $\mathcal N f = h \circ F$ with $g,h \colon Y \to \mathbb R$. In the case of the bistochastic kernel construction described in \cref{sec:markov}, these functions are given by
    \begin{displaymath}
        g = \sum_{j=0}^\infty \langle \phi_j, f\rangle \psi_j^{(Y)}, \quad
        h = \sum_{j=0}^\infty \frac{\langle \phi_j, f\rangle}{\sigma_j} \psi_j^{(Y)},
    \end{displaymath}
    where $\psi_j^{(Y)} = \int_M \hat k^{(Y)}(\cdot, F(x)) \gamma_j(x) \, d\mu(x)$. Here, $\hat k^{(Y)} \colon Y \times Y \to \mathbb R$ is the kernel on data space obtained by normalization of $k^{(Y)}$ using analogous steps to~\eqref{eq:k_asym}.
    \label{rk:nystrom_pullback}
\end{remark}

To verify well-posedness of our spectral approximation scheme, we will also need the space of distributions $H^{-1} := (H^1)^*$ given by the continuous dual of $H^1$ as a subspace of $H$. Since $H$ is a self-dual space, it can be identified with a subspace of $H^{-1}$, leading to the triple
\begin{displaymath}
    H^1 \subset H \subset H^{-1}.
\end{displaymath}
By Hilbert space duality we also have that $H^{-1}$ is a Hilbert space with orthonormal basis $\{ \phi^{(-1)}_j := \lambda_j^{-1/2} \phi_j \} $. In particular, the action of a distribution $\alpha = \sum_{j=0}^\infty a_j \phi^{(-1)}_j \in H^{-1}$ on an observable $f = \sum_{j=0}^\infty c_j \phi_j^{(1)} \in H^1 $ is given by $\alpha f = \sum_{j=0}^\infty a_j c_j$. It follows from the fact that $\lambda_j$ decays to 0 as $j \to \infty$ with no accumulation points other than 0 that the embeddings $H^1 \hookrightarrow H$ and $H \hookrightarrow H^{-1}$ are both compact.

\section{Spectral approximation scheme}
\label{sec:scheme}

In what follows, we will build a family of densely-defined operators $V_{z,\tau}\colon D(V_{z,\tau}) \to H$, $D(V_{z,\tau}) \subseteq H$, parameterized by $z,\tau > 0$ with the following properties.
\begin{enumerate}[label=(P\arabic*)]
    \item \label[prty]{prty:p1} $V_{z,\tau}$ is skew-adjoint.
    \item \label[prty]{prty:p2} $V_{z,\tau}$ has compact resolvent.
    \item \label[prty]{prty:p3} $V_{z,\tau}$ is real, $(V_{z,\tau} f)^* = V_{z,\tau}(f^*)$ for all $f\in D(V)$.
    \item \label[prty]{prty:p4} $V_{z,\tau}$ annihilates constant functions, $V_{z,\tau} \bm 1 = 0$.
    \item \label[prty]{prty:p5} $V_{z,\tau} \srto V$ in the iterated limits $z \to 0^+$ after $\tau \to 0^+$.
\end{enumerate}
In the above, $\tau$ is the kernel smoothing parameter introduced in \cref{sec:markov}, used here to effect compactification of the resolvent. Moreover, as we will see momentarily, the parameter $z$ can be intuitively thought of as controlling the approximate ``invertibility'' of a bounded transformation applied to $V$ prior to compactification.

Let $\tilde H = \{ f \in H: \langle \bm 1, f\rangle = 0 \}$ be the closed subspace of $H$ consisting of zero-mean functions with respect to the invariant measure $\mu$. Let also $\Pi_0= \langle \bm 1, \cdot \rangle_H$ and $\tilde \Pi = \Id - \Pi_0$, where $\tilde\Pi$ is the orthogonal projection onto $\tilde H$. Since $V \bm 1 = 0$, $\tilde H$ is an invariant subspace of $V$. To construct $V_{z,\tau}$, we begin by considering a bounded, skew-adjoint transformation of the generator, obtained from the bounded, continuous, antisymmetric function $q_z\colon i\bbR \to \bbC$,
\begin{equation}
    \label{eq:qz}
    q_z(i\omega) = \frac{i\omega}{z^2 + \omega^2}, \quad \ran q_z = i\left[-\frac{1}{2z}, \frac{1}{2z}\right],
\end{equation}
via the Borel functional calculus:
\begin{displaymath}
    Q_z := q_z(V\rvert_{\tilde H}) \equiv R_z^* V R_z, \quad Q_z \in B(\tilde H), \quad Q_z^* = - Q_z,
\end{displaymath}
where $R_z = R(z, V)$. Note that above (and in what follows when convenient) we consider $V$ and $R_z$ as operators on $\tilde H$.

The function $q_z$ is plotted in \cref{fig:qz}.
While it is not invertible, its restriction $\tilde{q}_z$ on the subset $i\Omega_z := i(-\infty, -z] \cup i[z, \infty)$ of the imaginary line \textit{is} invertible, with inverse $\tilde{q}_z^{-1}\colon i[-(2z)^{-1}, (2z)^{-1}] \setminus \{0\} \to i\bbR$,
\begin{equation}
    \label{eq:qzinverse}
    \tilde{q}_z^{-1}(i\omega) = i \frac{1 + \sqrt{1 - 4z^2\omega^2}}{2\omega}\, .
\end{equation}
As a result, $\tilde V_z := \tilde{q}_z^{-1}(Q_z)$ is a skew-adjoint operator with dense domain $D(\tilde V_z) \subset \tilde H$ that reconstructs the generator $V$ on the spectral domain $\Omega_z$. We extend $\tilde V_z$ to the full Hilbert space $H$ by defining $V_z = \tilde \Pi \tilde V_z \tilde \Pi$. The value of $z$ controls the approximate invertibility of the function $q_z$, as values in $i(-z, z) \setminus \{ 0 \}$ will be mapped to the to the ``wrong'' value under $\tilde q_z^{-1}$. This intuition informs the $z \to 0^+$ limit, which is made precise in \cref{thm:vz_conv} below. This result is central to our scheme as it allows approximating the generator in a spectral sense using the bounded, skew-adjoint operators $Q_z$.

\begin{figure}
    \centering
    \includegraphics[width=0.5\linewidth, draft=false]{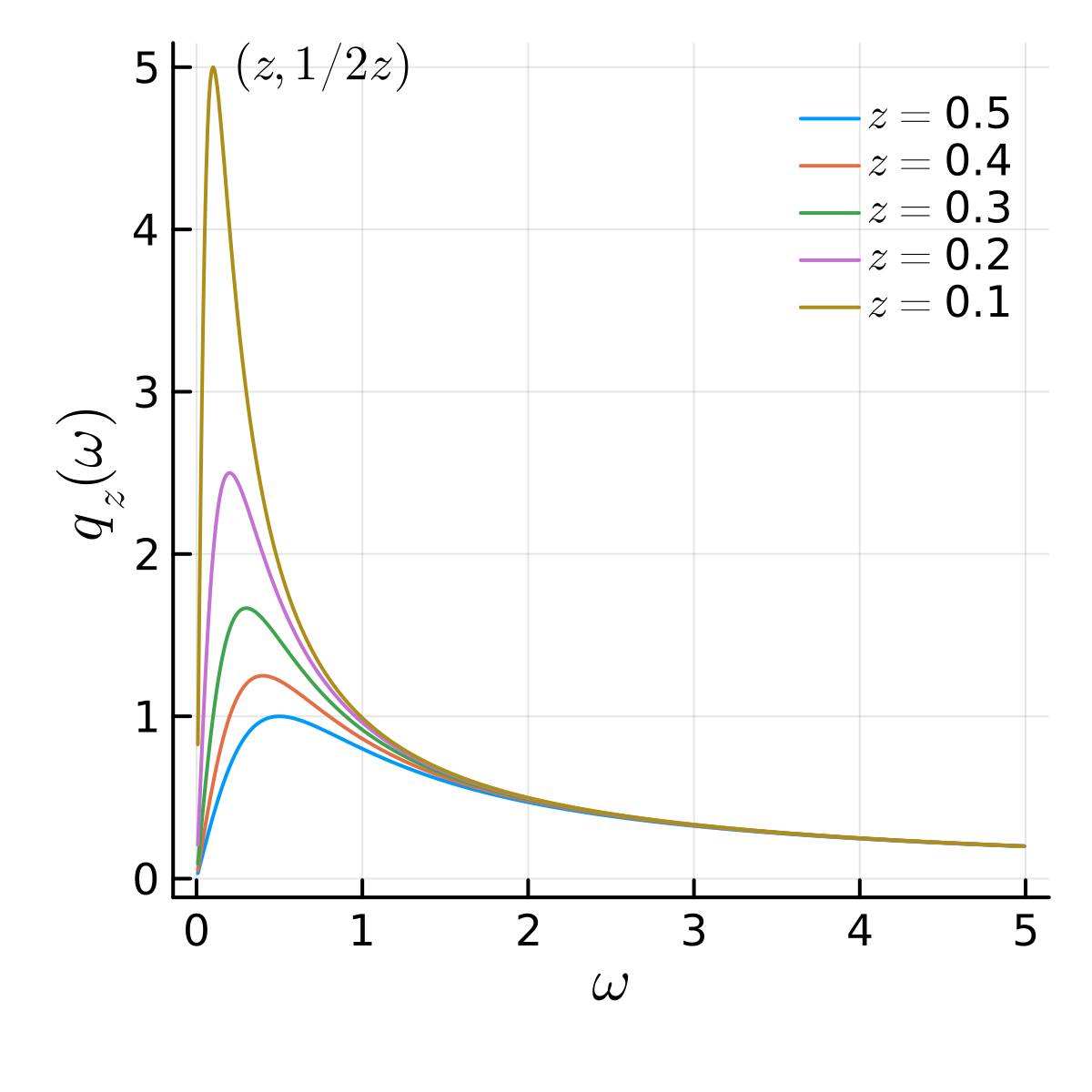}
    \caption{Graphs of the transformation function $q_z$ for the generator for representative values of $z$.}
    \label{fig:qz}
\end{figure}

\begin{theorem}
    As $z \to 0^+$, $V_z$ converges to $V$ in strong resolvent sense. As a result, it converges to $V$ in strong dynamical sense, $e^{t V_z} \sdto U^t$, and it converges spectrally in the sense of \cref{thm:spec-conv}.
    \label{thm:vz_conv}
\end{theorem}

\begin{proof}
    Recall the spectral resolution of $V$ in~\eqref{eq:spec_decomp}, $V = \int_{i\bbR} i \omega\, dE(i\omega)$, and define the PVM $\tilde E\colon \mathcal B(i \mathbb R) \to B(\tilde H)$ as $\tilde E = \tilde \Pi \circ E$. Note that $\tilde{E}$ does not have an atom at 0 (i.e., $\tilde{E}(\{0 \}) = 0$) and $V\rvert_{\tilde H} = \int_{i\mathbb R} i\omega \, d\tilde E(i\omega)$. For $y >0$, we also define the resolvent function $r_y\colon i\bbR \to \bbC$ as $r_y(i\omega) = 1/(y -i \omega)$, so that $R(y, A) = r_y(A)$ for any skew-adjoint operator $A$.

    We then have that
    \begin{displaymath}
        R(y, \tilde{V}_z) = r_y(\tilde{V}_z) = r_y(\tilde{q}^{-1}_z( q_z(V\rvert_{\tilde H})) = \int_{i \bbR} r_y(\tilde{q}^{-1}_z (q_z(i\omega)) \, d\tilde{E}(i\omega),
    \end{displaymath}
    giving
    \begin{align}
        \nonumber R(y, \tilde{V}_z) &= \int_{i \Omega_z} r_y(\tilde{q}^{-1}_z (q_z(i\omega)) \, d \tilde{E}(i\omega) + \int_{-z}^z r_y(\tilde{q}^{-1}_z (q_z(i\omega)) \, d \tilde{E}(i\omega) \\
        \nonumber &= \int_{i \Omega_z} r_y(i\omega) \, d \tilde{E}(i\omega) + \int_{-z}^z r_y(i\omega) \, d \tilde{E}(i\omega) - \int_{-z}^z r_y(i\omega) \, d \tilde{E}(i\omega) + \int_{-z}^z r_y(\tilde{q}^{-1}_z (q_z(i\omega)) \, d \tilde{E}(i\omega) \\
                  &= R(y, V\rvert_{\tilde H}) - \int_{-z}^z r_y(i\omega) \, d \tilde{E}(i\omega) + \int_{-z}^z r_y(\tilde{q}^{-1}_z (q_z(i\omega)) \, d \tilde{E}(i\omega) \label{eq:Rterms}
    \end{align}
    We will prove the strong resolvent convergence $\tilde{V}_z \srto V\rvert_{\tilde H}$ by showing that the two rightmost terms of \eqref{eq:Rterms} converge strongly to zero on $\tilde H$.

    For $\omega \in (-z, z) \setminus \{ 0 \}$,
    \begin{equation*}
        \tilde{q}^{-1}_z(q_z(i\omega)) = \tilde{q}^{-1}_z\left(\frac{i\omega}{z^4 + \omega^2}\right) = iz^2 / \omega,
    \end{equation*}
    so we have
    \begin{equation*}
        \abs{r_y(\tilde{q}^{-1}_z (q_z(i\omega))} = \left\lvert\frac{1}{y - z^2 / \omega}\right\rvert = \left(y^2 + \frac{z^2}{\omega^2}\right)^{-1/2} \leq \frac{1}{y}.
    \end{equation*}
    Similarly, $\abs{r_y(i\omega)} < 1/y$.

    Next, for $f \in \tilde H$, let $\tilde E_f \colon \mathcal B(i \mathbb R) \to \mathbb R_+$ be the positive finite Borel measure given by $\tilde E_f(\Omega) = \langle f, \tilde E(\Omega) f\rangle$. Using the above upper bounds for $\lvert r_y(\tilde{q}^{-1}_z(q_z(i\omega))) \rvert$ and $\abs{r_y(i\omega)}$ in combination with \cref{lem:bound} from \ref{app:bound}, we get
    \begin{align*}
       \lVert R(y, \tilde{V}_z)f - R(y, V)f \rVert_H^2 &= \int_{-iz}^{iz} |\tilde{q}^{-1}_z (q_z(i\omega)) - r_y(i\omega)|^2 \, d \tilde E_f(i\omega) \\
                                                                &= \int_{(-iz, iz) \setminus \{ 0 \}} \left\lvert\tilde{q}^{-1}_z (q_z(i\omega)) - r_y(i\omega)\right\rvert^2 \, d \tilde E_f(i\omega) \\
                                                                &\leq \frac{4}{y^2} \int_{(-iz, iz) \setminus \{ 0 \}} \, d\tilde E_f(i\omega),
   \end{align*}
    where we have used the fact that $\tilde E(\{0\}) = 0$ to obtain the second equality. It follows from upper continuity of finite measures that for every sequence $z_1, z_2, \ldots$ of strictly positive numbers decreasing to 0,
    \begin{displaymath}
        \limsup_{n \to \infty}\lVert R(y, \tilde{V}_{z_n})f - R(y, V)f \rVert_H^2 \leq \frac{4}{y^2} \tilde E_f\left(\bigcap_{n>0} (-iz_n, iz_n) \setminus \{ 0 \}\right) = \frac{4}{y^2} \tilde E_f(\emptyset) = 0,
    \end{displaymath}
    proving that $\tilde V_z \srto V$ on $\tilde H$. We can then conclude that $V_z \srto V$ on $H$ by noting that $R(y, V\rvert_{\tilde H}) - R(y, V_z) = (R(y, V) - R(y, \tilde V_z)) \tilde \Pi$.
\end{proof}

Next, let  $b_z\colon i\bbR \to i\bbR$ be any continuous extension of $\tilde{q}_z^{-1}$ to the entire imaginary line. For a choice of kernels $p_\tau\colon \mathcal M \times \mathcal M \to \mathbb R$ satisfying \crefrange{prty:K1}{prty:K5}, we approximate $Q_z$ by the compact operators $Q_{z,\tau}\colon H \to H$, $\tau > 0$, where
\begin{equation}
    \label{eq:qz_op}
    Q_{z,\tau} = R_z^* V_\tau R_z, \quad V_\tau = G_{\tau / 2} V G_{\tau / 2}.
\end{equation}
Define the unbounded, skew-adjoint operator $V_{z,\tau}\colon D(V_{z,\tau}) \to H$ as
\begin{displaymath}
    V_{z,\tau} = \tilde V_{z,\tau} \tilde \Pi, \quad \tilde V_{z,\tau} = b_z(Q_{z,\tau}).
\end{displaymath}
Our main spectral approximation result is as follows.
\begin{theorem}
    For $z, \tau > 0$ the operators $V_{z,\tau}$ satisfy \crefrange{prty:p1}{prty:p5}. As such, in the iterated limit $z \to 0^+$, $\tau \to 0^+$, $V_{z, \tau}$ converges spectrally to $V$ in the sense of \cref{thm:spec-conv}.
    \label{thm:vztau_conv}
\end{theorem}
\begin{proof}
\Crefrange{prty:p1}{prty:p4} follow from the construction of $V_{z, \tau}$. Property \ref{prty:p5} follows from the fact that $\lim_{\tau \to 0} G_\tau f = f$ for $f \in H$ (\cref{prty:K3}), so we have $V_{z,\tau} \srto V_z$ as $\tau \to 0^+$ and the subsequent application of \cref{thm:vz_conv}.
\end{proof}

Consider now an eigendecomposition of the compact, skew-adjoint operator $Q_{z,\tau}$,
\begin{equation}
    \label{eq:eigQ}
    Q_{z,\tau} \xi_{j, z, \tau} = \beta_{j, z, \tau} \xi_{j, z, \tau}, \quad \beta_{j, z, \tau} \in i\bbR, \quad \xi_{j,z,\tau} \in \tilde H \setminus \{ 0 \},
\end{equation}
where the eigenpairs are indexed by $j \in \bbN_0$ and $\{\xi_{j, z, \tau}\}_{j \in \bbN}$ is an orthonormal basis of $\tilde H$. Since $Q_{z,\tau}$ is a real operator (i.e., $(Q_{z,\tau} f)^* = Q_{z,\tau}(f^*)$), it follows that the eigenvectors $\xi_{j,z,\tau}$ corresponding to nonzero $\beta_{j,z,\tau}$ come in complex-conjugate pairs.

Applying $b_z$ to the eigenvalues of $Q_{z,\tau}$ leads to an eigendecomposition of $\tilde V_{z,\tau}$:
\begin{displaymath}
    \tilde V_{z,\tau} \xi_{j, z, \tau} = i \omega_{j, z, \tau} \xi_{j, z, \tau}, \quad i \omega_{j, z, \tau} = b_z(\beta_{j, z, \tau}).
\end{displaymath}
Correspondingly, $ \{\xi_{j, z, \tau}\}_{j \in \bbN} \cup \{ \xi_{0,z,\tau} = \bm 1 \}$ is an orthonormal basis of $H$ consisting of eigenfunctions of $\tilde V_{z,\tau}$ with $\xi_{0,z,\tau}$ corresponding to eigenfrequency $\omega_{0,z,\tau} = 0$.

In light of \cref{thm:vztau_conv}, we view the eigenpairs $(\omega_{j,z,\tau}, \xi_{j,z,\tau})$ as generalizations of the true generator eigenpairs $(\omega_j, \xi_j)$ from \cref{sec:spec_decomp} that provide an orthonormal basis of $H_p$. Defining the unitary operators $U^t_{z,\tau} = e^{t V_{z,\tau}}$, we have
\begin{equation}
    \label{eq:koopman_evo_approx}
    U^t_{z,\tau} f= \sum_{j \in \mathbb N_0} \langle f, \xi_{j,z,\tau} \rangle e^{i \omega_{j,z,\tau}} \xi_{j,z,\tau}, \quad \forall f \in H,
\end{equation}
which generalizes~\eqref{eq:koopman_evo_hp} to the entire Hilbert space $H$. Moreover, by equivalence of strong resolvent convergence of $V_{z,\tau}$ to $V$ (\cref{prty:p5}) and strong dynamical convergence, we have $\lim_{z\to 0^+} \lim_{\tau\to 0^+} U^t_{z,\tau} f = U^t f$, so~\eqref{eq:koopman_evo_approx} consistently approximates the true Koopman evolution of observables in $H$.

\begin{remark}
    In general, the spectrum of $Q_{z,\tau}$ does not need to be a subset of the interval $i[-(2z)^{-1}, (2z)^{-1}]$ that contains the spectrum of $Q_z$. Still, in the numerical experiments performed in this work (see \cref{sec:examples}) we have not observed any cases where the inclusion $\sigma(Q_{z,\tau}) \subseteq i[-(2z)^{-1}, (2z)^{-1}]$ is violated. With such an inclusion, $b_z(\beta_{j, z, \tau}) = \tilde{q}_z^{-1}(\beta_{j, z, \tau})$ and an explicit choice of extension $b_z \supseteq \tilde q_z^{-1}$ is not needed.
\end{remark}

In \cref{sec:numimpl} below, our goal will be to produce approximations of the eigenpairs $(\omega_{j,z,\tau}, \xi_{j,z,\tau})$ using data-driven Galerkin methods for variational generalized eigenvalue problems. Besides the $L^2$ elements $\xi_{j,z,\tau}$, we will also be interested in smooth eigenfunctions $\zeta_{j,z,\tau}$ lying in the RKHS $\mathcal H \subset C^1(M)$ from \cref{sec:rkhs}. We obtain such eigenfunctions by applying the isometry $T\colon H \to \mathcal H$, viz.
\begin{equation}
    \label{eq:zeta_def}
    \zeta_{j,z,\tau} = T \xi_{j,z,\tau}.
\end{equation}
The resulting functions $\zeta_{j,z,\tau} \in \mathcal H$ are eigenvectors of the skew-adjoint operator $W_{z,\tau} := T^* V_{z,\tau} T$ on $\mathcal H$ which is unitarily equivalent to $V_{z,\tau}$, and form an orthonormal basis of $\mathcal H(X) \subseteq \mathcal H$. Furthermore, from the polar decomposition~\eqref{eq:polar_decomp_p} we get $\iota e^{t W_\tau} T = G^{1/2} e^{t V_\tau}$. From the latter and \cref{thm:vztau_conv} we deduce that $W_\tau$ generates a smooth approximation of Koopman evolution on $H$ (while being unitary on $\mathcal H$),
\begin{displaymath}
    \lim_{\tau\to 0^+} \iota e^{t W_\tau} T f = G^{1/2} U^t f, \quad \forall f \in H.
\end{displaymath}

\section{Finite-dimensional approximation}
\label{sec:numimpl}

\begin{algorithm}
    \caption{Numerical operator approximation for training data $x_0, x_1, \dotsc, x_{N - 1}$, symmetric positive-definite Markovian kernel $p_N$, resolvent parameter $z >0$, smoothing parameter $\tau > 0$, and number of basis functions $L \leq N -1$. Additionally, $\tilde{q}_z^{-1}(i\omega)$ is the function defined in~\eqref{eq:qzinverse}, and $\vec{V}$ denotes the dynamical vector field that generates the dynamics, \eqref{eq:v_phi}. The algorithm computes a rank-$L$ skew-adjoint operator $V_{z, \tau, L, N} \in \hat{H}_N$, along with its nonzero eigenfrequencies $\omega_{j,z, \tau,L,N}$ and eigenfunctions $\xi_{j,z, \tau,L,N}$. The corresponding operator $W_{z, \tau, L, N}$ with eigenfunctions $\hat\zeta_{j,z, \tau,L,N}$ in the RKHS $\mathcal{H}_N$ is also computed.}
  \label{alg:numerical}
  \begin{algorithmic}[1]
    
      \STATE{Compute kernel eigenvectors $ \{ \phi_{j,N}\}_{j=0}^L$ and corresponding eigenvalues $ \{ \lambda_{j,N} \}_{j = 0}^L$ from the kernel integral operator defined by $p_N$, i.e. $G_Nf = \int_X p_N(\cdot, x) f(x)\, d\mu_N(x)$.}
      \STATE{Define functions $\varphi_{j,N}$ to be continuous representatives of $\phi_{j,N}$ with respect to the sampling measure $\mu_N$; that is, $\varphi_{j,N}(x_n) = \phi_{j,N}(x_n)$ for all $ n \in \{ 0, \ldots, N-1 \}$ (equivalently; $\iota_N \varphi_{j,N} = \phi_{j,N}$).}
    %   \STATE{Compute eigenfunction derivatives $\{\phi'_{j,N}\}_{j=1}^L$ through the directional derivative of $\varphi_j$ with respect to the vector field $\vec{V}$, \[\iota_N \varphi'_{j,N} = \phi'_{j,N}, \quad \varphi_j' = \vec V \cdot \nabla \varphi_j.\] We can then define the approximation of the Koopman generator applied to these eigenvectors as generator approximation with these eigenvectors, such that $V_N \varphi_{j, N} \equiv \phi'_{j,N}$.}
    \STATE{Compute eigenfunction derivatives $\{\phi'_{j,N}\}_{j=1}^L$, where $\phi'_{j,N} = \iota_N \varphi'_{j,N}$ and $\varphi_j' = \vec V \cdot \nabla \varphi_j.$ Define the Koopman generator approximation $V_N$ such that $V_N \phi_{j, N} = \phi'_{j,N}$.}
    \STATE{Compute the $L \times L$ matrix representation $\bm V_N$ of $V_N$ with elements $\bm V_{ij, N} = \langle \phi_{i, N} , V_N \phi_{j, N} \rangle_N$. Compute $L \times L$ matrices $\bm V_N^{(2)}$ and $\bm{\tilde V}_N$, where  $\bm V_{ij, N}^{(2)} = \langle V_N \phi_{i, N} , V_N \phi_{j, N} \rangle_N$, and $\bm{\tilde V}_N = (\bm V_N - \bm V_N^\top)/2$.}
    %   \STATE{Compute eigenvector derivatives $\{ \phi'_{j,N}\}_{j=1}^L$ through the directional derivative of $p_N$ with respect to the vector field $\vec{V}$ i.e. $ \phi'_{j,N} = \frac{1}{\lambda_{j,N}} \int_X p'_N(\cdot, x) \phi'_{j,N}(x) d\mu_N(x)$. We can then define the approximation of the Koopman generator applied to these eigenvectors as generator approximation with these eigenvectors, such that $V_N \varphi_{j, N} \equiv \phi'_{j,N}$.}
      \STATE{Define $\bm \Lambda_{\tau/2,N}$ to be the $L\times L$ matrix approximation of $G_{\tau / 2}$ in the $\phi_{j,N}$ basis, where $\bm \Lambda_{\tau/2,N} = \diag(\lambda_{1,\tau/2,N}, \ldots, \lambda_{L,\tau/2,N})$, $\lambda_{j,\tau/2,N} = e^{-\tau \eta_{j, N}}$, and $\eta_{j, N} = \frac{\lambda_{j,N}^{-1} - 1}{\lambda_{0,N}^{-1} - 1}$.}
      \STATE{Compute the $L \times L$ matrices $\bm A_{\tau,N} = \bm \Lambda_{\tau/2,N} \bm{\tilde V}_N \bm \Lambda_{\tau/2,N}$ and $\bm B_{z,N} = z^2 \bm I + \bm V^{(2)}_N - (\bm V_N + \bm V_N^\top)$.}
      \STATE{Solve the matrix generalized eigenvalue problem $\mathbf{A}_{\tau, N} \mathbf{c}_j = \beta_{j, z, \tau, L, N} \mathbf{B}_{z, N} \mathbf{c}_j$ for eigenvectors $\mathbf{c}_j$ and eigenvalues $\beta_{j, z, \tau, L, N}$.}
      \STATE{Compute eigenfrequencies $\omega_{j, z, \tau, L} = b_z(\beta_{j, z, \tau, L})$  with corresponding eigenvectors $\xi_{j,z,\tau,L,N}$ of $V_{z, \tau, L, N}$ by $\xi_{j,z,\tau,L,N} = \sum_{i=1}^L (z \phi_{i,N} - \phi'_{i,N}) c_{ij}.$ Set $\omega_{0, z, \tau, L,N} = 0$ and $\xi_{0,z,\tau,L,N}$ to be constant.}
      \STATE{Define the functions $\varphi'_{i,L',N} := T_{L',N} \phi'_{i,N} = \sum_{j=0}^{L'-1} \langle \phi_{j,N}, \phi'_{i,N}\rangle_N \lambda_{j,N}^{1/2} \varphi_{j,N}$, where $T_{L',N}$.}
      \STATE{Define the RKHS operator $W_{z, \tau, L, N} = \hat T^*_N V_{z, \tau, L, N} \hat T_N$ to have eigenvalues $i\omega_{j, z, \tau, L,N}$ with eigenvectors $\hat\zeta_{j,z,\tau,L,N}$ approximated by $\zeta_{j,z,\tau,L,L',N} := \sum_{i=1}^L (z \sigma_{i,N}\varphi_{i,N} - \varphi'_{i,L',N}) c_{ij}$.}
\end{algorithmic}
\end{algorithm}

\subsection{Generalized eigenvalue problem}

To take advantage of physics-informed computation of function gradients (see~\eqref{eq:v_grad} and~\eqref{eq:vgrad_phi}), we transform the eigenvalue problem \eqref{eq:eigQ} into a generalized eigenvalue problem that replaces the resolvents in the left-hand side by a positive operator that is a quadratic polynomial of $V$ in the right-hand side. First, for every $z \in \rho(V)$, $z - V \colon D(V) \to H$ is a surjective operator, so for every eigenvector $\xi_{j,z,\tau}$ there exists a unique $u_{j,z,\tau} \in D(V)$ such that $\xi_{j,z,\tau} = (z - V) u_{j,z,\tau}$, giving
\begin{equation}
    \label{eq:pre_gev}
    R_z^* V_\tau u_{j,z,\tau} = \beta_{j,z,\tau} (z - V) u_{j,z,\tau}, \quad u_{j,z,\tau} \in D(V).
\end{equation}
Next, supposing that $\beta_{j,\tau,z} \neq 0$ (which corresponds to the solutions of interest), we have $(z - V) u_{j,z,\tau} \in \ran R_z^* = D(z+V)$. Thus, applying $z+V$ to the left- and right-hand sides of~\eqref{eq:pre_gev}, we get
\begin{equation}
    \label{eq:gev_strong}
    V_\tau u_{j,z,\tau} = \beta_{j,z,\tau} (z^2 - V^2) u_{j,z,\tau}, \quad u_{j,z,\tau} \in D(V^2), \quad \beta_{j,z,\tau} \neq 0.
\end{equation}
Note that $z^2 - V^2$ is a strictly positive operator, i.e., $\langle f, (z^2 - V^2) f \rangle > 0$ whenever $f \neq 0$, by skew-adjointness of $V$.

\subsection{Variational formulation}
\label{sec:variational}

Equation~\eqref{eq:gev_strong} defines our generalized eigenvalue problem in strong form. We next transform this problem into weak (variational) form by taking the inner product of both sides by a test function $f \in D(V)$, and using skew-adjointness of $V$ to obtain
\begin{equation}
    \label{eq:pre_varev}
    \langle f, V_\tau u_{j,z\tau} \rangle = \beta_{j,z,\tau}\langle \langle f, (z^2 - V^2) u_{j,z,\tau}\rangle\rangle = \beta_{j,z,\tau} (z^2 \langle f, u_{j,z,\tau}\rangle + \langle V f, V u_{j,z,\tau}\rangle).
\end{equation}

To obtain a well-posed variational eigenvalue problem from~\eqref{eq:pre_varev}, we introduce the Hilbert space $H_V = D(V) \cap \tilde H$, equipped with the inner product $\langle f, g\rangle_V := \langle f, g\rangle + \langle  Vf, Vg\rangle$. Note that $H_V$ is a Hilbert space since $V$ is a closed operator. Moreover, $H_V$ contains $\tilde H^1 := H^1 \cap \tilde H$ (the space of zero-mean functions in $H^1$) as a subspace and embeds continuously into $H$. Considered as an operator from $H_V$ to $H$, the generator $V$ is bounded.

By continuity and compactness of the embeddings $H_V \hookrightarrow H$ and $H \hookrightarrow H^{-1}$, respectively, $H_V$ embeds into $H^{-1}$ compactly, and we have:

\begin{lemma}
    \label{lem:sesqui}
    The sesquilinear forms $A_\tau \colon H_V \times H_V \to \mathbb C$ and $B_z\colon H_V \times H_V \to \mathbb C$ with
    \begin{displaymath}
        A_\tau(f, g) = \langle f, V_\tau g \rangle, \quad B_z(f, g) = z^2 \langle f, g \rangle + \langle  V f, V g\rangle
    \end{displaymath}
    satisfy the upper bounds
    \begin{displaymath}
        \lvert A_\tau(f, g)\rvert \leq C \lVert f \rVert_{H_V} \lVert g \rVert_{H^{-1}}, \quad \lvert B_z(f, g)\rvert \leq (1+z^2) \lVert f\rVert_{H_V} \lVert g \rVert_{H_V}, \quad \forall f,g \in H_V \times H_V,
    \end{displaymath}
    for a constant $C$. In addition, $B_z$ is coercive,
    \begin{displaymath}
        \lvert B_z(f, f)\rvert \geq \min\{1, z^2\} \lVert f \rVert_{H_V}^2, \quad \forall f \in H_V.
    \end{displaymath}
\end{lemma}

\begin{proof}
    Boundedness and coercivity of $B_z$ follow from
    \begin{align*}
        \lvert B_z(f, g)\rvert & \leq z^2 \lVert f\rVert_H \lVert g\rVert_H + \lVert V f\rVert_H \lVert V g\rVert_H  \\
                               &\leq z^2 \lVert f\rVert_{H_V} \lVert g\rVert_{H_V} + \lVert f\rVert_{H_V} \lVert g\rVert_{H_V}  \\
                               &= (1 + z^2)\lVert f\rVert_{H_V} \lVert g\rVert_{H_V}
    \end{align*}
    and
    \begin{align*}
        \lvert B_z(f, f)\rvert &= z^2 (\lVert f\rVert^2_H + z^{-2} \lVert Vf \rVert_H^2) \\
                               & \geq z^2 \min\{1, z^{-2} \} \lVert f \rVert_{H_V}^2 \\
                               & = \min\{z^2, 1\} \lVert  f\rVert_{H_V}^2,
    \end{align*}
    respectively. For the upper bound on $A_\tau$, observe that $V G_{\tau/2}$ is a bounded operator on $H$ satisfying $\lVert  V G_{\tau/2}\rVert \leq \lVert \vec V \cdot \nabla \rVert \lVert p_{\tau/2} \rVert_{C^1(M \times M)}$, where the operator norm $\lVert \vec V \cdot \nabla\rVert$ is taken on $B(C^1(M), C(M))$. Using this fact and boundedness of $G^{-1} G_\tau$ (see \cref{sec:markov}), we compute
    \begin{align*}
        \lvert A_\tau(f,g )\rvert &= \lvert \langle V G_{\tau/2} f, G_{\tau/2 } g\rangle\rvert \\
                                  & \leq \lVert V G_{\tau/2} f \rVert_H \lVert G_{\tau/2}g \rVert_H \\
                                  &\leq  \lVert V G_{\tau/2}\rVert\lVert f\rVert_H \lVert G_{\tau/2} G^{-1}\rVert \lVert G g\rVert_H\\
                                  &= \lVert V G_{\tau/2}\rVert \lVert G_{\tau/2} G^{-1}\rVert  \lVert f\rVert_H \lVert g\rVert_{H^{-1}} \\
                                  &\leq \lVert V G_{\tau/2}\rVert \lVert G_{\tau/2} G^{-1}\rVert  \lVert f\rVert_{H_V}\lVert g\rVert_{H^{-1}}.
    \end{align*}
    which implies the claimed upper bound for $C =  \lVert V G_{\tau/2}\rVert \lVert G_{\tau/2} G^{-1}\rVert$.
\end{proof}

Results on variational eigenvalue problems \cite[e.g.,][section~8]{BabuskaOsborn91}, \cref{lem:sesqui}, and compactness of the embedding $H_V \hookrightarrow H^{-1}$ imply the existence of a compact operator $T_{z,\tau}\colon H_V \to H_V$ such that $\alpha A_\tau(f,u) = B_z(f, u)$ holds for some $u \neq 0$ for all $f \in H_V$ if and only if $\alpha T_{z,\tau} u = u$; i.e., $(\alpha^{-1},u)$ is an eigenpair of $T_{z,\tau}$ whenever $\alpha\neq 0$. Putting together these facts, we deduce that the following is a well-posed variational eigenvalue problem, yielding solutions of~\eqref{eq:gev_strong} in $H_V$.

\begin{definition}[variational eigenvalue problem for $Q_{z,\tau}$]
    \label{def:varev}
    Find $\beta_{j,z,\tau} \in \mathbb C$ and $u_{j,z,\tau} \in H_V \setminus \{ 0 \}$ such that
    \begin{displaymath}
        A_\tau(f, u_{j,z,\tau}) = \beta_{j,z,\tau} B_z(f, u_{j,z,\tau}), \quad \forall f \in H_V.
    \end{displaymath}
\end{definition}

For every solution $(\beta_{j,z,\tau}, u_{j,z,\tau})$ of this problem with $\beta_{j,z,\tau} \neq 0$ we have that
\begin{equation}
    \label{eq:psi_xi}
    \xi_{j,z,\tau} = (z - V) u_{j,z,\tau}
\end{equation}
is an eigenvector of $Q_{z,\tau}$ with corresponding eigenvalue $\beta_{j,z,\tau}$. Conversely, for every eigenvector $\xi_{j,z,\tau}$ of $Q_{z,\tau}$ corresponding to nonzero eigenvalue $\beta_{j,z,\tau}$, $(\beta_{j,z,\tau}, u_{j,z,\tau})$ with $u_{j,z,\tau} = R_z \xi_{j,z,\tau}$ solves the problem in \cref{def:varev}. Furthermore, since $u_{j,z,\tau}$ lies in $\ran G_{\tau/2} \subset H^1$, it has a representative in $\tilde u_{j,z,\tau} \in \mathcal H \subset C^1(M)$ obtained via the Nystr\"om operator from~\eqref{eq:nystrom} as $\tilde u_{j,z,\tau} = \mathcal N u_{j,z,\tau}$. As a result, $\xi_{j,z,\tau}$ has a continuous representative
\begin{displaymath}
    \tilde \xi_{j,z,\tau} = (z - \vec V \cdot \nabla) \tilde u_{j,z,\tau} \in C(M), \quad \iota \tilde \xi_{j,z,\tau} = \xi_{j,z,\tau}.
\end{displaymath}
We obtain the corresponding smooth eigenfunctions $\zeta_{j,z,\tau} \in \mathcal H$ of $W_{z,\tau}$ via~\eqref{eq:zeta_def}.

\subsection{Galerkin scheme}
\label{sec:galerkin}

We compute approximate solutions of the variational eigenvalue problem in \cref{def:varev} using Galerkin methods. To that end, we first establish that the kernel eigenfunctions $\phi_j$ from \cref{sec:markov} provide a dense subspace $\tilde E := \spn \{ \phi_1, \phi_2, \ldots \}$ of $H_V$, in which we will compute approximate solutions. Defining $ E = \spn \{ \phi_0, \phi_1, \ldots \} $ and using $\overline{(\cdot)}$ to denote operator closure, the density of $\tilde E$ in $H_V$ will be a corollary of the following proposition.

\begin{proposition}
    The restriction $V\rvert_E$ of $V$ onto $E$ is a well-defined, essentially skew-adjoint operator; that is, $\overline{V|_E} = V$.
\end{proposition}

\begin{proof}
    Well-definition of $V\rvert_E$ follows from the fact that $\phi_j$ lies in $D(V)$ for all $j \in \mathbb N_0$ (see \cref{sec:markov}). For essential skew-adjointness, recall that a necessary and sufficient condition for a symmetric operator $A\colon D(A) \to H$ to be essentially self-adjoint is that the images $(i \pm A)(D(A))$ are dense in $H$ \cite[e.g.,][Section~VIII.2]{ReedSimon80}. Letting $A = (V|_E)/i$, it follows that $V|_E$ is essentially skew-adjoint iff $(1 \pm V)E$ are dense subspaces of $H$. Suppose that $f$ is an element of $D(V)$ such that $f \in ((1-V)E)^\perp$. Then for all $j \in \mathbb{N}_0$, we have that $\langle (1 - V) \phi_j, f \rangle = 0$, so $\langle \phi_j, (1 + V) f \rangle = 0$. Since $\{\phi_j\}_{j \in \mathbb N_0}$ is an orthonormal basis of $H$, we have $(1 + V) f = 0$ which implies that $f = 0$ because $V$ has no eigenvalues equal to $-1$. This implies that $((1-V)E)^\perp \subseteq D(V)^\perp = \{ 0 \}$. We can then conclude that $(1 - V)E$ is dense in $H$. A similar argument shows that $(1 + V)E$ is also dense in $H$. The proposition then follows.
\end{proof}

\begin{corollary}
    The space $\tilde E$ is a dense subspace of $H_V$.
    \label{cor:dense}
\end{corollary}

\begin{proof}
    Let $G(A) = \{ (f, Af) \in H \times H \, : f \in D(A) \}$ denote the graph of an operator $A\colon D(A) \to H$. The graph $G(A)$ is equipped with the inner product $\langle (f, Af), (g, Ag) \rangle_{G(A)} = \langle f, g \rangle + \langle Af, Ag \rangle$, and it is canonically isomorphic to $D(A)$ equipped with the inner product $\langle f, g \rangle_A =  \langle f, g \rangle + \langle Af, Ag \rangle$. Explicitly, the isomorphism $\epsilon_A \colon D(A) \to G(A)$ is given by $\epsilon_A f = (f, Af)$. Since $V = \overline{V|_E}$, we have that $G(V) = \overline{G(V\rvert_E)}^{\norm{\cdot}_{H \times H}}$. We also have that $G(V) = \epsilon_V( D(V)) $ and $G(V|_E) = \epsilon_{V|_E} E = \epsilon_{V}E$. As such, we get $\epsilon_V(H_V) = \overline{\epsilon_V(E)}^{\norm{\cdot}_{H \times H}}$, which implies $D(V)= \overline{E}^{\norm{\cdot}_{D(V)}}$. Since $H_V = \tilde \Pi(D(V))$ and $\tilde E = \tilde \Pi E$, it follows that $H_V = \overline{\tilde E}^{\norm{\cdot}_{H_V}}$, proving the corollary.
\end{proof}

For $L \in \mathbb N$, consider the approximation spaces $E_L = \spn \{ \phi_1, \ldots, \phi_L \}$ and the orthogonal projections $\Pi_L \colon H \to H$ and $\Pi_{V,L}\colon H_V \to H_V  $ with $\ran \Pi_L = E_L \subset H$ and $\ran \Pi_{V,L} = E_L \subset H_V$. Note that $\Pi_{V,L}$ converges strongly to the identity on $H_V$ as $L\to\infty$ by \cref{cor:dense}. Our Galerkin method solves the following problem.

\begin{definition}[variational eigenvalue problem for $Q_{z,\tau}$; Galerkin approximation]
    \label{def:varev_galerkin}
    Find $\beta_{j,z,\tau,L} \in \mathbb C$ and $u_{j,z,\tau,L} \in E_L \setminus \{ 0 \}$ such that
    \begin{displaymath}
        A_z(f, u_{j,z,\tau,L}) = \beta_{j,z,\tau,L} B_z(f, u_{j,z,\tau, L}), \quad \forall f \in E_L.
    \end{displaymath}
\end{definition}

For every solution $(\beta_{j,z,\tau,L}, u_{j,z,\tau,L})$ we define $\xi_{j,z,\tau,L} \in H$ as
\begin{displaymath}
    \xi_{j,z,\tau,L} = (z - V) u_{j,z,\tau,L},
\end{displaymath}
the continuous representative $\tilde\xi_{j,z,\tau,L} \in C(M)$ as
\begin{equation}
    \label{eq:xi_nyst}
    \tilde \xi_{j,z,\tau,L} = (z - \vec V \cdot \nabla) \tilde u_{j,z,\tau,L}, \quad \tilde u_{j,z,\tau,L} = \mathcal N u_{j,z,\tau,L},
\end{equation}
and the RKHS function $\zeta_{j,z,\tau,L} \in \mathcal H$ as
\begin{displaymath}
    \zeta_{j,z,\tau,L} = T \xi_{j,z,\tau,L},
\end{displaymath}
where $\iota \tilde \xi_{j,z,\tau,L} = \xi_{j,z,\tau,L}$. One can directly verify that any two eigenvectors $\xi_{j,z,\tau,L}$ (resp.\ $\zeta_{j,z,\tau,L}$) corresponding to distinct eigenvalues $\beta_{j,z,\tau,L}$ are orthogonal in $H$ (resp.\ $\mathcal H$), and any linearly independent set corresponding to the same eigenvalue can be chosen to be orthogonal. Moreover, when $L$ is even, the eigenvalues/eigenvectors come in complex-conjugate pairs, similarly to the eigenvalues/eigenvectors of the infinite-rank operator $Q_{z,\tau}$. Henceforth, we shall assume that $L$ is chosen even.

We have $\lim_{L\to\infty} \lVert \Pi_{V,L} f - f\rVert_{H_V} = 0$ for all $f \in H_V$ by \cref{cor:dense}. Moreover, by boundedness and coercivity of $B_z$ on $H_V \times H_V$ (\cref{lem:sesqui}), for every $f \in H_V \setminus \{ 0 \}$ there exists $L_* \in \mathbb N$ and $c>0$ such that $B_z(f, f) \geq c$ holds for all $ L > L_*$. By results on Galerkin approximation of variational eigenvalue problems \cite[e.g.,][]{BabuskaOsborn91} these facts imply that, as $L\to \infty$, every nonzero $\beta_{j,z,\tau,L}$ converges to $\beta_{j,z,\tau}$ and that the corresponding eigenspaces --- $\mathbb E^{(V)}_{j,z,\tau}:= \spn \{ u_{i,z,\tau} : \beta_{i,z,\tau} = \beta_{j,z,\tau} \}$ and $\mathbb E^{(V)}_{j,z,\tau,L} := \spn \{ u_{i,z,\tau,L} : \beta_{i,z,\tau,L} = \beta_{j,z,\tau,L} \}$ --- converge in the sense of strong convergence of orthogonal projections. That is, with $\mathcal U_{j,z,\tau} \colon H_V \to H_V$ and $\mathcal U_{j,z,\tau,L} \colon H_V \to H_V$ denoting the orthogonal projections onto $\mathbb E^{(V)}_{j,z,\tau}$ and $\mathbb E^{(V)}_{j,z,\tau,L}$, we have
\begin{displaymath}
    \lim_{L\to\infty} \lVert \mathcal U_{j,z,\tau, L} f - \mathcal U_{j,z,\tau} f \rVert_{H_V} = 0, \quad \forall f \in H_V.
\end{displaymath}
Since $V$ is bounded as linear map from $H_V$ to $H$, it follows additionally that $\lim_{L\to\infty} \lVert (z- V) (\mathcal U_{j,z,\tau,L} - \mathcal U_{j,z,\tau,}) f \rVert_H = 0$ for every $f \in H_V$. This implies in turn that the orthogonal projections $\Xi_{j,z,\tau,L} \colon H \to H$ onto the subspaces $\mathbb E_{j,z,\tau,L} := \spn \{ \xi_{i,z,\tau,L} : \beta_{i,z,\tau,L} = \beta_{j,z,\tau,L} \} \subset H$ converge strongly to the orthogonal projections $\Xi_{j,z,\tau}\colon H \to H$ onto $\mathbb E_{j,z,\tau} := \spn \{ \xi_{i,z,\tau} : \beta_{i,z,\tau} = \beta_{j,z,\tau} \} \subset H$, i.e.,
\begin{displaymath}
    \lim_{L\to\infty} \lVert \Xi_{j,z,\tau} f - \Xi_{j,z,\tau,L} f\rVert_H =0, \quad \forall f \in H.
\end{displaymath}
Defining the frequencies $\omega_{j,z,\tau,L} = b_z(\beta_{j,z,\tau,L}) / i$ and the skew-adjoint, finite-rank operators $Q_{z,\tau,L} \colon \tilde H \to \tilde H$, $\tilde V_{z,\tau,L} = b_z(Q_{z,\tau,L}) $, and $V_{z,\tau,L} \colon H \to H$, where
\begin{displaymath}
    Q_{z,\tau,L} = \sum_{j=1}^L \beta_{j,z,\tau,L} \xi_{j,z,\tau,L}\langle \xi_{j,z,\tau,L, \cdot}\rangle, \quad \tilde V_{z,\tau,L} = \sum_{j=1}^L i \omega_{j,z,\tau,L} \xi_{j,z,\tau,L}\langle \xi_{j,z,\tau,L, \cdot}\rangle,
\end{displaymath}
and $V_{z,\tau,L} = \tilde \Pi \tilde V_{z,\tau,L} \tilde \Pi$, we have the following proposition:

\begin{proposition}
    As $L\to\infty$, $Q_{z,\tau,L}$ converges to $Q_{z,\tau}$ strongly, and $V_{z,\tau,L}$ converges to $V_{z,\tau}$ in strong resolvent sense.
\end{proposition}

\begin{proof}
    The strong convergence of $Q_{z,\tau,L}$ to $Q_{z,\tau}$ is a direct consequence of the spectral convergence $\beta_{j,z,\tau,L} \to \beta_{j,\tau}$ and $\Xi_{j,z,\tau,L} \sto \Xi_{j,z,\tau}$.
    Next, observe that, for $z>0$, the resolvent function $r_z\colon i \mathbb R \to \mathbb C$, $r_z(i\omega) = z - (i\omega)^{-1}$ is a bounded continuous function that maps the imaginary line to the circle in the complex plane with radius $(2z)^{-1}$ and center $(2z)^{-1}$ \cite[see, e.g.,][]{GiannakisValva24}. Since $\tilde R_{z,\tau} = r_z \circ b_z( Q_z)$ and $\tilde R_{z,\tau,L} = r_z \circ b_z (Q_{z,\tau})$ give the resolvents of $\tilde V_{z,\tau}$ and $\tilde V_{z,\tau,L}$, respectively, the strong convergence $R_{z,\tau,L} \sto R_{z,\tau}$ follows from strong convergence of $Q_{z,\tau,L}$ to $Q_{z,\tau}$ and boundedness and continuity of $r_z \circ b_z$. The strong convergence of the resolvent $R_{z,\tau,L}$ of $V_{z,\tau,L}$ to $R_{z,\tau}$ follows from $R_{z,\tau,L} \sto R_{z,\tau}$ and the facts that $R_{z,\tau,L} = \tilde R_{z,\tau,L} + z^{-1} \Pi_0$ and $R_{z,\tau} = \tilde R_{z,\tau} + z^{-1} \Pi_0$.
\end{proof}

Similarly to the infinite-rank operator $V_{z,\tau}$, $V_{z,\tau,L}$ is unitarily equivalent to a skew-adjoint operator $W_{z,\tau,L} = T^* V_{z,\tau,L} T$ acting on the RKHS $\mathcal H$, where
\begin{displaymath}
    W_{z,\tau,L} = \sum_{j=1}^L i \omega_{j,z,\tau,L} \zeta_{j,z,\tau,L}\langle \zeta_{j,z,\tau,L, \cdot}\rangle_{\mathcal H}.
\end{displaymath}
This operator has eigenvectors $\zeta_{j,z,\tau,L} \in \mathcal H$, and can be thought of as generating a smooth approximation of the evolution in $H$ generated by $V_{z,\tau,L}$; see \cref{sec:scheme}.

Next, to examine the above Galerkin approximation from a computational standpoint, define the $L \times L$ generator matrix $\bm V = [V_{ij}]_{i,j=1}^L$ with entries $V_{ij} = \langle \phi_i, V \phi_j\rangle = -V_{ji}$. Note that for the kernel construction described in \cref{sec:markov}, we can use~\eqref{eq:vgrad_phi} to express $V_{ij}$ using the singular vectors of the integral operator $\tilde K$ and the directional derivative kernel $\hat k'$, i.e.,
\begin{equation}
    \label{eq:v_ij}
    V_{ij} = \langle \phi_i, \tilde K' \gamma_j \rangle \equiv \int_X \int_X \phi_i(x) (\vec V_x \cdot \hat k(x, z)) \gamma_j(z) \, d\mu(x) \, d\mu(z).
\end{equation}
In particular, under our assumption of known dynamical vector field $\vec V$, the integrand in the right-hand side can be evaluated in a physics-informed manner by computing derivatives of the kernel function $\hat k$. In the experiments of \cref{sec:examples}, we perform these computations using automatic differentiation. Defining $\bm V^{(2)} = [V^{(2)}_{ij}]_{i,j=1}^L$ as the $L \times L$ symmetric positive-definite matrix with elements $V^{(2)}_{ij} = \langle V\phi_i, V\phi_j\rangle$, we have
\begin{equation}
    \label{eq:v2_ij}
    V^{(2)}_{ij} = \langle \tilde K' \gamma_i, \tilde K' \gamma_j\rangle \equiv \int_X \int_X \int_X \gamma_i(y) (\vec V_x \cdot \hat k(x, y)) (\vec V_x \cdot \hat k(x, z))\gamma_j(z) \, d\mu(x) \, d\mu(y) \, d\mu(z).
\end{equation}
In what follows, $\bm \Lambda_\tau$ will be the $L \times L$ diagonal matrix equal to $\diag (\lambda_{1,\tau}, \ldots, \lambda_{L,\tau})$, and $\bm I$ will be the $L \times L$ identity matrix.

Expanding $u_{j,z,\tau,L} = \sum_{i=1}^L c_{ij} \phi_i$, $c_{ij} \in \mathbb C$, and forming the column vector $\bm c_j = (c_{1j}, \ldots, c_{Lj})^\top \in \mathbb C^L$ of expansion coefficients, solutions of the variational problem in \cref{def:varev_galerkin} can be equivalently obtained by solving the matrix generalized eigenvalue problem
\begin{equation}
    \label{eq:matgev_galerkin}
    \bm A_\tau \bm c_j = \beta_{j,z,\tau,L} \bm B_z \bm c_j,
\end{equation}
where $\bm A_\tau := [A_\tau(\phi_i, \phi_j)]_{i,j=1}^L$ and $\bm B_z := [B_z(\phi_i, \phi_j)]_{i,j=1}^L$ are $L\times L$ matrix representations of the sesquilinear forms $A_\tau$ and $B_z$, respectively, on $E_L$. Explicitly, we have
\begin{displaymath}
    \bm A_\tau = \bm \Lambda_{\tau/2} \bm V \bm \Lambda_{\tau/2}, \quad \bm B_z = z^2 \bm I + \bm V^{(2)}
\end{displaymath}
Moreover, the corresponding eigenvectors $\xi_{j,z,\tau,L}$ and $\zeta_{j,z,\tau,L}$ of $V_{z,\tau,L}$ and $W_{z,\tau,L}$, respectively, can be obtained from
\begin{equation}
    \label{eq:xi_deriv_matvec}
    \xi_{j,z,\tau,L} = \sum_{i=1}^L (z \phi_i - \phi'_i) c_{ij}, \quad \zeta_{j,z,\tau,L} = \sum_{i=1}^L (z \sigma_i\varphi_i - T\phi'_i) c_{ij},
\end{equation}
where the directional derivatives $\phi'_i \in H$ are computed again via automatic differentiation using~\eqref{eq:vgrad_phi} and the RKHS functions $T \phi'_i$ are given by~\eqref{eq:rkhs_unitary}. Note that, in general, computation of $T \phi_i'$ via~\eqref{eq:rkhs_unitary} requires evaluation of an infinite linear combination of RKHS basis functions $\psi_i$. In practical applications, we can use the partial isometries $T_{L'}$ from~\eqref{eq:rkhs_partial_iso} to truncate this sum to a finite number of terms $L' \in \mathbb N$. This leads to the approximation
\begin{equation}
    \label{eq:zeta_partial_isometry}
    {\zeta_{j,z,\tau,L,L'} := \sum_{i=1}^L (z \sigma_i\varphi_i - \varphi'_{i,L'}) c_{ij}, \quad \varphi'_{i,L'} := T_{L'} \phi'_i},
\end{equation}
which converges to $\zeta_{j,z,\tau,L}$ as $L' \to \infty$ in the norm of $\mathcal H$. For completeness, we note the formula $\tilde\xi_{j,z,\tau,L} = \sum_{i=1}^L (z \varphi_i - \varphi'_i) c_{ij} \in C(M)$ for the continuous representative $\tilde \xi_{j,z,\tau,L}$ of $\xi_{j,z,\tau,L}$ from~\eqref{eq:xi_nyst}.

\subsection{Data-driven approximation}
\label{sec:data_driven}

Practical implementation of the Galerkin scheme described in~\cref{sec:galerkin} relies on evaluation of the integrals for the generator matrix elements $V_{ij}$ in~\eqref{eq:v_ij}. Even if the dynamical vector field $\vec V$ is known, these integrals cannot typically be evaluated in closed form, e.g., due to lack of explicit knowledge of the kernel eigenbasis and/or the invariant measure. As a result, one must oftentimes resort to numerical approximation of these integrals. Our main focus is on data-driven quadrature schemes, where integrals with respect to the invariant measure $\mu$ are approximated by ergodic time averaging along numerical trajectories.

Let $x_0, x_1, \ldots$ be a sequence of points in $M$ that is equidistributed with respect to $\mu$. This means that the sampling measures $\mu_N := \sum_{n=0}^{N-1} \delta_{x_n} / N$ converge to the invariant measure in the sense of weak-$^*$ convergence of Borel measures on the compact manifold $M$,
\begin{equation}
    \label{eq:sampling_conv}
    \lim_{N\to\infty} \int_M f \, d\mu_N = \int_M f \, d\mu, \quad \forall f \in C(M).
\end{equation}
Under the ergodicity assumption in~\cref{sec:dynamical_system}, $\mu$-a.e.\ initial condition $x_0$ and Lebesgue-a.e.\ sampling interval $\Delta t>0$ will yield a sequence $x_n = \Phi^{n\, \Delta t} (x_0)$ satisfying~\eqref{eq:sampling_conv}. Moreover, for systems with so-called observable ergodic invariant measures \cite[e.g.,][]{Blank17}, $x_0$ may be drawn from a subset of $M$ of positive ambient measure (e.g., volume measure induced from a Riemannian metric). While little can be said about rates of convergence of $\int_M f \, d\mu_N$ to $\int_M f\, d \mu$ for arbitrary dynamical systems and observables, there are examples (oftentimes involving hyperbolic dynamics) exhibiting central limit theorems for sufficiently smooth observables \cite{Baladi00}, with corresponding $O(N^{-1/2})$ rates. An analysis of convergence rates is beyond the scope of this paper. However, it is worthwhile noting that the fact that our scheme does not require time-ordered snapshots (by virtue of using known equations of motion) provides flexibility in the strategy employed to generate the samples $x_n$. For example, grid-based quadrature or Markov chain Monte Carlo would also be suitable, so long as the matrices involved in the variational eigenvalue problem in \cref{def:varev_galerkin} can be consistently approximated as $N \to \infty$.

For a given sampling measure $\mu_N$, our data-driven approximation scheme is built by replacing the infinite-dimensional Hilbert space $H= L^2(\mu)$ by the finite-dimensional space $\hat H_N := L^2(\mu_N)$ with inner product $\langle f, g \rangle_N := \int_M f^* g \, d \mu_N \equiv \sum_{j=0}^{N-1} f^*(x_n) g(x_n)/N$. For simplicity of exposition, we will tacitly assume that the points $x_n$ are all distinct so that $\hat H_N$ has dimension $N$. The finite set $X_N = \{ x_0, \ldots, x_{N-1}\} \subset M$ will then denote the support of $\mu_N$.

We construct data-driven versions of the kernel $p\colon M \times M \to \mathbb R_+$, the integral operator $G\colon H \to H$, the SVD $(\phi_j \in H, \sigma_j \in \mathbb R_+, \gamma_j \in H)$ of the kernel integral operator $\tilde K\colon H \to H$, the eigenvalues $\lambda_j, \lambda_{j,\tau} $, and the continuously differentiable representatives $\varphi_j \in C^1(M)$ of $\phi_j$ analogously to the procedure described in \cref{sec:markov}, replacing throughout the invariant measure $\mu$ by $\mu_N$, and the initial kernel $k$ is chosen to have variable bandwidth (see \ref{app:vbkernel}). The resulting data-driven objects will be denoted using $N$ subscripts, with additional $\hat{(\cdot)}$ accents when needed to avoid confusion with symbols already introduced in previous sections, e.g., $p_N \colon M \times M \to \mathbb R$, $\hat G_N \colon \hat H_N \to \hat H_N$, $\phi_{j,N} \in \hat H_N$, $\sigma_{j,N} \in \mathbb R_+$, $\gamma_{j,N} \in \hat H_N$, $\hat\lambda_{j,N}, \lambda_{j,\tau,N} \in \mathbb R_+$, and $\varphi_{j,N} \in C^1(M)$. Moreover, $\iota_N \colon C(M) \to \hat H_N$ will denote the restriction map from continuous functions on $M$ to equivalence classes of discretely sampled functions in $\hat H_N$. Note that the functions $\varphi_{j,N}$ are continuous representatives of $\phi_{j,N}$ with respect to the sampling measure $\mu_N$. That is, we have $\varphi_{j,N}(x_n) = \phi_{j,N}(x_n)$ for all $ n \in \{ 1, \ldots, N-1 \}$ (equivalently, $\iota_N \varphi_{j,N} = \phi_{j,N}$).

Further discussion of these constructions along with convergence results in the large-data limit, $N\to \infty$, can be found in DGV. In brief, using results on spectral approximation of kernel integral operators \cite[e.g.,][]{VonLuxburgEtAl08}, it can be shown that every (strictly positive) eigenvalue $\lambda_j$ of $G$ can be approximated by eigenvalues of $\hat G_N$,
\begin{equation}
    \label{eq:lambda_conv}
    \lim_{N\to\infty} \lambda_{j,N} = \lim_{N\to\infty}\lambda_j,
\end{equation}
and for every corresponding eigenfunction $\phi_j \in H$ with continuous representative $\varphi_j$ there exists a sequence of eigenfunctions $\phi_{j,N} \in \hat H_N$ whose continuous representatives $\varphi_{j,N}$ converge to $\varphi_j$ in $C(M)$ norm.

In what follows, $\mathcal H_N \subset C^1(M)$ will be the RKHS with reproducing kernel $p_N$, and $\hat T_N \colon \hat H_N \to \mathcal H_N$, $\mathcal N_N \colon \hat H_N \to \mathcal H_N$, and $\mathcal E_N \colon \hat H_N \to \mathbb R_+$ the corresponding isometric embedding, Nystr\"om operator and Dirichlet energy, respectively. Note that $\mathcal N_N$ and $ \mathcal E_N$ are defined on the entire Hilbert space $\hat H_N$ by finite dimensionality of this space (unlike their densely defined counterparts $\mathcal N$ and $\mathcal E$, respectively, on the infinite-dimensional Hilbert space $H$). We also let $\psi_{j,N}$ be orthonormal basis vectors of $\mathcal H_N(X_N)$, defined analogously to $\psi_j$ in the case of $\mathcal H(X)$. For the purposes of the physics-informed scheme studied in this paper (which requires evaluation of directional derivatives $\vec V \cdot \nabla$ along the dynamical vector field), we augment \crefrange{prty:K1}{prty:K5} required of the kernels $p_\tau$ by an additional assumption on $C^1$-norm convergence of eigenfunctions.

\begin{enumerate}[label=(K\arabic*)]
    \setcounter{enumi}{5}
    \item \label[prty]{prty:K6} For every $j \in \mathbb N_0$ (and sufficiently large $N$) the eigenfunctions $\phi_{j,N} \in \hat H_N$ of $G_{\tau,N}$ are chosen such that their corresponding representatives $\varphi_{j,N} \in C^1(M)$ have a $C^1(M)$ limit $\varphi_j$,
        \begin{displaymath}
            \lim_{N\to\infty}\lVert \varphi_{j,N} - \varphi_j \rVert_{C^1(M)} = 0,
        \end{displaymath}
        where $\varphi_j$ is the $C^1(M)$ representative of the eigenfunction $\phi_j \in H$ from~\eqref{eq:varphi}.
    % \item \label[prty]{prty:K6} For every $j \in \mathbb N_0$ (and sufficiently large $N$) the eigenfunctions $\phi_{j,N} \in \hat H_N$ of $G_{\tau,N}$ are chosen such that their corresponding representatives $\varphi_{j,N} \in C^1(M)$ have a $C^1(M)$ limit $\varphi_j$ as $N \to \infty$, where $\varphi_j$ is the $C^1(M)$ representative of the eigenfunction $\phi_j \in H$ of $G_\tau$ corresponding to eigenvalue $\lambda_j$.
\end{enumerate}

We then have:

\begin{proposition}
    \label{prop:c1}
For the kernel construction described in \cref{sec:markov}, every eigenfunction $\phi_j$ of $G_\tau$ has a corresponding sequence of eigenfunctions $\phi_{j,N}$ of $G_{\tau,N}$ satisfying \cref{prty:K6}.
\end{proposition}

\begin{proof}
    See \ref{app:proof_prop_c1}.
\end{proof}

\begin{remark}
    The reason that \cref{prty:K6} requires a choice of eigenfunctions $\varphi_{j,N}$ is that eigenvectors of linear operators are arbitrary up to multiplication by nonzero scalars, or up to a general invertible linear combination for multidimensional eigenspaces. An equivalent way of formulating the assumption without invoking a choice of eigenfunctions (but at the expense of introducing new notation) would be to require strong convergence of spectral projections onto the corresponding eigenspaces.
\end{remark}

With these assumptions and results in place, we consider the following approximation $ \hat V_N \colon \hat H_N \to C(M)$ of the dynamical vector field $\vec V \cdot \nabla \colon C^1(M) \to C(M)$,
\begin{displaymath}
    \hat V_N = (\vec V \cdot \nabla) \circ \mathcal N_N,
\end{displaymath}
and the approximation $V_N\colon \hat H_N \to \hat H_N $ of the generator $V$ given by $V_N = \iota_N \circ \hat V_N$. Unlike the true generator $V$, $V_N$ is in general not antisymmetric, so we also employ an antisymmetric approximation, $\tilde V_N = (V_N - V_N^*)/2$, where the adjoint is taken on $\hat H_N$. \Cref{fig:phi} shows plots of the $C^1(M)$ representative $\varphi_{j,N} = \mathcal N_N \phi_{j,N}$ and the directional derivative $\varphi'_{j,N} = \vec V \cdot \nabla \varphi_{j,N} \equiv \hat V_N \phi_{j,N}$ for a kernel eigenfunction $\phi_{j,N}$ computed from a dataset sampled near the Lorenz attractor in $\mathbb R^3$ (to be used in the numerical experiments of \cref{sec:examples}).

\begin{figure}
    \includegraphics[width=\linewidth,draft=false]{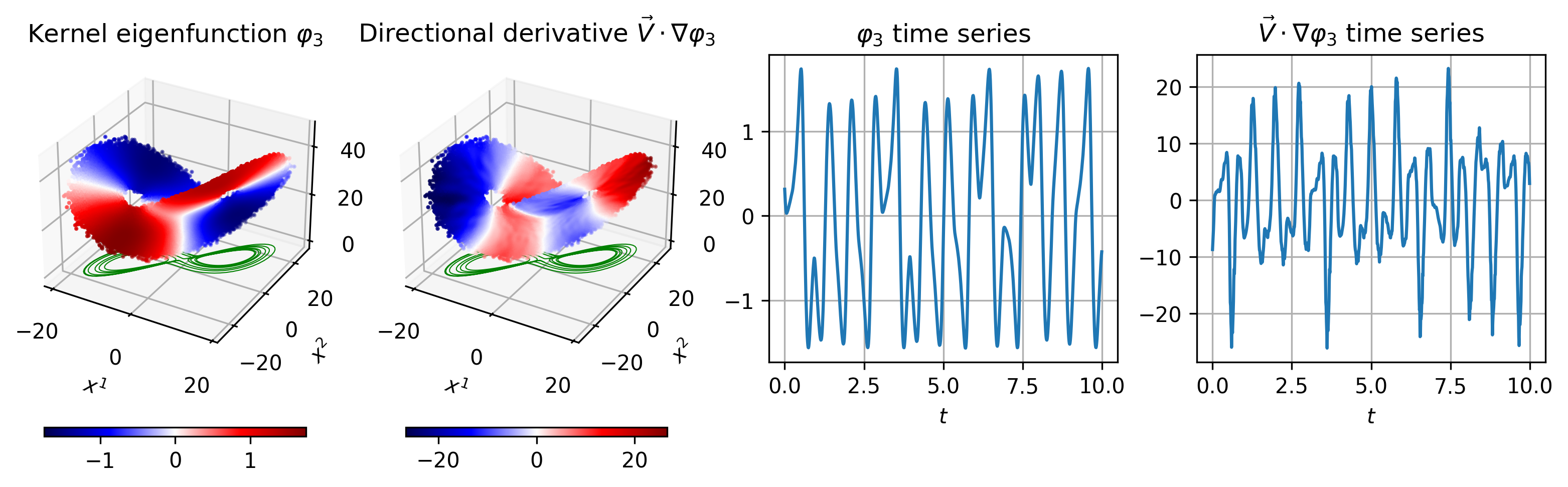}
    \includegraphics[width=\linewidth,draft=false]{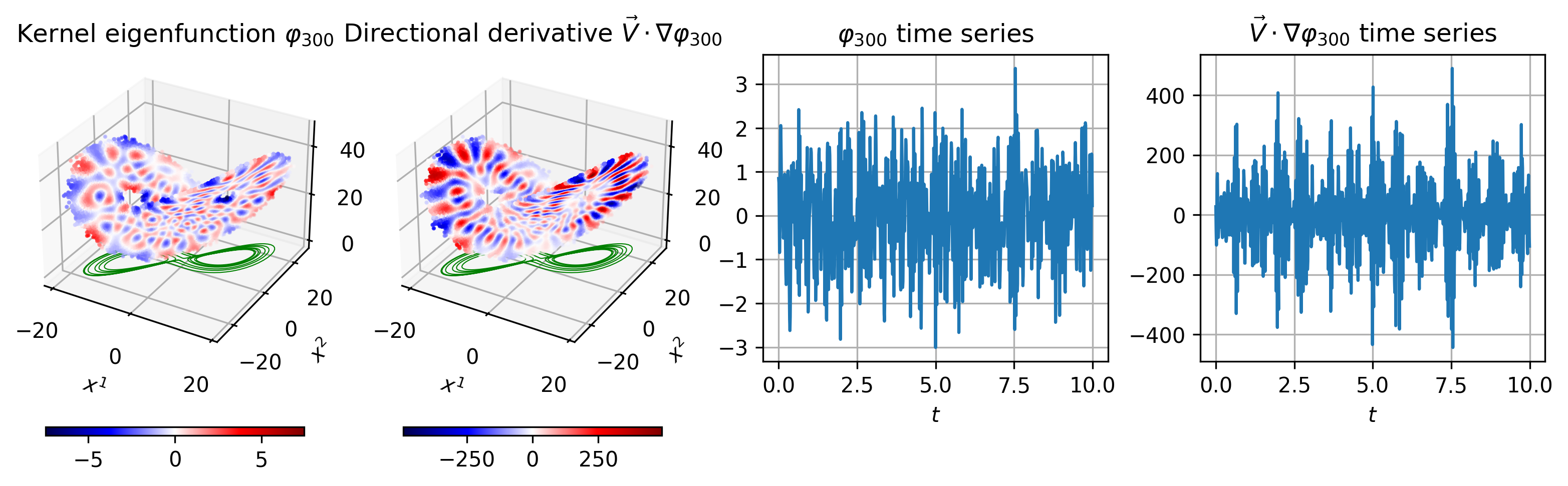}
    \caption{$C^1(M)$ representatives $\varphi_{j,N}$ of eigenfunctions $\phi_{j,N}$ (first and third column from the left), and the directional derivatives $ \varphi_{j,N}' = \vec V \cdot \nabla \varphi_{j,N} \in C(M)$ (second and fourth column) with respect to the dynamical vector field $\vec V$ of the L63 system (see~\eqref{eq:v_l63}) for $j=3$ (top row) and $j=300$ (bottom row). The eigenfunctions were computed using the training data employed in the L63 experiments in \cref{sec:examples}. The scatterplots in the first and second columns show values of $\varphi_{j,N}$ and $\varphi'_{j,N}$, respectively, on the training dataset. The time series in the third and fourth columns show the corresponding values sampled along a dynamical trajectory that is distinct from the training dataset (see again \cref{sec:examples}). The $(x^1, x^2)$ coordinates of the test trajectory are plotted as a green line in the $x^3 = 0$ plane of the three-dimensional plot domain in the first and third panels for reference. Observe that as $j$ increases $\varphi_{j,N}$ exhibits smaller-scale variability on the Lorenz attractor. The use of the variable-bandwidth kernel \eqref{eq:k_vb} was important for obtaining a well-behaved data-driven basis $\varphi_{j,N}$ at large $j$; see \cref{fig:phi_nb} for a comparison.}
    \label{fig:phi}
\end{figure}

Next, as a data-driven analog of the space $H_V$ from \cref{sec:galerkin}, we use $H_{V,N}$ --- this space is equal to $\hat H_N$ as a vector space but equipped with the inner product $\langle f, g \rangle_{V_N} := \langle f, g\rangle_N + \langle V_N f, V_N g\rangle_N$. Moreover, as analogs of the subspaces $\tilde E$ and $E_1, E_2, \ldots$, respectively, we define $\tilde E_N = \spn\{\phi_{1,N}, \ldots, \phi_{N-1,N}\}$ and $ E_{L,N} = \spn \{ \phi_{1,N}, \ldots, \phi_{L,N} \}$ for $L \in \{ 1,\ldots, N \}$. Note that since $\phi_{0,N}$ can be chosen as the constant vector $\bm 1_N := \iota_N \bm 1$ (by Markovianity of $\hat G_N$), we have that $\tilde E_N$ is equal as a vector space to $\tilde H_N := \{ f \in \hat H_N: \langle \bm 1_N, f \rangle_N =0 \}$; i.e., the space of zero-mean functions with respect to the sampling measure $\mu_N$. In what follows, $\tilde \Pi_N$ and $\Pi_{L,N}$ will be the orthogonal projections on $\hat H_N$ mapping into $\tilde H_N$ and $E_{L,N}$ respectively.

Using the approximate generator $V_N$ and its antisymmetric version $\tilde V_N$, we define sesquilinear forms $A_{\tau,N}\colon H_{V,N} \times H_{V,N} \to \mathbb C$ and $B_{z,N} \colon H_{V,N} \times H_{V,N} \to \mathbb C $ that will serve as data-driven analogs of $A_\tau$ and $B_z$ from \cref{lem:sesqui},
\begin{displaymath}
    A_{\tau,N}(f, g) = \langle f, V_{\tau,N} g \rangle_N, \quad B_{z,N}(f,g) = z^2 \langle (z - V_N)f, (z - V_N) g\rangle_N,
\end{displaymath}
where $V_{\tau,N} = G_{\tau/2,N} \tilde V_N G_{\tau/2, N}$. With these definitions, the data-driven version of the Galerkin-approximated variational eigenvalue problem from \cref{def:varev_galerkin} is as follows.

\begin{definition}[variational eigenvalue problem for $Q_{z,\tau}$; data-driven approximation]
    \label{def:varev_datadriven}
    Find $\beta_{j,z,\tau,L,N} \in \mathbb C$ and $u_{j,z,\tau,L,N} \in E_{L,N} \setminus \{ 0 \}$ such that
    \begin{displaymath}
        A_{\tau,N}(f, u_{j,z,\tau,L,N}) = \beta_{j,z,\tau,L} B_{z,N}(f, u_{j,z,\tau, L, N}), \quad \forall f \in E_{L,N}.
    \end{displaymath}
\end{definition}

Similarly to \cref{sec:galerkin}, for every solution $(\beta_{j,z,\tau,L,N}, u_{j,z,\tau,L,N})$ we compute $\xi_{j,z,\tau,L,N} = (z - V_N) u_{j,z,\tau,L,N} \in E_{L,N}$, the continuous representative $\tilde \xi_{j,z,\tau,L,N} = \mathcal N_N \xi_{j,z,\tau,L,N} \in C(M)$, the RKHS eigenfunction $\hat\zeta_{j,z,\tau,L,N} = \hat T_N \xi_{j,z,\tau,L,N} \in \mathcal H_N$, and the corresponding eigenfrequency $\omega_{j,z,\tau,L,N} = b_z(\beta_{j,z,\tau,L,N})/ i \in \mathbb R$. The vectors $\xi_{j,z,\tau,L,N}$ can again be chosen to form an orthonormal basis of $E_{L,N}$ with respect to the $\langle \cdot, \cdot\rangle_N$ inner product. Together with the corresponding eigenvalues and eigenfrequencies, they reconstruct skew-adjoint operators $Q_{z,\tau,L,N}\colon \tilde H_N \to \tilde H_N$, $\tilde V_{z,\tau,L,N} = b_z(Q_{z,\tau,L,N})$, and $V_{z,\tau,L,N} \colon \hat H_N \to \hat H_N$, where
\begin{align*}
    Q_{z,\tau,L,N} &= \sum_{j=1}^L \beta_{j,z,\tau,L,N} \xi_{j,z,\tau,L,N}\langle \xi_{j,z,\tau,L, \cdot}\rangle_N, \\
    \tilde V_{z,\tau,L,N} &= \sum_{j=1}^L i \omega_{j,z,\tau,L,N} \xi_{j,z,\tau,L,N}\langle \xi_{j,z,\tau,L, \cdot}\rangle_N, \\
    V_{z,\tau,L,N} &= \tilde \Pi_N \tilde V_{z,\tau,L,N} \tilde \Pi_N.
\end{align*}
Moreover, $V_{z,\tau,L,N}$ is unitarily equivalent to a skew-adjoint operator $W_{z,\tau,L,N} = T^*_N V_{z,\tau,L,N} T_N$ acting on the RKHS $\mathcal H_N$, where
\begin{displaymath}
    W_{z,\tau,L,N} = \sum_{j=1}^L i \omega_{j,z,\tau,L,N} \zeta_{j,z,\tau,L,N}\langle \zeta_{j,z,\tau,L,N}, \cdot\rangle_{\mathcal H_N}.
\end{displaymath}

Computationally, solving the variational eigenvalue problem in \cref{def:varev_datadriven} is equivalent to solving the generalized eigenvalue problem
\begin{equation}
    \label{eq:matgev_datadriven}
    \bm A_{\tau,N} \bm c_j = \beta_{j,z,\tau,L,N} \bm B_{z,N} \bm c_j,
\end{equation}
where $\bm A_{\tau,N} := [A_{\tau,N}(\phi_{i,N}, \phi_{j,N})]_{i,j=1}^L$ and $\bm B_{z,N} := [B_{z,N}(\phi_{i,N}, \phi_{j,N})]_{i,j=1}^L$ are $L\times L$ matrices representing $A_{\tau,N}$ and $B_{z,N}$, respectively, and $\bm c_j = (c_{1j}, \ldots, c_{Lj})^\top \in \mathbb C^L$ is a column vector containing the expansion coefficients $c_{ij}$ of $u_{j,z,\tau,L,N}$ in the $\phi_{i,N}$ basis of $E_{L,N}$. Aside from differences owing to non-antisymmetry of $V_N$, the matrices $\bm A_{\tau,N}$ and $\bm B_{z,N}$ can be computed similarly to $\bm A_\tau$ and $\bm B_z$, respectively, in the eigenvalue problem~\eqref{eq:matgev_datadriven}. Specifically, we have
\begin{displaymath}
    \bm A_{\tau,N} = \bm \Lambda_{\tau/2,N} \bm{\tilde V}_N \bm \Lambda_{\tau/2,N}, \quad \bm B_{z,N} = z^2 \bm I + \bm V^{(2)}_N - (\bm V_N + \bm V_N^\top),
\end{displaymath}
where $\bm \Lambda_{\tau/2,N} = \diag(\lambda_{1,\tau/2,N}, \ldots, \lambda_{L,\tau/2,N})$, $\bm{\tilde V}_N = (\bm V_N - \bm V_N^\top)/2$, $\bm V_N = [V_{ij,N}]_{i,j=1}^L$, and $\bm V_N^{(2)} = [V_{ij,N}^{(2)}]_{i,j=1}^L$ with $V_{ij,N} = \langle \phi_{i,N}, V_N \phi_{j,N} \rangle_N$ and $V_{ij,N}^{(2)} = \langle V_N \phi_{i,N}, V_N \phi_{j,N} \rangle_N$. \Cref{fig:mat} shows heat maps of these matrices obtained from the same L63 dataset used in the eigenfunction plots of \cref{fig:phi}. Similarly to~\eqref{eq:xi_deriv_matvec}, we form the eigenvectors $\xi_{j,z,\tau,L,N} \in \hat H_N$ and $\hat \zeta_{j,z,\tau,L,N} \in \mathcal H_N$ from the linear combinations
\begin{displaymath}
    \xi_{j,z,\tau,L,N} = \sum_{i=1}^L (z \phi_{i,N} - \phi'_{i,N}) c_{ij}, \quad \hat\zeta_{j,z,\tau,L,N} = \sum_{i=1}^L \sigma_{i,N} (z \varphi_{i,N} - \hat T_N \phi'_{i,N}) c_{ij}.
\end{displaymath}
Moreover, similarly to $\zeta_{j,z,\tau,L,L'}$ from~\eqref{eq:zeta_partial_isometry} we approximate $\hat \zeta_{j,z,\tau,L,N}$ by $\zeta_{j,z,\tau,L,L',N} \in \mathcal H_N$ using the partial isometries $T_{N,L'} \colon \hat H_N \to \mathcal H_N$ with $L' \leq N - 1$,
\begin{displaymath}
    \zeta_{j,z,\tau,L,L',N} := \sum_{i=1}^L (z \sigma_{i,N}\varphi_{i,N} - \varphi'_{i,L',N}) c_{ij}, \quad \varphi'_{i,L',N} := T_{L',N} \phi'_{i,N}.
\end{displaymath}
We also note that the vectors $\xi_{j,z,\tau,L,N} \in \hat H_N$ have continuous representatives $\tilde\xi_{j,z,\tau,L,N} \in C(M)$ (cf.\ $\tilde \xi_{j,z,\tau,L}$ in~\eqref{eq:xi_nyst}), where
\begin{equation}
    \label{eq:xi_nyst_datadriven}
    \tilde \xi_{j,z,\tau,L,N} = \sum_{i=1}^L (z \varphi_{i,N} - \varphi'_{i,N}) c_{ij}.
\end{equation}

\begin{figure}
    \includegraphics[width=\linewidth,draft=false]{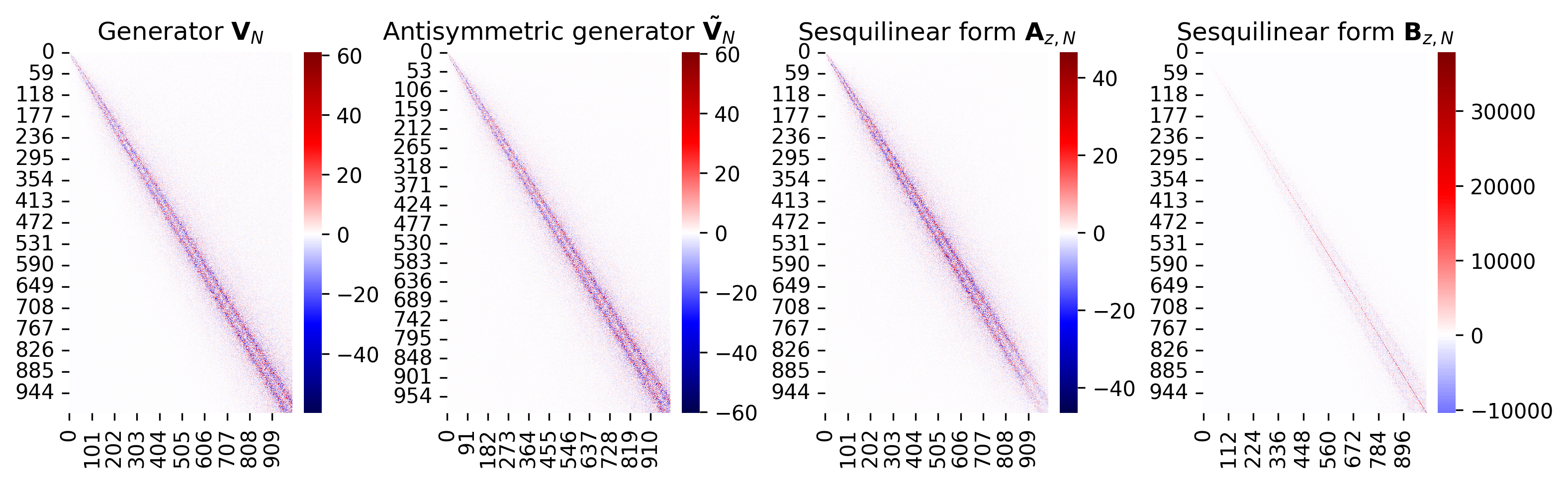}
    \caption{Matrices associated with the data-driven variational eigenvalue problem from \cref{def:varev_datadriven}. From left to right, the panels show heat maps of the generator matrix $\bm V_N$, its antisymmetric part $\bm{\tilde V}_N$, the matrix representation $\bm A_{z,N}$ of the sesquilinear form $A_{z,N}$, and the matrix representation $\bm B_{z,N}$ of the sesquilinear form $B_{z,N}$. Matrices were computed for $L=1000$, $z=0.1$, and $\tau=1\times 10^{-5}$ using the same data and remaining parameters as the L63 experiments from \cref{sec:examples}.}
    \label{fig:mat}
\end{figure}

By~\eqref{eq:lambda_conv} and \cref{prty:K6}, as $N \to \infty$ at fixed $L$, the matrices $\bm A_{\tau,N}$ and $\bm B_{z,N}$ converge to $\bm A_\tau$ and $\bm B_z$ from \cref{sec:galerkin}, respectively, in any matrix norm. This implies convergence of eigenvalues and eigenfrequencies, $\lim_{N\to\infty} \beta_{j,z,\tau,L,N} = \beta_{j,z,\tau,L}$ and $\lim_{N\to\infty} \omega_{j,z,\tau,L,N} = \omega_{j,z,\tau,L}$, respectively, as well as convergence of the corresponding spectral projections on $\mathbb C^L$. As a result, by $C^1(M)$-norm convergence of $\varphi_{j,\tau,N}$ to $\varphi_{j,\tau}$, we have convergence of the functions $\tilde\xi_{j,z,\tau,L,N}$ to $\tilde\xi_{j,z,\tau,L}$ in $C(M)$.

In summary, the results of the data-driven scheme described in this subsection converge to those from the Galerkin scheme of \cref{sec:galerkin} as $N \to \infty$ at fixed $L$, and the results of the latter converge to solutions of the variational eigenvalue problem from \cref{sec:variational} as $L\to\infty$. For completeness, we include a formula for computing the  elements of the data-driven generator matrix $\bm V_N$,
\begin{displaymath}
    V_{ij,N} = \langle \phi_{i,N}, \tilde K'_N \gamma_{j,N} \rangle_N \equiv \int_X \int_X \phi_{i,N}(x) (\vec V_x \cdot \hat k_N(x, z)) \gamma_{j,N}(z) \, d\mu_N(x) \, d\mu_N(z).
\end{displaymath}
We see that this formula is structurally very similar to \eqref{eq:v_ij}, with obvious replacements of analytical terms by data-driven ones. In particular, the directional derivative $\vec V_x \cdot \hat k_N(x, z)$ of the (data-driven) kernel function $\hat k_N$ is still computed with respect to the true vector field $\vec V$. The entries $V_{ij,N}^{(2)}$ of $\mathbf V_N^{(2)}$ can be computed in a similar number via an analogous formula to \eqref{eq:v2_ij}. It should also be noted that when working with a pullback  kernel $k^{(Y)} \colon Y \times Y \to \mathbb R$ from data space (as will be the case in the examples of \cref{sec:examples}), we have
\begin{equation}
    \label{eq:xi_dataspace}
    \tilde \xi_{j,z,\tau,L,N} = F \circ \xi^{(Y)}_{j,z,\tau,L,N}, \quad  \zeta_{j,z,\tau,L,N} = F \circ \zeta^{(Y)}_{j,z,\tau,L,N}
\end{equation}
for everywhere-defined functions $\xi^{(Y)}_{j,z,\tau,L,N}\colon Y \to \mathbb C$ and $\zeta^{(Y)}_{j,z,\tau,L,N}\colon Y \to \mathbb C$ on data space; see \cref{rk:nystrom_pullback}.

\section{Numerical examples}
\label{sec:examples}

We apply the spectral decomposition technique described in \cref{sec:numimpl} to three measure-preserving ergodic systems: a linear rotation on $\mathbb T^2$, a Stepanoff flow on $\mathbb T^2$ \cite{Oxtoby53}, and the L63 system on $\mathbb R^3$ \cite{Lorenz63}. The linear torus rotation is a prototypical system with pure point spectrum and analytically known Koopman eigenvalues and eigenfunctions. The Stepanoff flow has an ergodic invariant measure supported on a smooth manifold (the Lebesgue measure on the 2-torus), and is characterized by topological weak mixing, i.e., absence of nonconstant continuous Koopman eigenfunctions due to the presence of a fixed point. The L63 system has an ergodic invariant measure supported on a fractal set (the SRB measure on the Lorenz attractor), and is mixing with an associated continuous spectrum of the Koopman operator.

Our experimental setup follows closely that of \cite{GiannakisValva24}. In each example, we numerically generate state space trajectory data $x_0, \ldots, x_{N-1} \in \mathcal M$, $x_n = \Phi^{n\,\Delta t}(x_0)$, where $\Delta t>0$ is a fixed sampling interval and $x_0$ an arbitrary initial condition. We then embed the state space trajectory in data space $Y = \mathbb R^d$ via an embedding $F\colon X \to Y$ to produce time-ordered samples $y_0, \ldots, y_{N-1}$ with $y_n = F(x_n)$. The primary difference between our setup and that of \cite{GiannakisValva24} is that the time step $\Delta t$ will be relatively large: for example, in the L63 experiments in this work we use $\Delta t = 3$, while  \cite{GiannakisValva24} uses a significantly shorter interval $\Delta t = 0.02$ in order to approximate the Koopman resolvent by a Laplace transform. Here, we expect that this increased value of $\Delta t$ will result in a more uniform sampling of the invariant measure for a given training dataset size $N$.

Using the samples $y_n$, we build Markov operators $\hat G_N$ and compute their eigenvectors $\phi_{j,N}$ and corresponding eigenvalues $\hat \lambda_{j,N}$ via the approach described in \cref{sec:markov,sec:data_driven}, using a variable-bandwidth kernel $k^{(Y)}\colon Y \times Y \to \mathbb R$ of the form~\eqref{eq:k_vb}. We then solve the variational eigenvalue problem in \cref{def:varev_datadriven} to obtain eigenfunctions $\xi_{j,z,\tau,L,N}$ of the approximate generator $V_{z,\tau,L,N}$ and their corresponding eigenfrequencies $\omega_{j,z,\tau,L,N}$.
 
We order the computed eigenpairs $(\omega_{1,z,\tau,L,N}, \xi_{1,z,\tau,L,N}), (\omega_{2,z,\tau,L,N}, \xi_{2,z,\tau,L,N}), \ldots$ in order of increasing Dirichlet energy $E_{j,z,\tau,L,N} := \mathcal E_N(\xi_{j,z,\tau,L,N})$. As discussed in \cref{sec:rkhs}, the Dirichlet energy measures the regularity of a given function and is larger for functions that project onto $\phi_{j,N}$ with small corresponding eigenvalues $\hat \lambda_{j,N}$ of $\hat G_N$. The regularity measured with the Dirichlet energy can also be a useful (post-hoc) criterion for identifying spectral pollution; see \cite[Corollary 3]{DasEtAl21}. Alternatively, one can sort eigenfunctions in terms of the decay of their autocorrelation function as in \cite{GiannakisValva24}, which acts as a metric of eigenfunction predictability and ``closeness'' to  a true Koopman eigenfunction. Recall, in particular, from~\cref{sec:spec_decomp} that a Koopman eigenfunction $\xi_j \in \tilde{H}$ with corresponding eigenfrequency $\omega_j \in \mathbb R$ will have an autocorrelation function that behaves as a pure complex phase, $C_{\xi_j\xi_j}(t) = e^{i\omega_j t}$, without amplitude decay.

Following computation of the eigenpairs $(\omega_{j,z,\tau,L,N}, \xi_{jz,\tau,L,N})$, we build the RKHS eigenfunctions $\zeta_{j,z,\tau,L,L',N} \in \mathcal H_N$ and the corresponding data space functions $\zeta^{(Y)}_{j,z,\tau,L,L',N}$ using the partial isometry $T_{L,N} \colon \hat H_N \to \mathcal H_N$ (see \cref{rk:nystrom_pullback}). To assess the out-of-sample behavior of our scheme, we evaluate these functions on test data $\tilde y_0, \ldots, \tilde y_{N-1} \in Y$ with $\tilde y_n = F(\tilde x_n) $ obtained from a trajectory $\tilde x_n = \Phi^{n\, \tilde \Delta t}(\tilde x_0)$ with different initial condition $\tilde x_0 \in \mathcal M$ and sampling interval $\widetilde{\Delta t} > 0$ to those used for the generation of the training data. We choose $\widetilde{\Delta t} \ll \Delta t$ to test the ability of our scheme to reconstruct time series of observable values at high temporal resolution from temporally coarse training data.

For the remainder of this section, if there is no risk of confusion we suppress $z$, $\tau$, $L$, $L'$, and $N$ subscripts from our notation for eigenfunctions and eigenfrequencies, i.e., $\zeta_{j,z,\tau,L,L',N} \equiv \zeta_j$, $\omega_{j,z,\tau,L,N} \equiv \omega_j$, and $E_{j,z,\tau,L,N} \equiv E_j$. A summary of the dataset attributes and numerical parameters employed in our experiments is provided in \cref{tab:parameters}.

\begin{table}
    \caption{Dataset attributes and numerical parameters used in the torus rotation, Stepanoff flow, and L63 experiments.}
    \label{tab:parameters}
    \centering\small
    \begin{tabular}{lccc}
        & Torus rotation & Stepanoff flow & Lorenz 63 system \\
        \hline
        Data space dimension $d$ & 4 & 4 & 3 \\
        Training samples $N$ & 60,000 & 68,000 & 80,000 \\
        Training sampling interval $\Delta t$ & $\sqrt{7} \approx 2.65 $ & 3.0 & 3.0 \\
        Resolvent parameter $z$ & 0.1 & 0.1 & 0.1 \\
        Regularization parameter $\tau$ & $3 \times 10^{-3}$ & $1\times 10^{-6}$ & $1 \times 10^{-5}$ \\
        Number of basis functions $L$, $L'$ & 400 & 1500 & 1000 \\
        Test samples $\tilde N$ & 2000 & 2000 & 2000 \\
        Test sampling interval $\widetilde{\Delta t}$ & 0.01 & 0.01 & 0.01 \\
        \hline
   \end{tabular}
\end{table}

\begin{figure}[t]
    \centering
    \includegraphics[width=.45\linewidth, draft=false]{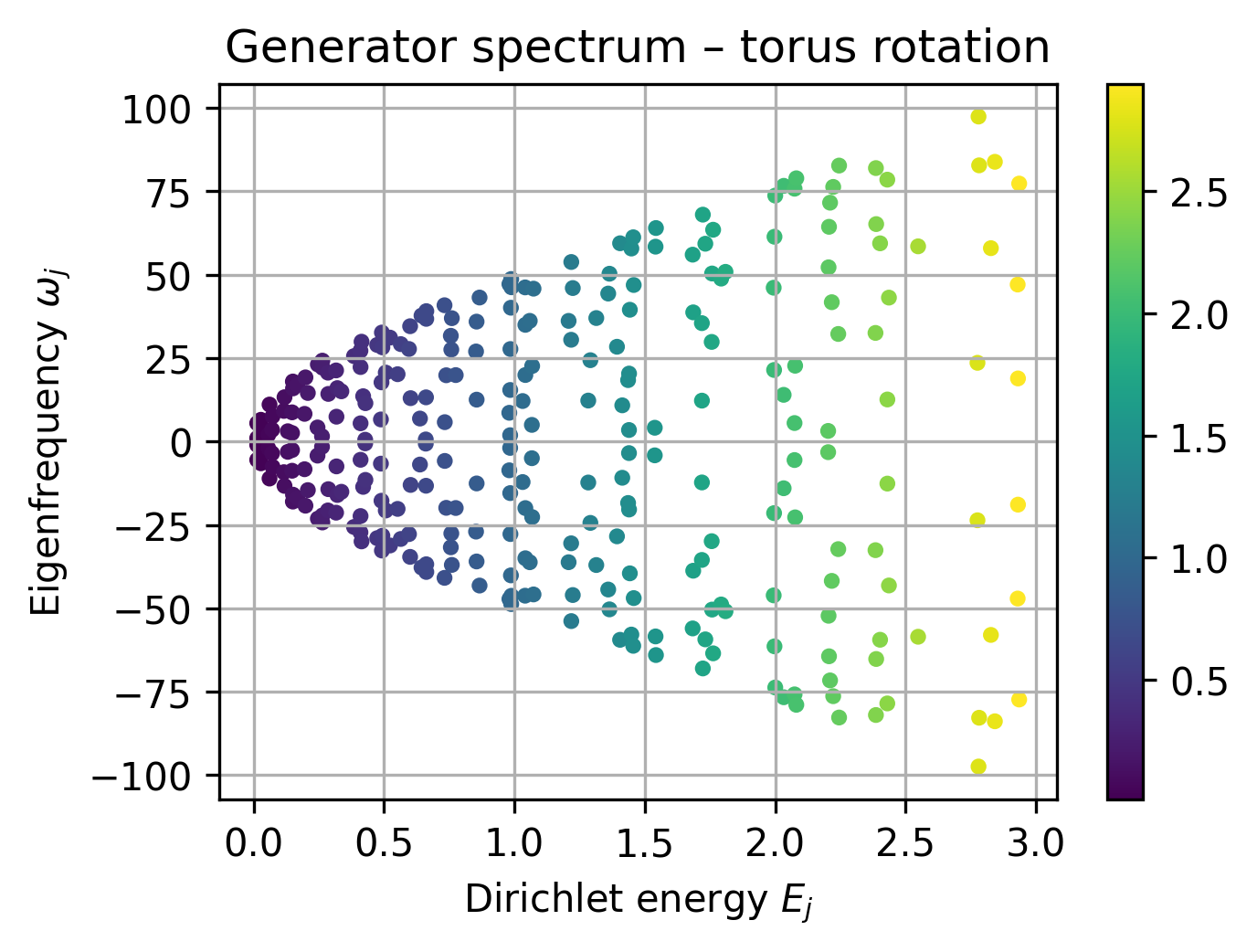}
    \hfill
    \includegraphics[width=.45\linewidth, draft=false]{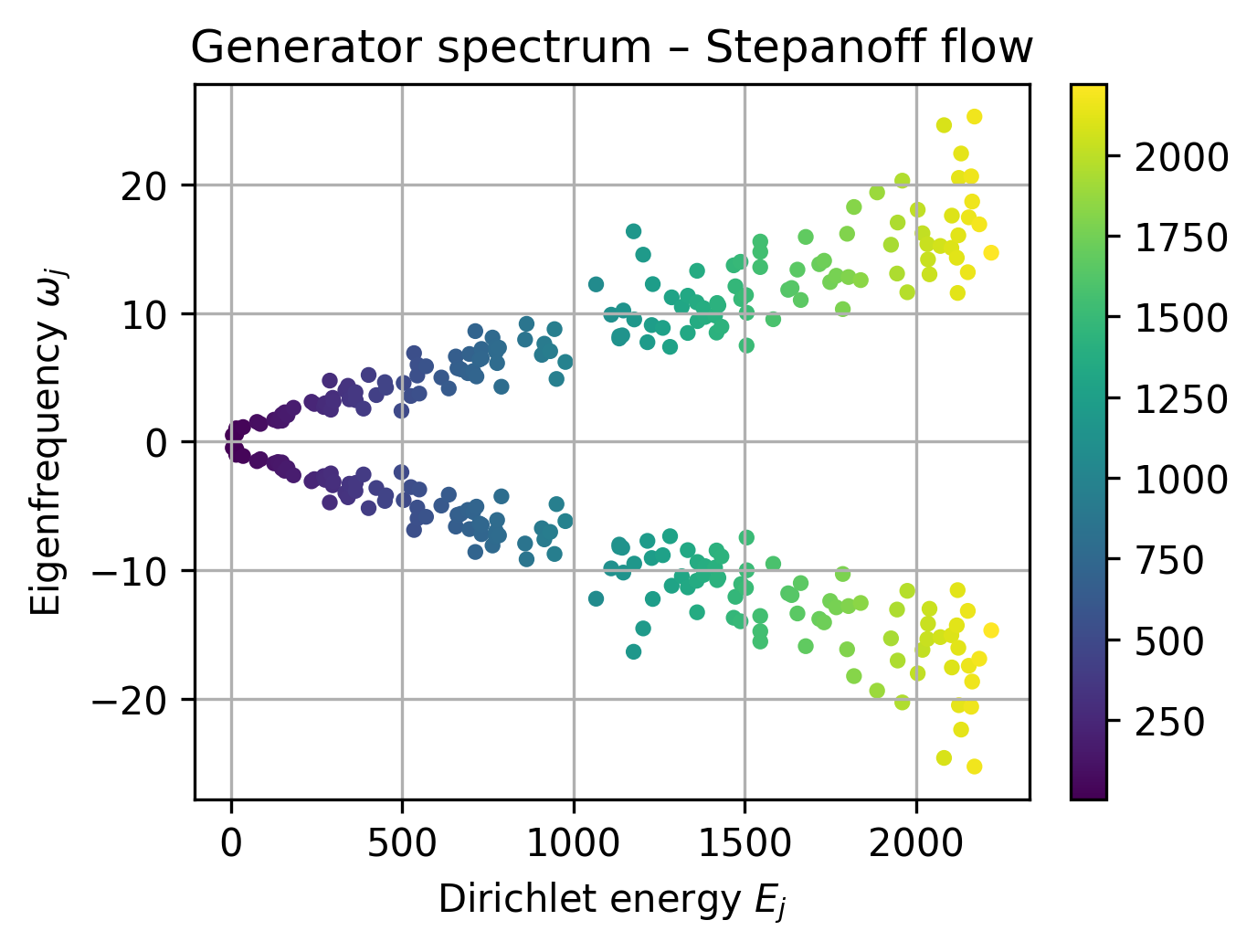}\\
    \includegraphics[width=.45\linewidth, draft=false]{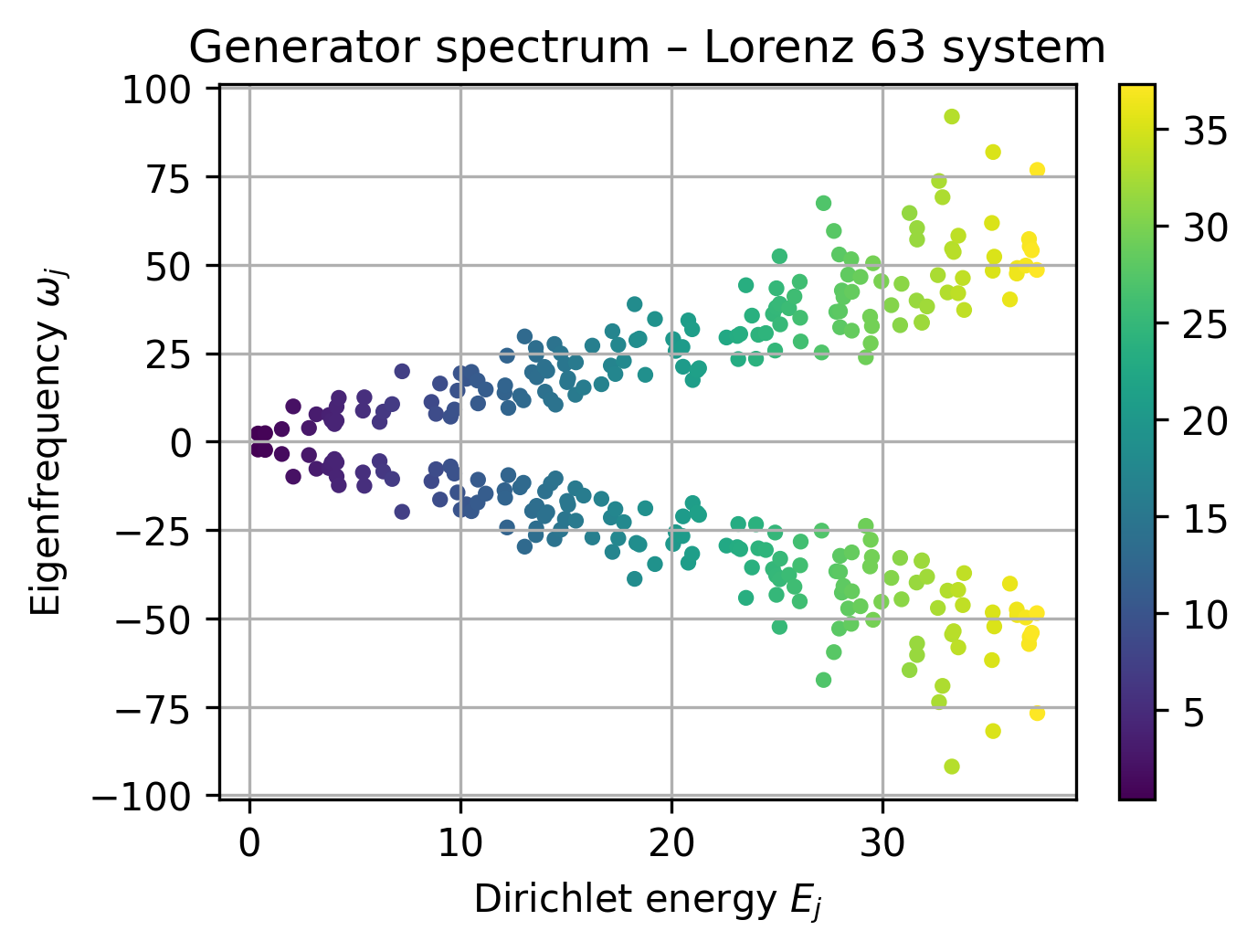}
    \caption{Regularized generator spectra for the torus rotation, Stepanoff flow, and L63 system. The plots show the first 300 eigenfrequencies $\omega_j$, ordered in increasing order of Dirichlet energy $E_j$, versus $E_j$.}
    \label{fig:spectra}
\end{figure}

\subsection{Linear torus rotation}
\label{sec:torus}

The linear torus rotation is generated by the vector field $\vec V\colon \mathbb T^2 \to T \mathbb T^2$ with components $\vec V(x) = (1, \alpha)$ in canonical angle coordinates, $x = (x^1, x^2) \in [0, 2\pi)^2$, where $\alpha$ is an irrational frequency parameter. The resulting dynamical flow is given in closed form by $\Phi^t(x) = \left( x + (1, \alpha ) t \right) \mod 2\pi$, and has a unique ergodic invariant Borel probability measure $\mu$ given by the normalized Haar measure on $\mathbb T^2$. Thus, in this example we have $\mathcal M = M = X = \mathbb T^2$. One readily verifies that the generator $V$ on $H= L^2(\mu)$ is diagonal in the Fourier basis $ \{ \phi_l(x) = e^{i l \cdot x} \}_{l \in \mathbb Z^2}$ of $H$,
\begin{displaymath}
    V \phi_l = i \alpha_l \phi_l, \quad \alpha_l = l_1 + \alpha l_2, \quad l = (l_1, l_2).
\end{displaymath}

On the basis of these facts, our data-driven eigenfunctions $\zeta_j$ should approximate (up to a phase factor) $\phi_{l_j}$ for some $l_j = (l_{1j}, l_{2j}) \in \mathbb Z^2$ depending on the ordering $\zeta_1, \zeta_2, \ldots$ used. Moreover, the corresponding eigenfrequencies $\omega_j$ should be approximately equal to integer linear combinations of the basic frequencies of the rotation, $\omega_j \approx l_{1j} + l_{2j} \alpha$. As described in \cref{sec:spec_decomp}, when normalized to unit norm in $H$, each eigenfunction $\phi_l$ induces a semiconjugacy with a circle rotation of frequency $\alpha_l$, and by continuity of Fourier functions that semiconjugacy is topological. This means that $\zeta_j$ should take values near the unit circle in the complex plane, and for a given initial condition $x \in \mathbb T^2$ the real and imaginary parts of the time series $t \mapsto \zeta_j(\Phi^t(x))$ should be sinusoids oscillating with frequency $\alpha_{l_j}$ and with a relative phase difference of 90$^\circ$.

We compute eigenfunctions $\zeta_j$ and their corresponding eigenfrequencies $\omega_j$ using a dataset of $N= \text{60,000}$ samples $x_n \in \mathbb T^2$ sampled at an interval $\Delta t = \sqrt{7} \approx 2.65$. Note that $\Delta t$ sits between the slow and fast oscillation periods of the system, $2\pi / \alpha \approx 1.15 < \Delta t < 2 \pi$. The dynamical states $x_n$ are embedded in the data space $Y = \mathbb R^4$ by means of the standard (flat) embedding of the 2-torus, $F(x) = (\cos x^1, \sin x^1, \cos x^2, \sin x^2)$, and we compute data-driven basis functions $\phi_{j,N}$ using the pullback kernel construction described in \cref{sec:markov} for the data $y_n = F(x_n)$. We compute eigenpairs $(\omega_j, \zeta_j)$ using $L=400$ such basis functions and the resolvent and regularization parameters $z = 0.1$ and $\tau= 0.003$, respectively. In addition, we evaluate the continuous representatives $\zeta_j$ on a dynamical trajectory $\tilde x_n$ of $\tilde N = 2000$ samples taken every $\widetilde{\Delta t} = 0.01$ time units. Note that $\widetilde{\Delta t} \ll 2 \pi / \alpha$ so the out-of-sample time series ($\zeta_j(\tilde x_0), \zeta_j(\tilde x_1), \ldots$) should resolve the fast oscillatory frequency of the system as well as a number of its harmonics.

\Cref{fig:eigs_torus} shows representative eigenfunction results from these computations, visualized as scatterplots of the real and imaginary parts of $\zeta_j$ on the 2-torus, traceplots of $\zeta_j(\tilde x_0), \zeta_j(\tilde x_1), \ldots$ in the complex plane, and the corresponding time series of $\Real\zeta_j(\tilde x_n)$ and $\Imag\zeta_j(\tilde x_n)$. The results are broadly consistent with the theoretical properties of Koopman eigenfunctions and eigenfrequencies mentioned above. In particular, the scatterplots of $\Real \zeta_j$ and $\Imag \zeta_j$ (first two columns of \cref{fig:eigs_torus} from the left) have the structure of plane waves consistent with the real and imaginary parts of Fourier functions. Moreover, the traceplots of $\zeta_j(\tilde x_n)$ (third column from the left) take values very close to the unit circle in $\mathbb C$, and the time series of $\Real\zeta_j(\tilde x_n)$ and $\Imag\zeta_j(\tilde x_n)$ (rightmost column) are near-perfect sinusoids with frequencies consistent with the computed eigenfrequencies $\omega_j$. In addition, as seen from the example of $\zeta_6$ in the third row of \cref{fig:eigs_torus}, the numerical results well-capture the multiplicative group structure of Koopman eigenfunctions and additive group structure of the corresponding eigenfrequencies. In this case, $\zeta_6$ exhibits a $(-1, 1)$ wavevector on the torus, and this is consistent with the product $\zeta_1 \zeta_3$ between eigenfunctions $\zeta_1$ and $\zeta_3$ (with wavevectors $(-1, 0)$ and $(0,1)$, respectively) shown in the first two rows. The corresponding eigenfrequency, $\omega_6 \approx 4.50$, is also consistent with $\omega_1 + \omega_3 \approx -1.00 + 5.49$ to two significant digits.

\begin{figure}
    \includegraphics[width=\linewidth, draft=false]{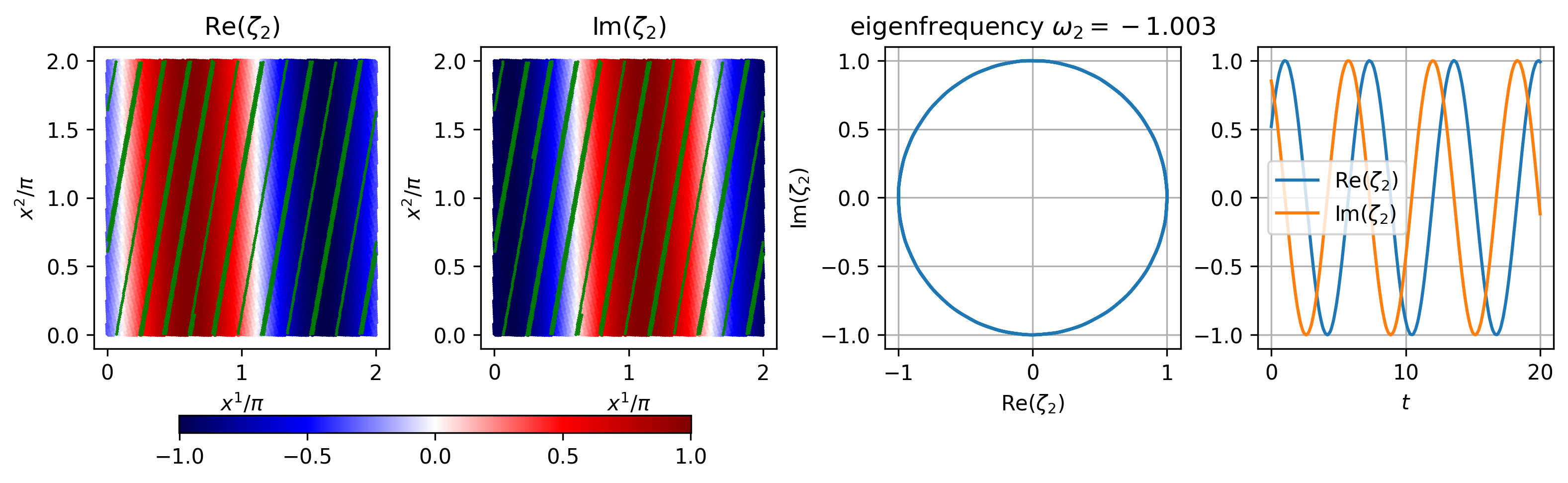}
    \includegraphics[width=\linewidth, draft=false]{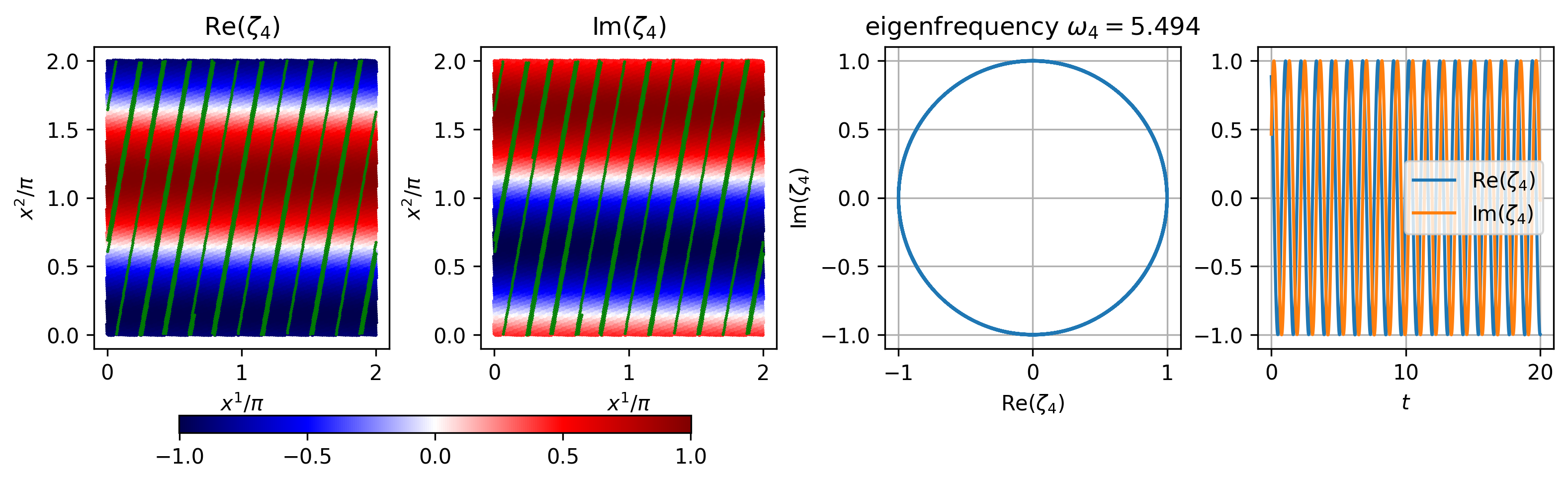}
    \includegraphics[width=\linewidth, draft=false]{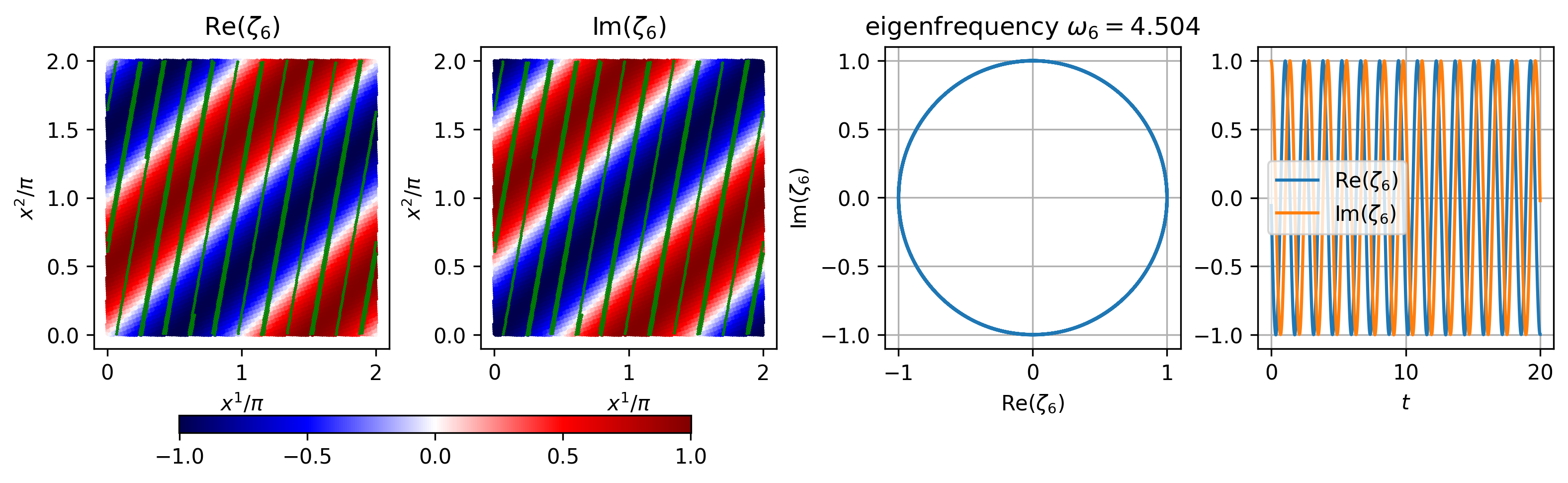}
    \caption{Representative generator eigenfunctions $\zeta_j$ for the linear torus rotation. In each row, the first two panels from the left show scatterplots of the real and imaginary parts of $\zeta_j(x_n)$ on the training data $x_n \in \mathbb T^2$. The third and fourth panels show the evolution of $\zeta_j(\tilde x_n)$ in the complex plane and time series of corresponding real and imaginary parts, sampled along the test trajectory $\tilde x_n$. The test trajectory $\tilde x_n$ is plotted with green lines in the first two panels for reference.}
    \label{fig:eigs_torus}
\end{figure}

In closing this subsection, we note that despite the highly structured nature of this system, accurate numerical spectral decompositions of the generator is non-trivial in approximation spaces of high dimension (here, $L$) since the point spectrum $\sigma_p(V) = \{ i (l_1 + l_2\alpha) \}_{l_1, l_2 \in \mathbb Z}$ is a dense subset of the imaginary line. Indeed, see the top row in \cite[Figure~5]{DasEtAl21} for an illustration of poor numerical performance of a naive approximation scheme for the generator as $L$ increases at fixed $N$ without using regularization.

\subsection{Stepanoff flow}
\label{sec:stepanoff}

Our next example comes from a class of Stepanoff flows on the 2-torus studied by \cite{Oxtoby53}. The dynamical vector field $\vec V\colon \mathbb T^2  \to \mathbb T^2$ has the coordinate representation $\vec V(x) = (V^1(x), V^2(x))$, where
\begin{equation}
    \label{eq:vec_stepanoff}
    V^1(x) = V^2(x) + (1- \alpha) (1 - \cos x^2), \quad V^2(x) = \alpha(1 - \cos(x^1 - x^2)),
\end{equation}
$x = (x^1, x^2) \in [0, 2\pi)^2$ and $\alpha$ is a real parameter. The Stepanoff vector field $\vec V$ has zero divergence with respect to the Haar measure $\mu$,
\begin{displaymath}
    \divr_\mu \vec V = \frac{\partial V^1}{\partial x^1} + \frac{\partial V^2}{\partial x^2} = 0,
\end{displaymath}
which implies that $\mu$ is an invariant measure under the associated flow $\Phi^t$. Thus, we have $\mathcal M = M = X = \mathbb T^2$ as in the linear rotation example from \cref{sec:torus}. Still, a major difference between the linear torus rotation and the Stepanoff flow is that the latter exhibits a fixed point at $x=0$, $\vec V(0) = 0$. In \cite{Oxtoby53} it is shown that the normalized Haar measure is the unique invariant Borel probability measure of this flow that assigns measure 0 to the singleton set $ \{ 0 \} \subset \mathbb T^2$ containing the fixed point.

Since any continuous, non-constant Koopman eigenfunction induces a semiconjugacy with a circle rotation of nonzero frequency, the existence of the fixed point at $x=0$ implies that the system has no continuous Koopman eigenfunctions; i.e., it is topologically weak-mixing. In fact, \cite{Oxtoby53} shows that when $\alpha$ is irrational the Stepanoff flow is topologically conjugate to a time-reparameterized irrational rotation on $\mathbb T^2$ with frequency parameters $(1, \alpha)$, for a time-reparameterization function that is singular at $x=0$. There are many examples of time-reparameterizations of pure-point-spectrum systems that yield measure-theoretically mixing dynamics \cite{KatokThouvenot06,Kocergin73}, including examples on tori \cite{Fayad02}. While, to our knowledge, and there are no results in the literature on the measure-theoretic mixing properties of Stepanoff flows, these considerations on time reparameterization in conjunction with the singularity of the time-parameterization function at $x=0$ suggest that the Koopman operator for the Stepanoff flow on $H=L^2(\mu)$ has non-trivial continuous spectrum. At the very least, the absence of continuous Koopman eigenfunctions implies that data-driven spectral computations are non-trivial for this class of systems even if eigenfunctions exist in $L^2(\mu)$.

Similarly to the linear rotation experiments from \cref{sec:torus}, we generate training data $y_n \in \mathbb R^4$ by embedding discrete-time samples $x_n \in \mathbb T^2$ of a dynamical trajectory under the Stepanoff flow using the flat embedding $F\colon \mathbb T^2 \to \mathbb R^4$. In this case, the flow $\Phi^t$ is not available analytically, so we compute the $x_n$ by numerical solution of the initial-value problem~\eqref{eq:v_phi}. We generate $N = \text{68,000}$ samples at a sampling interval $\Delta t = 3.0$. Using those samples, we compute basis functions $\phi_{j,N}$ as in \cref{sec:torus} and generator eigenpairs $(\omega_j, \zeta_j)$ using $L=1500$ basis functions and the parameters $z=0.1$, $\tau= 1 \times 10^{-6}$. We again generate a test dataset $\tilde x_n$ of $\tilde N = 2000$ samples taken on a dynamical trajectory, which independent of the training data every. The samples are taken at $\widetilde{\Delta t} = 0.01$ time units and we evaluate the continuous functions $\zeta_j$ on that dataset.

\Cref{fig:eigs_stepanoff} shows generator eigenfunction results obtained from these experiments. The layout of the plots is similar to that of the torus rotation results in \cref{fig:eigs_torus}. We see that despite the non-integrable nature of the dynamics, the depicted eigenfunctions $\zeta_j$ exhibit an approximate form of cyclicity. Particular eigenfrequencies, e.g., $\omega_{21} \approx 2.10 $ appear to be approximate harmonics of a basic frequency $\omega_3 \approx 1.03$. The level sets of the real and imaginary parts of the corresponding eigenfunctions $\zeta_j$ exhibit characteristic ``S-shaped'' patterns that are broadly aligned with the dynamical flow (see the second and third rows from the top in \cref{fig:eigs_stepanoff}).

\begin{figure}
    \includegraphics[width=\linewidth, draft=false]{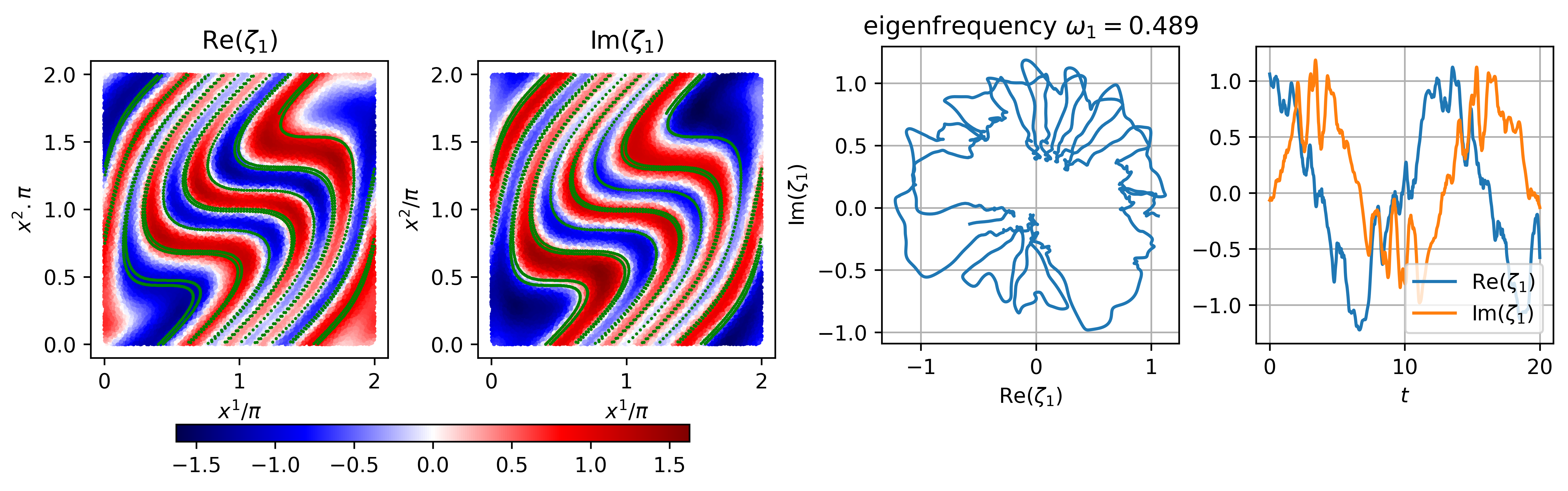}
    \includegraphics[width=\linewidth, draft=false]{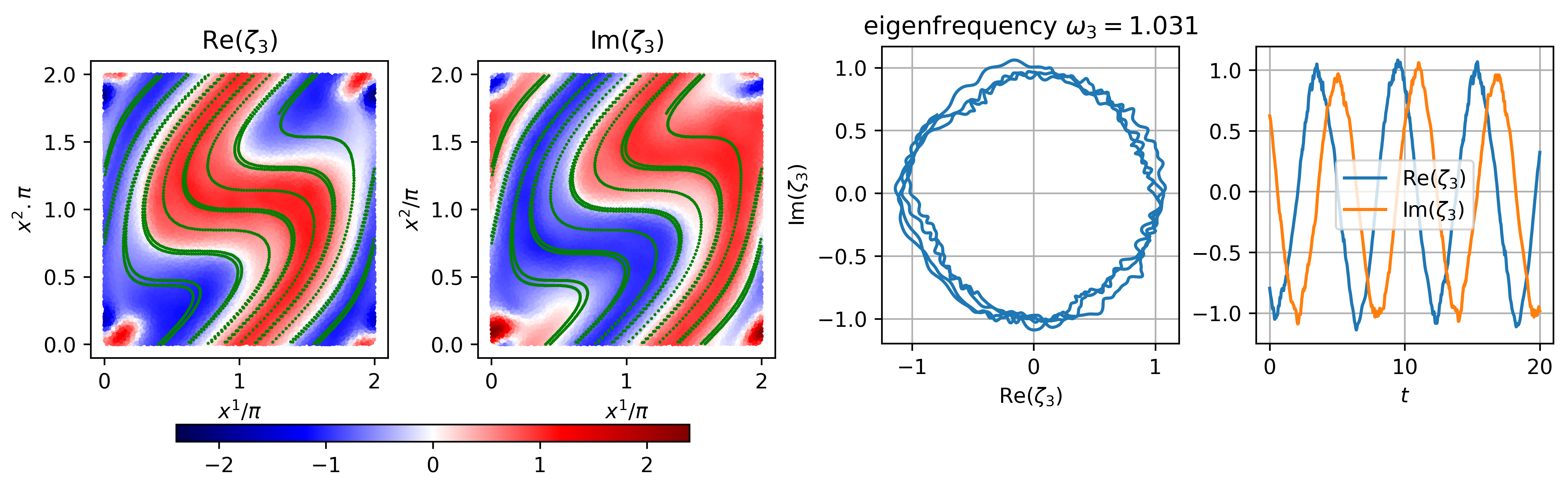}
    \includegraphics[width=\linewidth, draft=false]{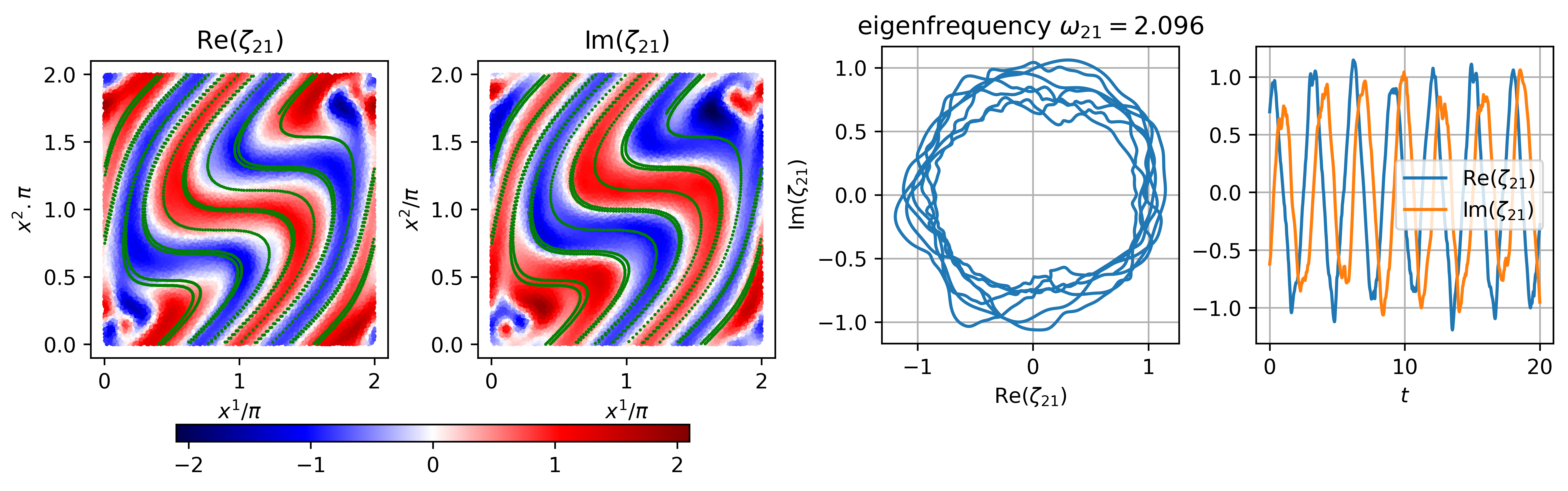}
    \includegraphics[width=\linewidth, draft=false]{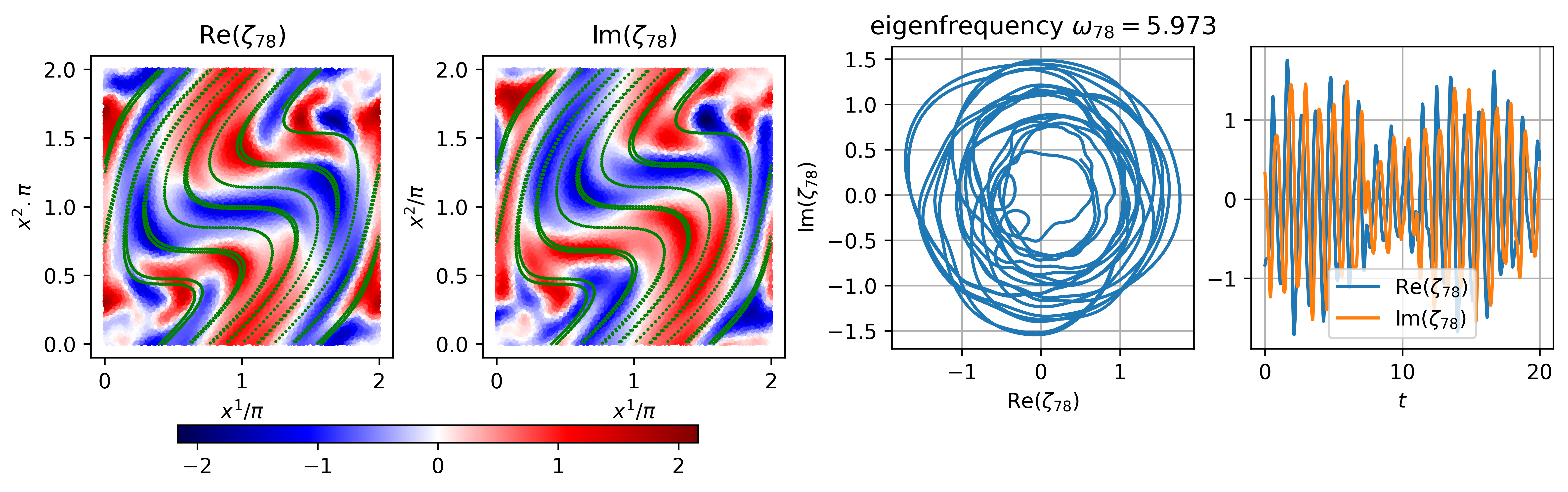}
    \caption{As in \cref{fig:eigs_torus}, but for regularized generator eigenfunctions for the Stepanoff flow.}
    \label{fig:eigs_stepanoff}
\end{figure}

A notable feature of these patterns is smaller-scale oscillations near the fixed point at $0$; this is particularly evident for the higher-frequency eigenfunction $\zeta_{97}$ shown in the bottom row of the figure. The presence of these small-scale oscillations is qualitatively consistent with the slowing down of the dynamical flow near the fixed point. In particular, an orbit passing through a region near the fixed point has to cross more eigenfunction wavefronts in a given time interval to produce consistent phase evolution $\zeta_j(\Phi^t(x)) \approx e^{i\omega_j t} \zeta_j(x) $ with more rapidly evolving orbits away from the fixed point.

Meanwhile, the leading nonconstant eigenfunction with respect to the Dirichlet energy ordering, $\zeta_1$, displayed in the top row of \cref{fig:eigs_stepanoff}, exhibits a similar S-shaped structure with weaker spatial variability near the fixed point and a correspondingly smaller eigenfrequency $\omega_1 \approx 0.49$. The corresponding time series exhibit oscillatory behavior that is generally consistent with a periodicity of $2\pi / \omega_1 \approx 12.8$, but the dominant low-frequency oscillations are evidently mixed with higher-frequency components.

As a final remark on the Stepanoff experiments, we mention that besides eigenfunctions such as $\zeta_1$, $\zeta_3$, and $\zeta_{21}$ (whose level sets are predominantly aligned with the dynamical flow), the numerical generator spectrum contains other eigenfunctions, e.g., $\zeta_{78}$ in the bottom row of \cref{fig:eigs_stepanoff}, featuring oscillations transverse to the flow and approximately cyclical time series behavior.

\subsection{Lorenz 63 system}
\label{sec:l63}

Our third numerical example is the L63 system on $\mathbb R^3$, generated by the vector field $\vec V \colon \mathbb R^3 \to T \mathbb R^3 \cong \mathbb R^3$, where
\begin{equation}
    \label{eq:v_l63}
    \vec V(x) = (V^1(x), V^2(x), V^3(x)) = (-\sigma(x^2 - x^1), x^1(\rho - x^3) - x^2, x^1 x^2 - \beta x^3)
\end{equation}
and $x = (x^1, x^2, x^3)$. We use the standard parameter values $\beta = 8/3$, $\rho = 28$, $\sigma = 10$, and the identity for the observation map $F\colon \mathbb R^3 \to \mathbb R^3$. For this choice of parameters, the L63 system is known to have a compact attractor $X \subset \bbR^3$ with fractal dimension $\approx 2.06$ that supports a unique (and observable) SRB measure $\mu$ \cite{Tucker99}. The system is also known to be mixing with respect to $\mu$ \cite{LuzzattoEtAl05}, which implies that the associated unitary group of Koopman operators $U^t$ has no nonzero eigenfrequencies. The $H_p$ subspace is then a one-dimensional space of  constant functions, while $H_c$ contains all zero-mean functions in $L^2(\mu)$ (i.e., $H_c = \tilde H$; see \cref{sec:spec_decomp}).

The L63 system with the standard parameter values is dissipative with respect to Lebesgue measure, $\divr_\text{Leb} \vec V = \sum_{i=1}^3 \frac{\partial V^i}{x^i} < 0$, and can be shown to possess compact absorbing balls containing $X$ \cite{LawEtAl14}. Setting the forward-invariant manifold $M$ to such an absorbing ball, we have $X \subset M \subset \mathcal M = \mathbb R^3$ and \crefrange{prty:K1}{prty:K6} are all rigorously satisfied. Compared to the linear rotation and Stepanoff flow, where the support of the invariant measure $\mu$ is a smooth manifold, $X= \mathbb T^2$, a challenging aspect of the L63 system is that $\mu$ is supported on a fractal set. The kernel eigenbasis $\phi_{j,N}$, provides an effective way of building Galerkin approximation spaces in such settings with invariant measures supported on geometrically complex sets.

In this example, we use a training dataset consisting of $N = \text{80,000}$ samples $y_n = x_n \in \mathbb R^3$ which are taken at a sampling interval of $\Delta t = 3$ time units from a numerical trajectory that was allowed to equilibrate near the Lorenz attractor. Note that $\Delta t$ is significantly longer than the Lyapunov timescale $1/\Lambda \approx 1.10$, where $\Lambda \approx 0.91$ is the positive Lyapunov exponent of the L63 system for the standard parameter values \cite{Sprott03}. As a result, we expect the training samples $x_n$ to exhibit a degree of statistical independence. Using this dataset, we compute generator eigenpairs $(\omega_j,\zeta_j)$ using $L=1000$ basis functions $\phi_{j,N}$ and the parameters $z=0.1$ and $\tau= 1 \times 10^{-5}$. As in \cref{sec:torus,sec:stepanoff}, we generate a test dataset $\tilde y_n = \tilde x_n$ of $\tilde N= 2000$ samples at higher temporal resolution $\widetilde{\Delta t} = 0.01$, and we evaluate the continuous functions $\zeta_j$ on the $\tilde x_n$ data.

\begin{figure}
    \includegraphics[width=\linewidth, draft=false]{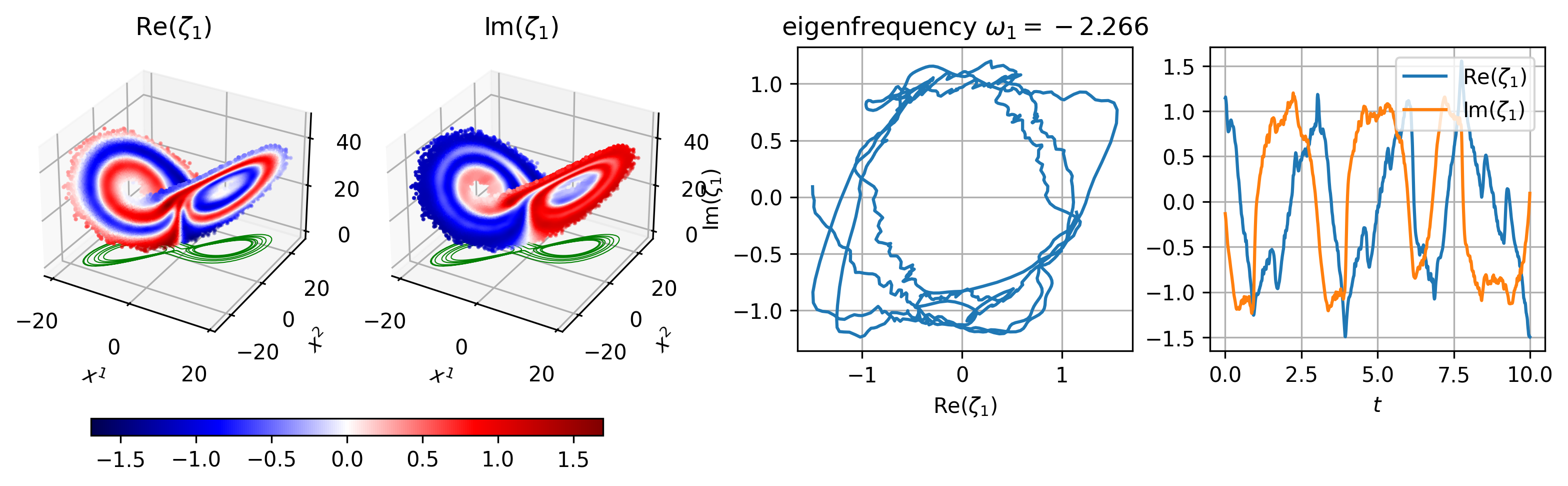}
    \includegraphics[width=\linewidth, draft=false]{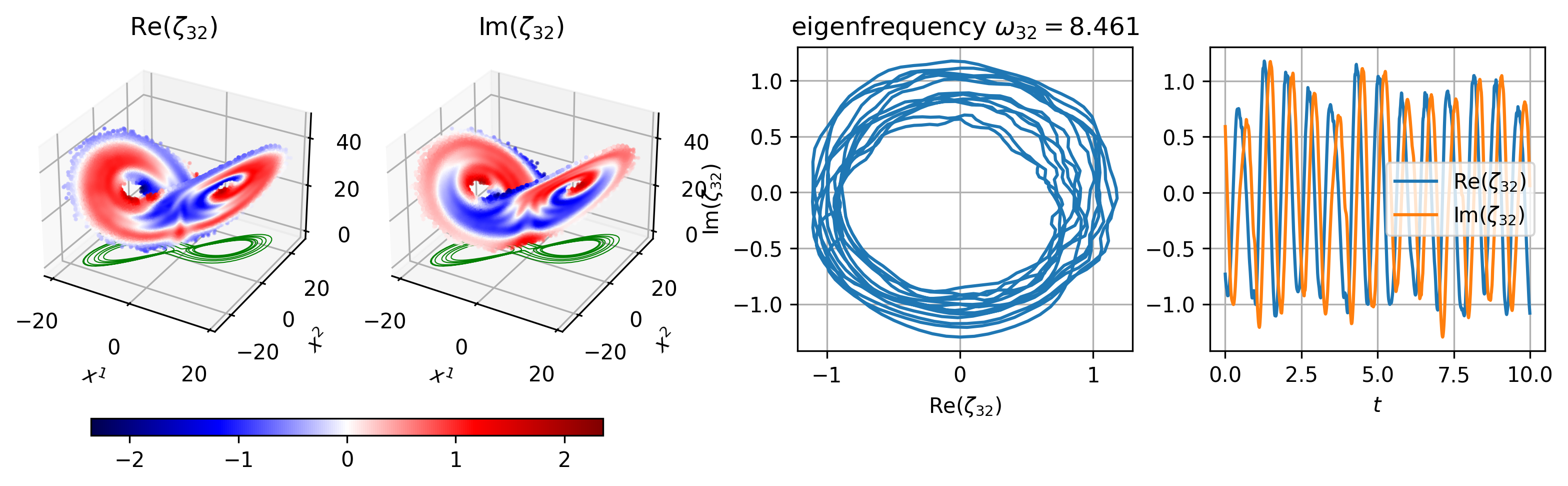}
    \includegraphics[width=\linewidth, draft=false]{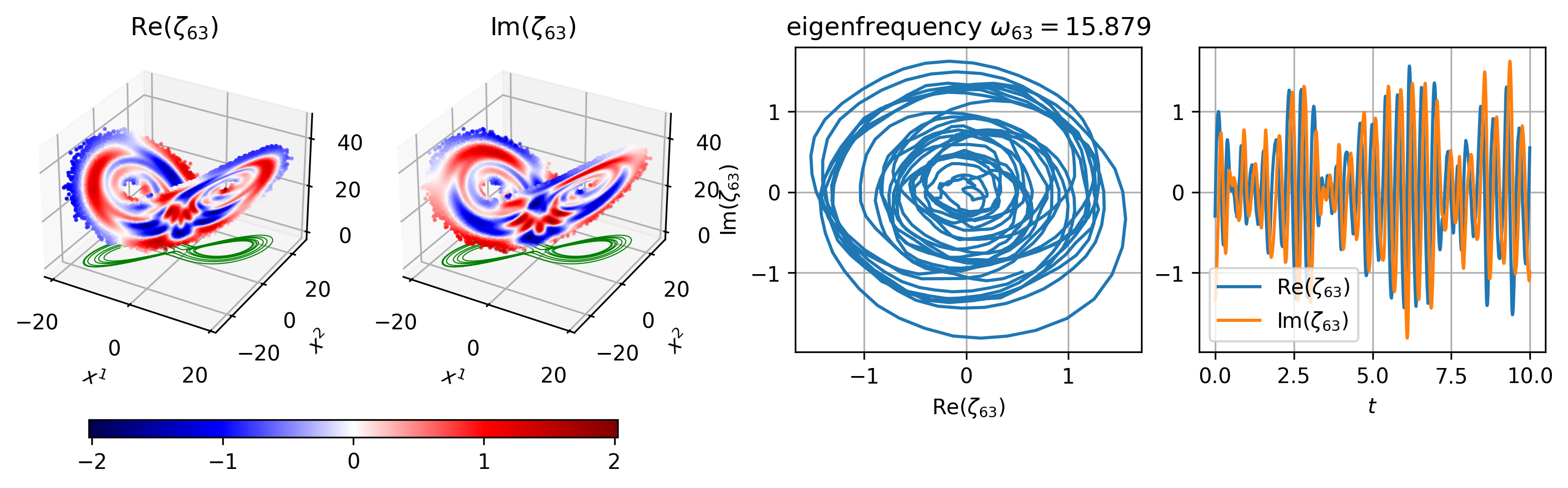}
    \includegraphics[width=\linewidth, draft=false]{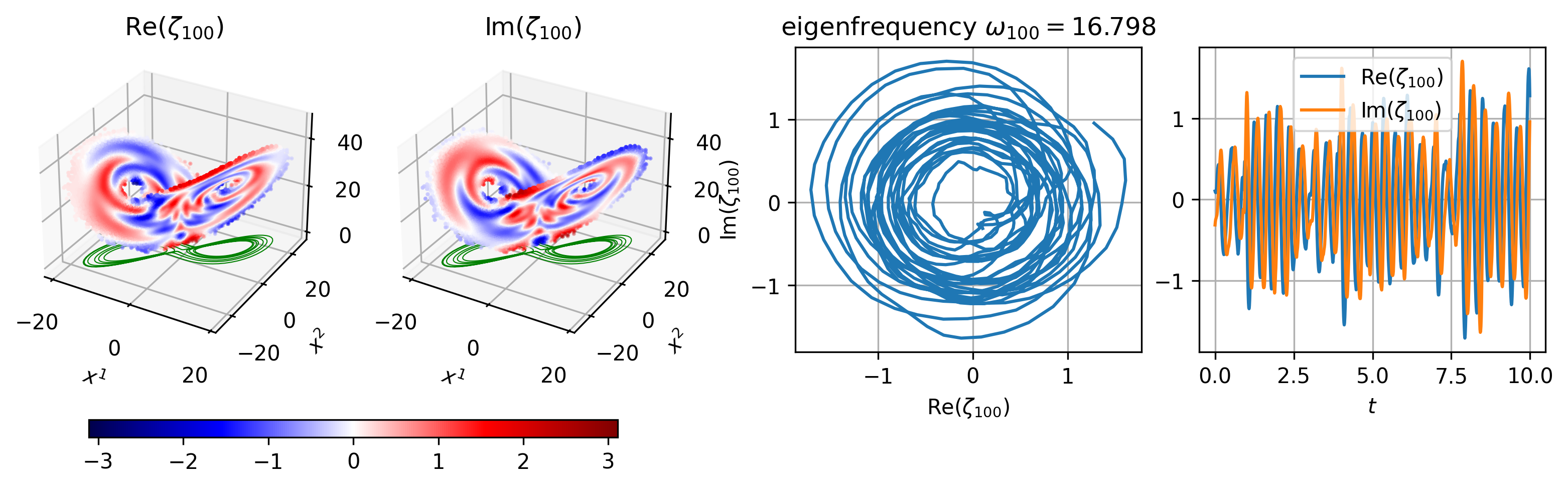}
    \caption{As in \cref{fig:eigs_torus}, but for regularized generator eigenfunctions for the L63 system.}
    \label{fig:eigs_l63}
\end{figure}

\Cref{fig:eigs_l63} shows representative eigenfunction results obtained from these experiments, visualized as scatterplots on the Lorenz attractor (first and second columns from the left), traceplots in the complex plane (third column from the left), and time series plots of the real and imaginary parts of $\zeta_j$ (rightmost column). Despite the mixing nature of the dynamics, the extracted eigenfunctions exhibit approximate cyclicity that persists on significantly longer timescales than the Lyapunov timescale of the system. One of these eigenfunctions, $\zeta_{32}$ (second row in \cref{fig:eigs_l63}), with corresponding eigenfrequency $\omega_j \approx 8.46$, shows a characteristic ``wavenumber-1'' pattern on the lobes on the attractor that qualitatively resembles results obtained by several other methods; see, e.g., \cite{KordaEtAl20,DasEtAl21,FroylandEtAl21,GiannakisValva24,ColbrookEtAl24}. Other results in \cref{fig:eigs_l63} include eigenfunction $\zeta_{100}$ (bottom row) with a corresponding eigenfrequency $\omega_{100} \approx 16.8$ that appears to be a ``wavenumber-2'' harmonic of $\zeta_{32}$. Additionally, the eigenfunction  $\zeta_{63}$ (third row) has an eigenfrequency $\omega_{63} \approx 15.9$ that is also close to $2 \omega_{32}$ but exhibits smaller-scale oscillations along the radial directions of the attractor lobes. As a lower-frequency example, eigenfunction $\zeta_1$ (top row), has an eigenfrequency of $\omega_1 \approx -2.23$ and a corresponding oscillatory timescale of $2\pi / \lvert \omega_1 \rvert \approx 2.78$ that is approximately 2.5 times longer than the Lyapunov timescale. Capturing such low-frequency coherent patterns is challenging with data-driven techniques, as evidenced by the noisier nature of the time series plots in the top row of \cref{fig:eigs_l63}. Nonetheless, the evolution of $\zeta_1$ is predominantly cyclical.

\section{Concluding remarks}
\label{sec:conclusions}

We have proposed a data-driven technique for spectral approximation of Koopman operators of continuous-time, measure-preserving, ergodic dynamical systems. A key feature of this method is that it is physics-informed, in the sense that it makes use of known equations of motion through the dynamical vector field without requiring access to pairs of snapshot data to approximate the Koopman operator. Another primary element is a bounded transformation of the Koopman generator, $V$, that renders it into a skew-adjoint operator that is amenable to compact approximations by smoothing with Markovian kernel integral operators. This leads to a two-parameter family $V_{z,\tau}$, $z,\tau>0$, of skew-adjoint operators on $L^2$ with compact resolvent, and thus with a discrete spectrum consisting entirely of eigenvalues and an associated orthonormal basis of eigenfunctions. Our main theoretical result, \cref{thm:spec-conv}, established spectral convergence of the family $V_{z,\tau}$ to the generator $V$ in the iterated limit of $z \to 0^+$ after $\tau\to 0^+$. The approximation of these operators in the data-driven setting converges in four successive limits: namely that of the number of data snapshots $N \to \infty$, basis functions $L \to \infty$, smoothing parameter $\tau\to 0^+$, and resolvent parameter $z \to 0^+$. In the discrete-time setting it has been shown that three successive limits are required for consistent spectral approximation of unitary Koopman operators of measure-preserving systmems  \citep{colbrook_limits_2024}. The use of iterated limits is thus likely to be necessary when accurately computing the spectrum of Koopman generators. Compared to our previous work, the physics-informed approach presented in this paper avoids having to take a limit of vanishing sampling interval, $\Delta t \to 0^+$, needed for finite-difference approximation of the generator \cite{DasEtAl21} or approximation of the Koopman resolvent via a Laplace transform \cite{GiannakisValva24}.

In addition, we have shown that the eigendecomposition of $V_{z,\tau}$ can be formulated as a generalized eigenvalue problem involving $V$ and the Markov smoothing operators. We developed variational Galerkin methods for approximating solutions of this problem in finite-dimensional spaces spanned by kernel eigenfunctions.  A noteworthy aspect of these approximation schemes is that they employ automatic differentiation to evaluate the action of the dynamical vector field on kernel functions, thus making direct use of equations of motion. Moreover, as with other kernel methods, these numerical schemes are well-suited for handling invariant measures supported on geometrically complex sets (e.g., fractal attractors) and/or measures with supports embedded in high-dimensional ambient data spaces. In particular, our usage of kernel integral operators serves the dual purpose of basis learning for Galerkin approximation and smoothing for spectral regularization of the generator.

The eigenfunctions of $V_{z,\tau}$ can be interpreted as generalizations of Koopman eigenfunctions of pure point spectrum systems, in the sense of providing an orthonormal basis of $L^2$ that contains elements with approximately cyclical behavior and slow decay of correlations. In addition, every eigenfunction in our construction has a representative in an RKHS of continuously differentiable functions, allowing out-of-sample evaluation by means of Nystr\"om operators that can be used, e.g., to reconstruct eigenfunction time series at high temporal resolution from coarsely sampled (in time) training data. We have demonstrated these features through a suite of numerical experiments involving low-dimensional systems with different spectral characteristics: (i) an ergodic torus rotation as a prototypical system with pure point spectrum; (ii) a Stepanoff flow on the 2-torus as an example with topological weak mixing (absence of continuous Koopman eigenfunctions) and smooth invariant measure; and (iii) the L63 system as an example with measure-theoretic mixing and invariant measure supported on a fractal set.

In terms of future work, it would be interesting to characterize limits of the eigenfunctions of $V_{z,\tau}$ in a space of distributions such as $H^{-1}$ (see \cref{sec:rkhs}), similarly to the analysis of \cite{ColbrookEtAl24} for unitary Koopman operators.  Another possible direction would be to employ the eigendecomposition of $V_{z,\tau}$ in supervised learning schemes for the Koopman evolution of observables.

\section*{Acknowledgments}

Dimitrios Giannakis acknowledges support from the US Office of Naval Research under MURI grant N00014-19-1-242, and the US Department of Defense, Basic Research Office under Vannevar Bush Faculty Fellowship grant N00014-21-1-2946. Claire Valva was supported by the US National Science Foundation Graduate Research Fellowship under grant DGE-1839302. The authors declare no competing interests.

\section*{Code availability}

Python code reproducing the numerical results in this paper is available at the repository \url{https://github.com/dg227/NLSA} under directory \url{Python/examples/resolvent_compactification}.

\appendix

\section{Auxiliary results}
\label{app:aux_results}

\subsection{Proof of \cref{prop:c1}.}
\label{app:proof_prop_c1}

First, note that when implemented with a data-adapted kernel $k$ such as the variable-bandwidth Gaussian kernel \eqref{eq:k_vb}, the bistochastic normalization procedure described in \cref{sec:markov} may start from a data-dependent kernel $k_N \colon \mathcal M \times \mathcal M \to \mathbb R$ approximating $k$; e.g.,
\begin{equation}
    \label{eq:k_vb_datadriven}
    k_N(x,y) = \exp \left( - \frac{\lVert x-y\rVert^2}{\epsilon^2 \rho_N(x)\rho_N(y)}\right),
\end{equation}
where $\rho_N \colon \mathcal M \to \mathbb R$ is a data-driven approximation of the bandwidth function $\rho$ in~\eqref{eq:k_vb}. For simplicity of exposition, in what follows we prove \cref{prop:c1} for a fixed, data-independent kernel $k$ whose restriction on $M \times M$ is $C^1$. A similar method of proof can be employed for data-dependent kernels $k_N$ that converge to $k$ as $N\to\infty$ in $C^1(M\times M)$ norm. This will hold, for instance, in the case of~\eqref{eq:k_vb_datadriven} where the bandwidth functions $\rho_N$ are built using kernel density estimation with $C^1$ kernels; see \cite{Giannakis19,DasEtAl21} for further details.

In the following, $C(M, T^*M)$ will denote the Banach space of continuous dual vector fields on $M$ (sections of the cotangent bundle $T^*M$), equipped with the norm $\lVert \alpha \rVert_{C(M,T^*M)} = \left\lVert x \mapsto \lVert \alpha(x)\rVert_{T^*_xM} \right\rVert_{C(M)}$. Here, $\lVert \cdot \rVert_{T^*_x(M)}$ is the norm on the dual tangent space $T_x^*M$ at $x \in M$ induced by some $C^1$ Riemannian metric on $M$ (all such choices of metric yield equivalent norms by compactness of $M$).  Letting also $d\colon C^1(M) \to C(M, T^*M)$ denote the exterior derivative, we can define the $C^1(M)$ norm on functions as
\begin{equation}
    \label{eq:c1_norm}
    \lVert f\rVert_{C^1(M)} = \lVert f \rVert_{C(M)} + \lVert df  \rVert_{C(M, T^*M)}.
\end{equation}

With $\kappa \in C^1(M \times M)$ and $\nu$ a Borel probability measure, we let $I_{\kappa,\nu} \colon C(M) \to C^1(M)$ be the integral operator defined as $I_{\kappa,\nu} f = \int_M \kappa(\cdot, x)f(x)\, d\nu(x)$ and $\mathfrak d_\kappa \in C(M \times M)$ the positive, continuous function $\mathfrak d_\kappa(x, y) = \lVert d_x\kappa(\cdot, y)\rVert_{T^*_x(M)}$. Using~\eqref{eq:c1_norm}, we get the following bound for the $C^1$ norm of $I_{\kappa,\nu} f$ that will be useful below,
\begin{align}
    \nonumber\lVert I_{\kappa,\nu} f\rVert_{C^1(M)} &= \lVert I_{\kappa,\nu} f  \rVert_{C(M)} + \lVert d I_{\kappa,\nu} f \rVert_{C(M, T^*M)} \\
                                     \nonumber & \leq \lVert \kappa\rVert_{C(M\times M)} \lVert f \rVert_{C(M)} + \lVert \mathfrak d_\kappa\rVert_{C(M \times M)} \lVert f\rVert_{C(M)} \\
                                     \label{eq:c1_bound} &= \lVert \kappa\rVert_{C^1(M\times M)} \lVert f\rVert_{C(M)}.
\end{align}
We also define the covector-field-valued integral operator $I_{d\kappa,\nu}\colon C(M) \to C(M, T^*M)$ as
\begin{displaymath}
    I_{d\kappa,\nu} f = \int_M d\kappa(\cdot, x) f(x) \, d\nu(x).
\end{displaymath}

Next, let $\mathbb E_\nu f = \int_M f \, d\nu$ denote the expectation of $f \in C(M)$ with respect to a Borel probability measure $\nu$ on $M$. Following \cite{VonLuxburgEtAl08}, we will employ the notion of Glivenko--Cantelli classes to identify sets of functions with uniform convergence of ergodic averages.

\begin{definition}
    Let $\nu_N$ be a sequence of Borel probability measures that converges, as $N\to \infty$, to $\nu$ in weak-$^*$ sense. A set $\mathcal G \subseteq C(M)$ is said to be a Glivenko--Cantelli class with respect to $\nu_N$ and $\nu$ if $\lim_{N\to\infty}\sup_{g \in \mathcal G} \lvert \mathbb E_{\nu_N} g - \mathbb E_\nu g  \rvert = 0$.
\end{definition}

If $\kappa\colon \mathcal M \times \mathcal M \to  \mathbb R$ and $f\colon M \to \mathbb R$ are continuous, it can be shown by results of \cite{VonLuxburgEtAl08} that the set $\mathcal G_{\kappa,f} = \{ \kappa(x,\cdot)f: x \in M \}$ is Glivenko--Cantelli with respect to the invariant measure $\mu$ and sampling measures $\mu_N$. For a compact manifold $M$, $\mathcal G_{d\kappa, f} = \{ y \mapsto d_x \kappa(\cdot, y)f(y): x \in M \}$ can also be shown to be a Glivenko--Cantelli class within $C(M, T^*M)$. We then have:

\begin{lemma}
    \label{lem:c1_conv}
    Let $\kappa_N \in C^1(M\times M)$ and $f_N \in C(M)$ be sequences of kernels and functions converging to $\kappa \in C^1(M \times M)$ and $f \in C(M)$, respectively, as $N \to \infty$. Then, $g_N = I_{\kappa_N, \mu_N} f_N$ converges to $g = I_{\kappa,\mu} f$ in $C^1(M)$.
\end{lemma}

\begin{proof}
    We split the estimate for $\lVert g_N - g\rVert_{C^1(M)}$ into three terms,
    \begin{displaymath}
        \lVert g_N -g \rVert_{C^1(M)} \leq T_{1,N} + T_{2,N} + T_{3,N},
    \end{displaymath}
    where
    \begin{align*}
        T_{1,N} &= \lVert I_{\kappa_N,\mu_N} f_N - I_{\kappa,\mu_N} f_N\rVert_{C^1(M)}, \\
        T_{2,N} &= \lVert I_{\kappa,\mu_N} f_N - I_{\kappa,\mu_N} f \rVert_{C^1(M)}, \\
        T_{3,N} &= \lVert I_{\kappa,\mu_N} f - I_{\kappa,\mu} f\rVert_{C^1(M)}.
    \end{align*}
    We can estimate $T_{1,N}$ and $T_{2,N}$ using~\eqref{eq:c1_bound} as
    \begin{align*}
    T_{1,N} &= \lVert (I_{\kappa_N} - I_{\kappa,\mu_N}) f_N\rVert_{C^1(M)} \leq \lVert \kappa_N - \kappa\rVert_{C^1(M\times M)} \lVert  f_N\rVert_{C(M)}, \\
        T_{2,N} &= \lVert I_{\kappa, \mu_N} (f_N - f)\rVert_{C^1(M)} \leq \lVert\kappa \rVert_{C^1(M\times M)} \lVert f_N - f\rVert_{C(M)},
    \end{align*}
    and we deduce that $\lim_{N\to\infty} T_{1,N} = \lim_{N\to\infty} T_{2,N} =0$ from the facts that $\lim_{N\to\infty} \lVert \kappa_N - \kappa \rVert_{C^1(M \times M)} = 0$ and $\lim_{N\to\infty}\lVert f_N - f\rVert_{C(M)} =0$. For $T_{3,N}$, we have
    \begin{displaymath}
        T_{3,N} = \lVert I_{\kappa,\mu_N} f - I_{\kappa,\mu} f\rVert_{C(M)} + \lVert I_{d\kappa,\mu_N} f - I_{d\kappa,\mu} f\rVert_{C(M, T^*M)},
    \end{displaymath}
    and it follows that $\lim_{N\to\infty} T_{3,N} = 0$ since $\mathcal G_{\kappa,f}$ and $ \mathcal G_{d\kappa, f}$ are Glivenko--Cantelli classes.
\end{proof}

We will now use \cref{lem:c1_conv} repeatedly to prove the proposition.

First, consider the normalization procedure used to construct the asymmetric kernel function $\hat k_N\colon \mathcal M \times \mathcal M \to \mathbb R$ (the data-driven analog of $\hat k$ from \eqref{eq:k_asym}):
\begin{displaymath}
    d_N = \int_M k(\cdot, x) \, d\mu_N(x), \quad q_N = \int_M \frac{k(\cdot, x)}{d_N(x)} \, d\mu_N(x), \quad \hat k_N(x, y) = \frac{k(x,y)}{d_N(x) q^{1/2}_N(y)}.
\end{displaymath}
Applying \cref{lem:c1_conv} for $\kappa= k$ and $f_N=\bm 1$, it follows that $d_N$ converges to $d$ in $C^1(M)$, and so does $1/d_N$ to $1/d$ since $1/d$ is strictly positive on the compact manifold $M$. Thus, from \cref{lem:c1_conv} with $\kappa = k$ and $f_N = 1 / d_N$, we get $C^1(M)$ convergence of $q_N$ to $q$. We therefore deduce $C^1(M \times M)$ convergence of $\hat k_N$ to $\hat k$.

Consider now the integral operators $\hat K_N \colon \hat H_N \to C^1(M)$, $\hat K_N^\top \colon \hat H_N \to C^1(M) $, and $\tilde K_N\colon \hat H_N \to \hat H_N$, where
\begin{displaymath}
    \hat K_N f = \int_M \hat k_N(\cdot, x)f(x) \, d\mu_N(x), \quad \hat K_N^\top f = \int_M \hat k_N(x, \cdot)f(x) \, d\mu_N(x),
\end{displaymath}
and $\tilde K_N = \iota_N \hat K_N$. Consider also an SVD of $\tilde K_N$,
\begin{displaymath}
    \tilde K_N = \sum_{j=0}^{N-1} \phi_{j,N} \sigma_{j,N} \langle \gamma_{j,N}, \cdot \rangle_N
\end{displaymath}
where $\sigma_{j,N} \geq 0$ are singular values, and $\{\phi_{j,N}\}_{j=0}^{N-1}$ and $\{\gamma_{j,N}\}_{j=0}^{N-1}$ are orthonormal bases of $\hat H_N$ consisting of left and right singular vectors, respectively. Since $\tilde K_N \tilde K_N^* \equiv G_N$ and $\tilde K_N^* \tilde K_N$ converge spectrally to $\tilde K \tilde K^* \equiv G$ and $\tilde K^* \tilde K$, respectively, for every $j \in \mathbb N$ there exists $N_* > 0$ such that, by strict positivity of $\sigma_j$, $\sigma_{j,N} > 0$ for all $N > N_*$. For every such $N$, the corresponding singular vectors $\phi_{j,N}$ and $\gamma_{j,N}$ have $C^1(M)$ representatives
\begin{displaymath}
    \varphi_{j,N} = \frac{1}{\sigma_{j,N}} \hat K_N \gamma_{j,N}, \quad  \tilde\gamma_{j,N} = \frac{1}{\sigma_{j,N}} \hat K_N^\top \phi_{j,N}.
\end{displaymath}
Moreover, for every choice of left and right singular vectors $\phi_j$ and $\gamma_j$ of $\hat K$ corresponding to $\sigma_j$, with continuous representatives $\varphi_j$ and $\tilde\gamma_j$, respectively, there exist sequences $(\phi_{j,N})_N$ and $(\gamma_{j,N})_N$ of left and right singular vectors of $\tilde K_N$ such that $\lim_{N\to\infty} \varphi_{j,N} = \varphi_j$ and $\lim_{N\to\infty} \tilde\gamma_{j,N} = \tilde \gamma_j $ in $C(M)$. Here, $\tilde \gamma_j$ is the continuous representative of $\gamma_j$ defined analogously to $\tilde \gamma_{j,N}$ using an integral operator $\hat K^\top\colon H \to C^1(M)$.

We complete the proof by verifying that the convergence of $\tilde\varphi_{j,N}$ to $\varphi_j$ is, in fact, in $C^1(M)$ norm. Indeed, since
\begin{displaymath}
    \varphi_{j,N} = \frac{1}{\sigma_{j,N}} \int_M \hat k_N(\cdot, x)\gamma_{j,N}(x)\, d\mu_N(x) = \frac{1}{\sigma_{j,N}} \int_M \hat k_N(\cdot, x)\tilde \gamma_{j,N}(x)\, d\mu_N(x) = \frac{1}{\sigma_{j,N}} I_{\hat k_N, \mu_N} \tilde \gamma_{j,N},
\end{displaymath}
it follows from \cref{lem:c1_conv} with $\kappa = \hat k_N$ and $f_N = \tilde \gamma_{j,N}$, and the fact that $\sigma_{j,N} \to \sigma_j$, that $\varphi_{j,N}$ converges to $\varphi_j$ in $C^1(M)$, proving the proposition.

\subsection{Results from the Borel functional calculus}
\label{app:bound}

Let $\mathbb H$ be a Hilbert space, $\Sigma$ a $\sigma$-algebra on a set $\mathbb X$, and $E \colon \Sigma \to B(H)$ a PVM. We recall the following standard results from the Borel functional calculus \cite[e.g.,][]{Folland16}:
\begin{itemize}
    \item For every $f \in \mathbb H$, the map $E_f\colon \Sigma \to \mathbb R_+$ with $E_f(S) = \langle f, E(S) f \rangle_{\mathbb H} $ is a positive finite measure.
    \item For every bounded $\Sigma$-measurable function $g\colon \mathbb X \to \mathbb C$ and every $f \in \mathbb H$, we have $\lVert G f \rVert_{\mathbb H}^2 = \int_{\mathbb X} \lvert g \rvert^2 \, dE_f$, where $G = \int_{\mathbb X} g \, dE$.
\end{itemize}

The following lemma is used in the proof of \cref{thm:vz_conv}.

\begin{lemma}
    Let $A\colon D(A) \to \mathbb H$ be a self-adjoint operator on a Hilbert space $\mathbb H$ with associated PVM $E\colon \mathcal B(\mathbb R) \to B(\mathbb H)$. Then, for every bounded Borel-measurable function $g\colon \mathbb R \to \mathbb C$ and Borel-measurable set $S \subseteq \mathbb R$, we have $ \lVert G_S f \rVert_{\mathbb H} \leq \lVert g \rVert_\infty \sqrt{E_f(S)}$, where $G_S = \int_{\mathbb S} g \, dE$.
    \label{lem:bound}
\end{lemma}

\begin{proof}
    We have $G_S = g_S(A)$, where $g_S = \chi_S g $ and $\chi_S \colon \mathbb R \to \mathbb R$ is the characteristic function of $S$. Then, by the properties of the Borel functional calculus listed above, we get
    \begin{displaymath}
        \lVert G_S f \rVert^2_{\mathbb H} = \int_{\mathbb X} \lvert g_S \rvert^2 \, dE_f \leq \lVert g_S\rVert^2_\infty E_f(S) \leq \lVert g\rVert_\infty^2 E_f(S). \qedhere
    \end{displaymath}
\end{proof}

\section{Variable-bandwidth kernels}
\label{app:vbkernel}
Here, we briefly describe the construction of the variable-bandwidth kernels employed for the numerical experiments in \cref{sec:examples}. In the applications that we consider, the state space $\mathcal M$ is embedded in a data space $Y$ by means of a map $F\colon \mathcal M \to Y$. We assume that $Y$ has an appropriate differentiable structure such that the restriction of $F$ onto $M$ is $C^1$. In this scenario, it is natural to work with kernels $k$ obtained by pullbacks of kernels $k^{(Y)} \colon Y \times Y \to \mathbb R$ on data space, i.e., $k(x,y) = k^{(Y)}(F(x), F(y))$.  As a concrete example, we mention the Gaussian radial basis function (RBF) kernel on $Y = \mathbb R^d$,
\begin{equation}
    \label{eq:rbf}
    k^{(Y)}(x, y) = k^\text{RBF}_\epsilon(x,y) \equiv  e^{- \lVert x- y\rVert^2 / \epsilon^2},
\end{equation}
where $\lVert \cdot \rVert$ is the Euclidean 2-norm and $\epsilon>0$ a bandwidth parameter.

In the experiments of \cref{sec:examples} we use a variable-bandwidth generalization  \cite{BerryHarlim16} of the Gaussian RBF, given by
\begin{equation}
    \label{eq:k_vb}
    k^{(Y)}(x,y) = \exp \left( - \frac{\lVert x-y\rVert^2}{\epsilon^2 r(x)r(y)}\right).
\end{equation}
Here, $r \in C^1(Y)$ is a strictly positive bandwidth function designed to take large (small) values in regions of small (large) sampling density of the data relative to Lebesgue measure on data space. In this paper, we  employ the bandwidth function $r(x) = \rho_N^{-1/m}(x)$, where $m \in \mathbb R$ is a numerically computed dimension parameter for the support of $\mu$, and $\rho_N$ is a density function given by
\begin{gather*}
    \rho_N(x) = \int_M \bar p_N(x, y) \, d\mu_N(y)
\end{gather*}
for a Markov kernel $\bar p_N \colon M \times M \to \mathbb R$. Setting $\bar k(x, y) = k^\text{RBF}_{\bar\epsilon}(F(x), F(y))$ for a bandwidth parameter $\bar \epsilon$ independent from $\epsilon$, we build $\bar p_N$ by applying to $\bar k$ the normalization procedure introduced in the diffusion maps algorithm \cite{CoifmanLafon06}:
\begin{displaymath}
   \bar p(x, y)= \frac{\bar k(x, y)}{\bar d_N(x) \bar q_N(x) \bar q_N(y)}, \quad d_N(x) = \int_M \frac{\bar k(x, y)}{\bar q_N(x) \bar q_N(y)} \, d\mu_N(y), \quad \bar q_N(x) = \int_M \bar k(x,y) \, d\mu_N(y).
\end{displaymath}
Moreover, we tune the bandwidth parameters $\epsilon$ and $\bar\epsilon$ automatically using a variant of a procedure proposed by \cite{CoifmanEtAl08}. This procedure also yields the dimension parameter $m$.

With this choice of kernel $\bar p_N$, if $\mu$ is a smooth measure supported on a Riemannian manifold then $\rho_N$ approximates the density of $\mu$ relative to the Riemannian volume measure. It can be shown (e.g., \cite{Giannakis19}) that as $\bar \epsilon$ and $\epsilon$ become small, the resulting Markov-normalized variable-bandwidth kernel approximates the heat kernel associated with a conformally transformed Riemannian metric whose volume measure has uniform density relative to $\mu$. In other words, the kernel-induced Riemannian geometry ``balances out'' the sampling distribution of the training data.

While in many applications $\mu$ is not smooth, empirically we find that the variable-bandwidth procedure improves the regularity of numerically computed eigenfunctions. For example, see \cref{fig:phi} and \cref{fig:phi_nb} for a comparison between kernel eigenfunctions computed on the Lorenz attractor with and without the use of variable-bandwidth kernels. The addition of the variable-bandwidth step leads to considerably smoother eigenfunctions $\phi_{j,N}$ and their directional derivatives under the dynamical vector field, particularly in regions near the boundary of the attractor where sampling is sparse.

For further details on the construction of the kernel~\eqref{eq:k_vb} and the bandwidth tuning algorithm as used in this paper we refer the reader to Appendices~A of \cite{DasEtAl21,Giannakis19}, as well as the code repository accompanying this paper.

\begin{figure}
    \includegraphics[width=\linewidth,draft=false]{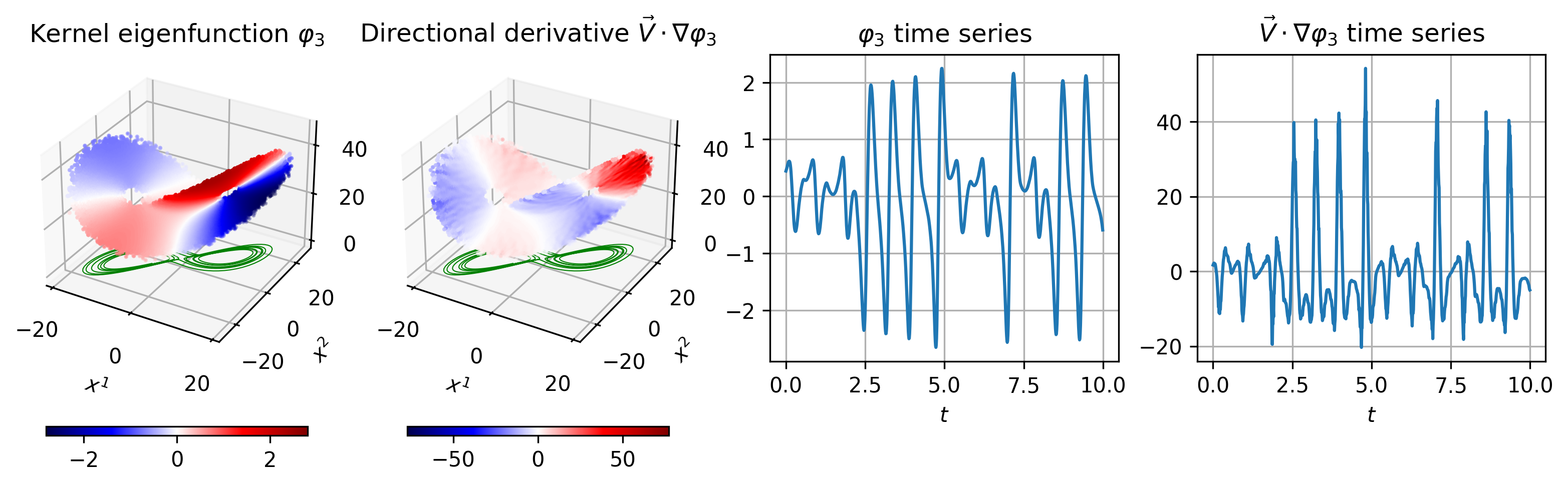}
    \includegraphics[width=\linewidth,draft=false]{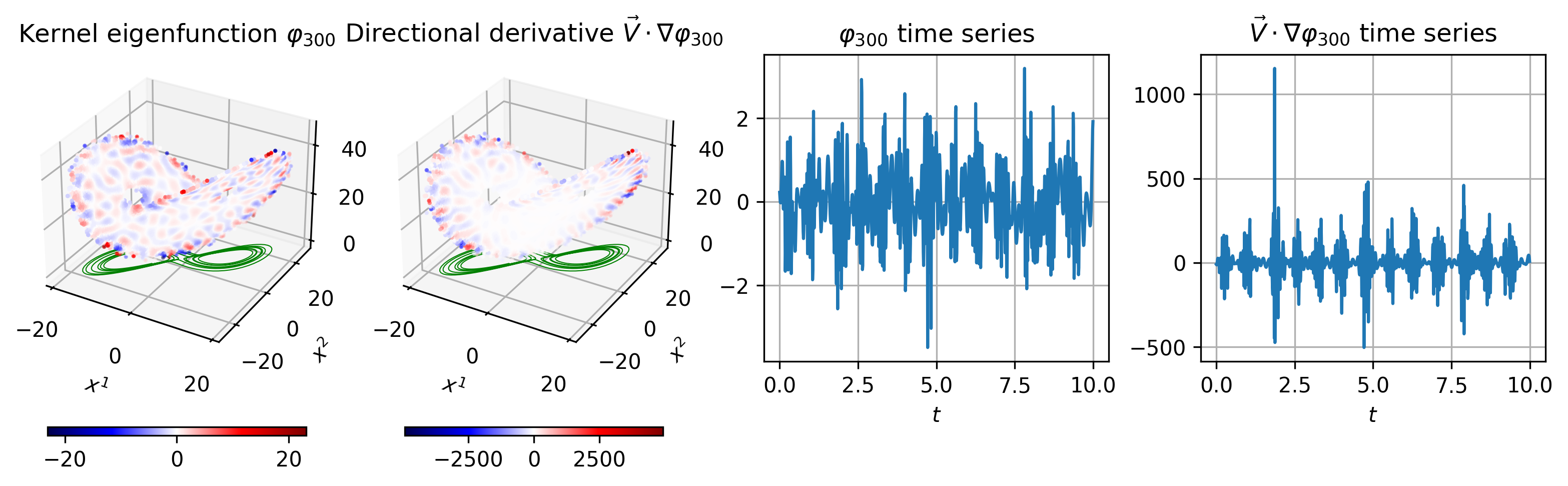}
    \caption{As in \cref{fig:phi}, but for eigenfunctions $\phi_{j,N}$ and directional derivatives $\varphi_{j,N}'$ computed using bistochastic Markov normalization of the fixed-bandwidth kernel~\eqref{eq:rbf}. While leading eigenfunctions are qualitatively similar to their variable-bandwidth counterparts (e.g., compare $\phi_{3,N}$ in the top row with the corresponding eigenfunction in \cref{fig:phi}), eigenfunctions further down the spectrum become significantly less regular. For example, $\phi_{300,N}$ shown in the bottom row is dominated by ``spikes'' near the boundary of the attractor where sampling is relatively sparser. The lack of regularity affects the directional derivatives $\varphi'_{j,N}$, as manifested by the spiking behavior of the $\varphi'_{300,N}$ time series in this example (bottom right panel).}
    \label{fig:phi_nb}
\end{figure}

% \bibliographystyle{elsarticle-num}
% \biboptions{sort&compress}
% \bibliography{referencesDG.bib,referencesCV.bib}

\begin{thebibliography}{10}
\expandafter\ifx\csname url\endcsname\relax
  \def\url#1{\texttt{#1}}\fi
\expandafter\ifx\csname urlprefix\endcsname\relax\def\urlprefix{URL }\fi
\expandafter\ifx\csname href\endcsname\relax
  \def\href#1#2{#2} \def\path#1{#1}\fi

\bibitem{Baladi00}
V.~Baladi, Positive Transfer Operators and Decay of Correlations, Vol.~16 of Advanced Series in Nonlinear Dynamics, World Scientific, Singapore, 2000.

\bibitem{EisnerEtAl15}
T.~Eisner, B.~Farkas, M.~Haase, R.~Nagel, Operator Theoretic Aspects of Ergodic Theory, Vol. 272 of Graduate Texts in Mathematics, Springer, Cham, 2015.

\bibitem{Froyland99}
G.~Froyland, Ulam's method for random interval maps, Nonlinearity 12~(4) (1999) 1029--1052.
\newblock \href {https://doi.org/10.1088/0951-7715/12/4/318} {\path{doi:10.1088/0951-7715/12/4/318}}.

\bibitem{DellnitzJunge99}
M.~Dellnitz, O.~Junge, On the approximation of complicated dynamical behavior, SIAM J. Numer. Anal. 36 (1999) 491.
\newblock \href {https://doi.org/10.1137/S0036142996313002} {\path{doi:10.1137/S0036142996313002}}.

\bibitem{Mezic05}
I.~Mezi\'c, Spectral properties of dynamical systems, model reduction and decompositions, Nonlinear Dyn. 41 (2005) 309--325.
\newblock \href {https://doi.org/10.1007/s11071-005-2824-x} {\path{doi:10.1007/s11071-005-2824-x}}.

\bibitem{KlusEtAl18}
S.~Klus, F.~N\"uske, P.~Koltai, H.~Wu, I.~Kevrekidis, C.~Sch{\"u}tte, F.~No{\'e}, Data-driven model reduction and transfer operator approximation, J. Nonlinear Sci. 28 (2018) 985--1010.
\newblock \href {https://doi.org/10.1007/s00332-017-9437-7} {\path{doi:10.1007/s00332-017-9437-7}}.

\bibitem{BruntonEtAl22}
S.~L. Brunton, M.~Budisi{\'c}, E.~Kaiser, J.~N. Kutz, Modern {K}oopman theory for dynamical systems, SIAM Rev. 64~(2) (2022) 229--340.
\newblock \href {https://doi.org/10.1137/21M1401243} {\path{doi:10.1137/21M1401243}}.

\bibitem{OttoRowley21}
S.~E. Otto, C.~W. Rowley, Koopman operators for estimation and control of dynamical systems, Annu. Rev. Control Robot. Auton. Syst. 4 (2021) 59--87.
\newblock \href {https://doi.org/10.1146/annurev-control-071020-010108} {\path{doi:10.1146/annurev-control-071020-010108}}.

\bibitem{Colbrook24}
M.~Colbrook, The multiverse of dynamic mode decomposition algorithms, in: Handbook of Numerical Analysis, Amsterdam, 2024, p.~88.

\bibitem{Mezic13}
I.~Mezi\'c, Analysis of fluid flows via spectral properties of the {K}oopman operator, Annu. Rev. Fluid Mech. 45 (2013) 357–--378.
\newblock \href {https://doi.org/10.1146/annurev-fluid-011212-140652} {\path{doi:10.1146/annurev-fluid-011212-140652}}.

\bibitem{FroylandEtAl21}
G.~Froyland, D.~Giannakis, B.~Lintner, M.~Pike, J.~Slawinska, Spectral analysis of climate dynamics with operator-theoretic approaches, Nat. Commun. 12 (2021).
\newblock \href {https://doi.org/10.1038/s41467-021-26357-x} {\path{doi:10.1038/s41467-021-26357-x}}.

\bibitem{WuEtAl17}
H.~Wu, F.~N{\"u}ske, F.~Paul, S.~Klus, P.~Koltai, F.~No{\'e}, Variational {K}oopman models: {S}low collective variables and molecular kinetics from short off-equilibrium simulations, J. Chem. Phys. 146 (2017).
\newblock \href {https://doi.org/10.1063/1.4979344} {\path{doi:10.1063/1.4979344}}.

\bibitem{SusukiEtAl16}
Y.~Susuki, I.~Mezi{\'c}, F.~Raak, T.~Hikihara, Applied {K}oopman operator theory for power systems technology, NOLTA 7~(4) (2016) 430--459.
\newblock \href {https://doi.org/10.1587/nolta.7.430} {\path{doi:10.1587/nolta.7.430}}.

\bibitem{MauroyEtAl20}
A.~Mauroy, I.~Mezi{\'c}, Y.~Susuki (Eds.), The {K}oopman Operator in Systems and Control, no. 484 in Lecture Notes in Control and Information Sciences, Springer, 2020.
\newblock \href {https://doi.org/10.1007/978-3-030-35713-9} {\path{doi:10.1007/978-3-030-35713-9}}.

\bibitem{Walters81}
P.~Walters, An Introduction to Ergodic Theory, Vol.~79 of Graduate Texts in Mathematics, Springer-Verlag, New York, 1981.

\bibitem{Tucker99}
W.~Tucker, The {L}orenz attractor exists, C. R. Acad. Sci. Paris, Ser. I 328 (1999) 1197--1202.

\bibitem{LuzzattoEtAl05}
S.~Luzzatto, I.~Melbourne, F.~Paccaut, The {L}orenz attractor is mixing, Comm. Math. Phys. 260~(2) (2005) 393--401.

\bibitem{RaissiEtAl19}
M.~Raissi, P.~Perdikaris, G.~E. Karniadakis, Physics-informed neural networks: {A} deep learning framework for solving forward and inverse problems involving nonlinear partial differential equations, J. Comput. Phys. 378 (2019) 686--707.
\newblock \href {https://doi.org/10.1016/j.jcp.2018.10.045} {\path{doi:10.1016/j.jcp.2018.10.045}}.

\bibitem{BaddooEtAl23}
P.~J. Baddoo, B.~Herrmann, B.~J. McKeon, J.~N. Kutz, S.~L. Brunton, Physics-informed dynamic mode decomposition, Proc. R. Soc. A 479 (2023).
\newblock \href {https://doi.org/10.1098/rspa.2022.0576} {\path{doi:10.1098/rspa.2022.0576}}.

\bibitem{DasEtAl21}
S.~Das, D.~Giannakis, J.~Slawinska, Reproducing kernel {H}ilbert space compactification of unitary evolution groups, Appl. Comput. Harmon. Anal. 54 (2021) 75--136.
\newblock \href {https://doi.org/10.1016/j.acha.2021.02.004} {\path{doi:10.1016/j.acha.2021.02.004}}.

\bibitem{GiannakisValva24}
D.~Giannakis, C.~Valva, Consistent spectral approximation of {K}oopman operators using resolvent compactification, Nonlinearity 37~(7) (2024).
\newblock \href {https://doi.org/10.1088/1361-6544/ad4ade} {\path{doi:10.1088/1361-6544/ad4ade}}.

\bibitem{Koopman31}
B.~O. Koopman, Hamiltonian systems and transformation in {H}ilbert space, Proc. Natl. Acad. Sci. 17~(5) (1931) 315--318.
\newblock \href {https://doi.org/10.1073/pnas.17.5.315} {\path{doi:10.1073/pnas.17.5.315}}.

\bibitem{KoopmanVonNeumann32}
B.~O. Koopman, J.~von Neumann, Dynamical systems of continuous spectra, Proc. Natl. Acad. Sci. 18~(3) (1932) 255--263.
\newblock \href {https://doi.org/10.1073/pnas.18.3.255} {\path{doi:10.1073/pnas.18.3.255}}.

\bibitem{Stone32}
M.~H. Stone, On one-parameter unitary groups in {H}ilbert space, Ann. Math 33~(3) (1932) 643--648.
\newblock \href {https://doi.org/doi.org/10.2307/1968538} {\path{doi:doi.org/10.2307/1968538}}.

\bibitem{Schmudgen12}
K.~Schm{\"u}dgen, Unbounded Self-Adjoint Operators on {H}ilbert Space, Vol. 265 of Graduate Texts in Mathematics, Springer Science+Business Media, 2012.

\bibitem{Oliveira09}
C.~R. de~Oliveira, Intermediate Spectral Theory and Quantum Dynamics, Vol.~54 of Progress in Mathematical Physics, Birkh{\"a}user, Basel, 2009.

\bibitem{Chatelin11}
F.~Chatelin, Spectral Approximation of Linear Operators, Classics in Applied Mathematics, Society for Industrial and Applied Mathematics, Philadelphia, 2011.

\bibitem{SriperumbudurEtAl11}
B.~K. Sriperumbudur, K.~Fukumizu, G.~R. Lanckriet, Universality, characteristic kernels and {RKHS} embedding of measures, J. Mach. Learn. Res. 12 (2011) 2389--2410.

\bibitem{CoifmanHirn13}
R.~Coifman, M.~Hirn, Bi-stochastic kernels via asymmetric affinity functions, Appl. Comput. Harmon. Anal. 35~(1) (2013) 177--180.
\newblock \href {https://doi.org/10.1016/j.acha.2013.01.001} {\path{doi:10.1016/j.acha.2013.01.001}}.

\bibitem{PaulsenRaghupathi16}
V.~I. Paulsen, M.~Raghupathi, An Introduction to the Theory of Reproducing Kernel {H}ilbert Spaces, Vol. 152 of Cambridge Studies in Advanced Mathematics, Cambridge University Press, Cambridge, 2016.

\bibitem{SteinwartChristmann08}
I.~Steinwart, A.~Christmann, Support Vector Machines, Information Science and Statistics, Springer, New York, 2008.

\bibitem{BabuskaOsborn91}
I.~Babu\v{s}ka, J.~Osborn, Eigenvalue problems, in: P.~G. Ciarlet, J.~L. Lions (Eds.), Finite Element Methods (Part 1), Vol.~II of Handbook of Numerical Analysis, North-Holland, Amsterdam, 1991, pp. 641--787.

\bibitem{ReedSimon80}
M.~Reed, B.~Simon, Methods of Modern Mathematical Physics I: Functional Analysis, Academic Press, New York, 1980.

\bibitem{Blank17}
M.~Blank, Egodic averaging with and without invariant measures, Nonlinearity 30 (2017) 4649--4664.
\newblock \href {https://doi.org/10.1088/1361-6544/aa8fe8} {\path{doi:10.1088/1361-6544/aa8fe8}}.

\bibitem{VonLuxburgEtAl08}
U.~von Luxburg, M.~Belkin, O.~Bousquet, Consitency of spectral clustering, Ann. Stat. 26~(2) (2008) 555--586.
\newblock \href {https://doi.org/10.1214/009053607000000640} {\path{doi:10.1214/009053607000000640}}.

\bibitem{Oxtoby53}
J.~C. Oxtoby, Stepanoff flows on the torus, Proc. Amer. Math. Soc. 4 (1953) 982--987.

\bibitem{Lorenz63}
E.~N. Lorenz, Deterministic nonperiodic flow, J. Atmos. Sci. 20 (1963) 130--141.
\newblock \href {https://doi.org/10.1175/1520-0469(1963)020<0130:DNF>2.0.CO;2} {\path{doi:10.1175/1520-0469(1963)020<0130:DNF>2.0.CO;2}}.

\bibitem{KatokThouvenot06}
A.~Katok, J.-P. Thouvenot, Spectral properties and combinatorial constructions in ergodic theory, in: B.~Hasselblatt, A.~Katok (Eds.), Handbook of Dynamical Systems, Vol.~1B, North-Holland, Amsterdam, 2006, Ch.~11, pp. 649--743.

\bibitem{Kocergin73}
A.~V. Ko\v{c}ergin, Time changes in flows and mixing, Math. USSR Izv. 7~(6) (1973) 1273--1294.
\newblock \href {https://doi.org/10.1070/IM1973v007n06ABEH002087} {\path{doi:10.1070/IM1973v007n06ABEH002087}}.

\bibitem{Fayad02}
B.~Fayad, Analytic mixing reparametrizations of irrational flows, Ergod. Th. \& Dynam. Sys. 22 (2002) 437--468.
\newblock \href {https://doi.org/10.1017/s0143385702000214} {\path{doi:10.1017/s0143385702000214}}.

\bibitem{LawEtAl14}
K.~Law, A.~Shukla, A.~M. Stuart, Analysis of the {3DVAR} filter for the partially observed {L}orenz'63 model, Discrete Contin. Dyn. Syst. 34~(3) (2013) 1061--10178.
\newblock \href {https://doi.org/10.3934/dcds.2014.34.1061} {\path{doi:10.3934/dcds.2014.34.1061}}.

\bibitem{Sprott03}
J.~C. Sprott, Chaos and Time-Series Analysis, Oxford University Press, Oxford, 2003.

\bibitem{KordaEtAl20}
M.~Korda, M.~Putinar, I.~Mezi\'c, Data-driven spectral analysis of the {K}oopman operator, Appl. Comput. Harmon. Anal. 48~(2) (2020) 599--629.
\newblock \href {https://doi.org/10.1016/j.acha.2018.08.002} {\path{doi:10.1016/j.acha.2018.08.002}}.

\bibitem{ColbrookEtAl24}
M.~J. Colbrook, C.~Drysdale, A.~Horning, \href{https://arxiv.org/abs/2405.00782}{Rigged {D}ynamic {M}ode {D}ecomposition: {D}ata-driven generalized eigenfunction decompositions for {K}oopman operators} (2024).
\newline\urlprefix\url{https://arxiv.org/abs/2405.00782}

\bibitem{colbrook_limits_2024}
M.~J. Colbrook, I.~Mezić, A.~Stepanenko, \href{http://arxiv.org/abs/2407.06312}{Limits and {Powers} of {Koopman} {Learning}}, arXiv:2407.06312 [math] (Jul. 2024).
\newblock \href {https://doi.org/10.48550/arXiv.2407.06312} {\path{doi:10.48550/arXiv.2407.06312}}.
\newline\urlprefix\url{http://arxiv.org/abs/2407.06312}

\bibitem{Giannakis19}
D.~Giannakis, Data-driven spectral decomposition and forecasting of ergodic dynamical systems, Appl. Comput. Harmon. Anal. 47~(2) (2019) 338--396.
\newblock \href {https://doi.org/10.1016/j.acha.2017.09.001} {\path{doi:10.1016/j.acha.2017.09.001}}.

\bibitem{Folland16}
G.~B. Folland, A {Course} in {Abstract} {Harmonic} {Analysis}, Chapman and Hall/CRC, 2016.
\newblock \href {https://doi.org/10.1201/b19172} {\path{doi:10.1201/b19172}}.

\bibitem{BerryHarlim16}
T.~Berry, J.~Harlim, Variable bandwidth diffusion kernels, Appl. Comput. Harmon. Anal. 40~(1) (2016) 68--96.
\newblock \href {https://doi.org/10.1016/j.acha.2015.01.001} {\path{doi:10.1016/j.acha.2015.01.001}}.

\bibitem{CoifmanLafon06}
R.~R. Coifman, S.~Lafon, Diffusion maps, Appl. Comput. Harmon. Anal. 21 (2006) 5--30.
\newblock \href {https://doi.org/10.1016/j.acha.2006.04.006} {\path{doi:10.1016/j.acha.2006.04.006}}.

\bibitem{CoifmanEtAl08}
R.~R. Coifman, Y.~Shkolnisky, F.~J. Sigworth, A.~Singer, Graph {L}aplacian tomography from unknown random projections, IEEE Trans. Image Process. 17~(10) (2008) 1891--1899.
\newblock \href {https://doi.org/10.1109/tip.2008.2002305} {\path{doi:10.1109/tip.2008.2002305}}.

\end{thebibliography}

\end{document}